\documentclass[a4paper,11pt]{amsart}
\pdfoutput=1
\usepackage[a4paper,margin=2.4cm]{geometry}
\usepackage{amssymb}
\usepackage{amsmath}
\usepackage{mathrsfs}
\usepackage{tikz-cd}

\newtheorem{iTheorem}{Theorem}
\newtheorem{iProp}{Proposition}
\newtheorem{iCor}{Corollary}
\newtheorem{Theorem}{Theorem}[section]

\newtheorem{lemma}[Theorem]{Lemma}
\newtheorem{prop}[Theorem]{Proposition}
\newtheorem{cor}[Theorem]{Corollary}
\theoremstyle{definition}
\newtheorem{defn}[Theorem]{Definition}

\theoremstyle{remark}
\newtheorem{rem}[Theorem]{Remark}

\newtheorem{ex}[Theorem]{Example}

\newcounter{numl}
\newcommand{\labelnuml}{\textup{(\roman{numl})}}

\DeclareSymbolFont{script}{U}{eus}{m}{n}
\DeclareSymbolFontAlphabet{\amathscr}{script}
\DeclareMathSymbol{\Wedge}{0}{script}{"5E}
\DeclareMathAlphabet{\mathrmsl}{OT1}{cmr}{m}{sl}

\renewcommand{\geq}{\geqslant}
\renewcommand{\leq}{\leqslant}

\newcommand{\R}{{\mathbb R}}
\newcommand{\C}{{\mathbb C}}
\newcommand{\N}{{\mathbb N}}

\newcommand{\T}{{\mathbb T}}
\newcommand{\G}{{\mathbb G}}
\newcommand{\V}{\mathrm{Vol}}
\newcommand{\PP}{{\mathbb P}}
\newcommand{\Pol}{{\Delta}}
\newcommand{\E}{{\mathcal E}}

\newcommand{\tor}{{\mathfrak t}}

\newcommand{\cK}{{\mathcal K}}
\newcommand{\cE}{{\bf M}}
\newcommand{\ce}{{\bf I}}
\newcommand{\cI}{{\bf I}}
\newcommand{\cJ}{{\bf J}}
\newcommand{\cF}{{\bf F}}
\newcommand{\m}{{\mu}} 
\newcommand{\x}{{\mu}} 
\newcommand{\p}{{p}} 
\newcommand{\q}{{q}} 
\newcommand{\K}{{\mathbb T}}
\newcommand{\cH}{{\mathcal K}}
\newcommand{\Ent}{\mathrm{Ent}}
\newcommand{\dtor}{r}
\newcommand{\bom}{{\omega_0}}

\newcommand{\bN}{P}
\newcommand{\cL}{{\mathcal L}}
\newcommand{\cO}{{\mathcal O}}
\newcommand{\cM}{{\mathcal M}}
\newcommand{\Hol}{{\mathscr H}}
\newcommand{\tstL}{{\mathscr L}}
\newcommand{\tstA}{{\mathscr A}}

\newcommand{\tstX}{{\mathscr X}}

\newcommand{\Fut}{{\mathscr F}}
\newcommand{\vol}{\mathrm{Vol}}
\newcommand{\tr}{{\rm tr}}
\newcommand{\sign}{\mathrmsl{sign}}

\newcommand{\Aut}{\mathrmsl{Aut}}
\newcommand{\Aff}{\mathrmsl{Aff}}

\renewcommand{\d}{{\mathrmsl d}}

\newcommand{\Hess}{\mathrmsl{Hess}}
\newcommand{\Scal}{\mathit{Scal}}
\newcommand{\Sm}{N}
\newcommand{\Sph}{{\mathbb S}}
\newcommand{\Ds}{{\amathscr D}}
\newcommand{\Hor}{{\amathscr H}}
\newcommand{\Hs}{{\mathscr H}}
\newcommand{\Z}{{\mathbb Z}}
\newcommand{\J}{{\mathbf J}}
\def\th/#1#2{{#1}^{#2}}
\newcommand{\PSH}{{\mathrm{PSH}}}
\newcommand{\MA}{{\mathrm{MA}}}
\newcommand{\ext}{{\mathrm{ext}}}

\newcommand{\HL}{{\rm HL}}
\newcommand{\Ha}{{\mathbb H}} 
\newcommand{\dF}{{\bf V}}
\newcommand{\dW}{{\bf W}}
\newcommand{\Ko}{{\mathbb K}}
\newcommand{\ko}{{\mathfrak k}}
\newcommand{\bL}{{\mathbb L}}

\begin{document}

\title[Weighted K-stability]{Weighted K-stability and coercivity with applications to extremal K\"ahler and Sasaki metrics }
\author[V. Apostolov]{Vestislav Apostolov} \address{Vestislav Apostolov \\ D{\'e}partement de Math{\'e}matiques\\ UQAM\\ C.P. 8888 \\ Succursale Centre-ville \\ Montr{\'e}al (Qu{\'e}bec) \\ H3C 3P8 \\ Canada \\ { and Institute of Mathematics and Informatics, Bulgarian Academy of Sciences}} 
\email{apostolov.vestislav@uqam.ca}

\author[S. Jubert]{Simon Jubert} \address{Simon Jubert \\ D{\'e}partement de Math{\'e}matiques\\ UQAM\\ C.P. 8888 \\ Succursale Centre-ville \\ Montr{\'e}al (Qu{\'e}bec) \\ H3C 3P8 \\ Canada \\ and  {Institut de Math\'ematiques de Toulouse \\ Universit\'e Paul Sabatier \\ 118 route de Narbonne\\ 31062 Toulouse\\ France}}
\email{simonjubert@gmail.com}

\author[A. Lahdili]{Abdellah Lahdili} \address{Abdellah Lahdili \\  {Department of Mathematics \\ University of Aarhus\\  Ny Munkegade 118, buil. 1530, 323 \\
8000 Aarhus \\ Denmark}}
\email{lahdili.abdellah@gmail.com}

\thanks{V.A. was supported in part by an NSERC Discovery Grant and a Connect Talent Grant of the R\'egion Pays de la Loire. S.J. was supported by PhD fellowships of the UQAM and the University of Toulouse  III - Paul
Sabatier. A.L. was supported by a postdoctoral fellowship of Aarhus University. We are very grateful to the anonymous referee for their careful reading of the manuscript and many valuable suggestions which substantially improved our work. We thank D. Calderbank,  T. Darvas, R. Dervan, E. Inoue, E. Legendre, C. Li  and G. Tian for their interest and comments on the manuscript. V.A. thanks P. Gauduchon for sharing with him  his unpublished notes~\cite{gauduchon-privite}, and C. T{\o}nnesen-Friedman for bringing the reference \cite{GMSW}  to our attention.}

\begin{abstract} We show that a compact weighted extremal K\"ahler manifold  (as defined by the third named author in \cite{lahdili1})  has coercive weighted Mabuchi energy with respect to a maximal complex torus $\T^{\C}$  in the reduced group of  complex automorphisms. This provides a vast extension and a unification of a number of results concerning  K\"ahler metrics satisfying special curvature conditions, including K\"ahler metrics with constant scalar curvature~\cite{BDL,CC}, extremal K\"ahler metrics~\cite{He-ext}, K\"ahler--Ricci solitons~\cite{DaRu} and their weighted extensions~\cite{B-N,HL}.
Our result implies the strict positivity of the weighted Donaldson--Futaki invariant of any non-product $\T^{\C}$-equivariant smooth K\"ahler test configuration with reduced central fibre,  a property known as $\T^{\C}$-equivariant weighted K-polystability on such test configurations. It also yields the $\T^{\C}$-uniform weighted K-stability on the class of smooth $\T^{\C}$-equivariant polarized test configurations  with reduced central fibre.  For a class of fibrations constructed from principal torus bundles over a product of Hodge cscK manifolds, we use our results in conjunction with results of Chen--Cheng~\cite{CC}, He~\cite{He-ext} and Han--Li~\cite{HL}  in order to characterize the existence of extremal K\"ahler metrics and Calabi--Yau cones associated to the total space,  in terms of the coercivity of the weighted Mabuchi energy of the fibre. This   yields a new existence result for Sasaki--Einstein metrics on certain Fano toric fibrations, extending  the results of Futaki--Ono--Wang~\cite{FOW} in the toric Fano case,  and of Mabuchi--Nakagawa~\cite{MN} in the case of Fano $\PP^1$-bundles. \end{abstract}


\maketitle


\section*{Introduction} This paper is concerned with the existence and obstruction theory of a class of  special K\"ahler metrics,   called \emph{weighted constant scalar curvature metrics},  which were introduced by the third named author in \cite{lahdili2}, giving a vast extension of the notion of K\"ahler metrics of constant scalar curvature (cscK for short), and providing a unification for a number of  related notions of K\"ahler metrics satisfying special curvature conditions.

\subsection{The weighted cscK problem}
Let $X$ be a smooth compact complex $m$-dimensional manifold with a given deRham cohomology class $\alpha\in H^{1,1}(X, \R)$ of K\"ahler metrics, and let $\T\subset \Aut_r(X)$ denote a fixed compact torus in the \emph{reduced} group $\Aut_r(X)$ of automorphisms of $X$, i.e. the connected subgroup of automorphisms of $X$ generated by the Lie algebra of real holomorphic vector fields with zeros, see e.g.~\cite{gauduchon-book}.  It is well-known that $\T$ acts in a hamiltonian way with respect to any $\T$-invariant K\"ahler metric $\omega\in \alpha$,  and the corresponding momentum map $\mu_{\omega}$ sends $X$ onto a compact convex polytope $\Pol \subset \tor^*$ in the dual vector space $\tor^*$ of the Lie algebra $\tor$ of $\T$ (cf.~\cite{Atiyah, GS}). Furthermore, up to translations, $\Pol$ is independent of the choice of $\omega\in \alpha$.  We shall further fix $\Pol$, giving rise to a \emph{normalization} of the corresponding momentum maps $\{\mu_{\omega}, \omega \in \alpha\}$.

Following \cite{lahdili2}, let $v(\x)>0$ and $w(\x)$ be smooth functions defined on $\Pol$. One can then consider the following condition for $\T$-invariant K\"ahler  metrics $\omega$  in  $\alpha$ (and fixed polytope $\Pol$), called \emph{$(v,w)$-cscK metric}:
\begin{equation}\label{weighted-cscK-0}
\Scal_{v}(\omega) = w(\m_{\omega}), 
\end{equation}
where the so-called \emph{$v$-scalar curvature} of $\omega$ is introduced by 
\begin{equation}\label{v-scal}
\Scal_v(\omega):= v(\m_{\omega}) \Scal(\omega) + 2\Delta_{\omega} v(\m_{\omega})  + \big\langle g_{\omega}, \m_{\omega}^*\left(\Hess(v)\right)\big\rangle, \end{equation}
with  $\Scal(\omega)$ being the usual scalar curvature of the riemannian metric $g_{\omega}$  associated to $\omega$,  $\Delta_{\omega}$ the Laplace operator of $g_{\omega}$,  and  where the contraction $\langle \cdot, \cdot \rangle$  is taken between the smooth $\tor^*\otimes \tor^*$-valued function $g_{\omega}$ on $X$  (the restriction of the riemannian metric $g_{\omega}$  to $\tor \subset C^{\infty}(X, TX)$)  and the smooth $\tor\otimes \tor$-valued function ${\m_{\omega}}^*\left(\Hess(v)\right)$ on $X$ (given by the pull-back by $\m_{\omega}$ of $\Hess(v) \in C^{\infty}(\Pol, \tor\otimes \tor)$). The relevance of \eqref{weighted-cscK-0} to various geometric conditions is discussed in detail in \cite{lahdili2}, but we mention below a few special cases which partly motivate the study in this article.
\begin{itemize}
\item $v=1$ and $w$ is a constant: this is the familiar cscK problem;
\item $v=1$ and $w=\ell$ with $\ell$ an affine-linear function on $\tor^*$:  \eqref{weighted-cscK-0} then describes an extremal K\"ahler metric in the sense of Calabi~\cite{calabi};
\item $v= e^{\ell}, w= 2(\ell + a) e^{\ell}$ where $\ell$ is an affine-linear function on $\tor^*$ and $a$ is a constant correspond to  the so-called $\mu$-cscK~\cite{Inoue}, extending the notion of K\"ahler--Ricci solitons~\cite{TZ} defined when $X$ is Fano and $\alpha= 2\pi c_1(X)$.
\item $v=\ell^{-m-1}, w=a \ell^{-m-2}$,  $\alpha=c_1(L)$, where $\ell$ is a positive affine-linear function on $\Pol$,  $m$ is the complex dimension of $X$,  $a$ is a constant,  and $L$ is  a polarization of $X$:  \eqref{weighted-cscK-0} then describes  a scalar flat cone K\"ahler metric on the affine cone $(L^{-1})^{\times}$ polarized by the lift of $\xi = d\ell$ to $L^{-1}$ via $\ell$, see \cite{AC,ACL}.
\end{itemize}
In general, the problem of finding a $\T$-invariant K\"ahler metric $\omega \in \alpha$ solving \eqref{weighted-cscK-0}  is obstructed in a similar way that the cscK problem is obstructed by the vanishing of the Futaki invariant: for any $\T$-invariant K\"ahler metric $\omega\in \alpha$ and any affine-linear function $\ell$ on $\tor^*$, one must have
\begin{equation}\label{Futaki}
{\rm Fut}_{v, w}(\ell):=\int_X \Big(\Scal_v(\omega) - w(\m_{\omega})\Big) \ell(\m_{\omega})\omega^{m} =0,\end{equation}
should a solution to \eqref{weighted-cscK-0} exists.  In \cite{lahdili2}, an unobstructed modification of \eqref{weighted-cscK-0} is proposed, extending Calabi's notion~\cite{calabi} of extremal K\"ahler metrics.  To this end, suppose that $v, w_0>0$ are given positive smooth functions on $\Pol$. One can then find a unique affine-linear function $\ell^{\ext}_{v,w_0}(\x)$ on $\tor^*$, called \emph{the extremal function}, such that  \eqref{Futaki} holds for the  weights $(v, w)=(v, \ell^{\ext}_{v,w_0}w_0)$.  In this case,  a  solution of the $(v, w)$-cscK problem \eqref{weighted-cscK-0} is referred to as \emph{$(v,w_0)$-extremal K\"ahler metric}. We emphasize that $(v, w_0)$-extremal K\"ahler metrics are  $(v, w)$-cscK metrics with a special property of the weight function  $w$, namely, $w = \ell w_0$ with $w_0>0$ on $\Pol$ and $\ell$ affine-linear. In particular, $(v, w)$-cscK metrics with $w\neq 0$ on $\Pol$ are $(v, w)$-extremal with $\ell^{\ext}_{v, w} = \sign(w_{|_{\Pol}})$ and $(v, 0)$-cscK metrics are $(v, w)$-extremal  with $\ell^{\ext}_{v, w} =0$ for any $w>0$. It follows that all the above listed special cases are examples of $(v, w)$-extremal K\"ahler metrics,  and thus the setup of $(v,w)$-extremal K\"ahler metrics allows one to study these cases altogether. 

\subsection{Relation to $v$-solitons} Motivated by  works  of T. Mabuchi~\cite{M1,M2} and consequent work by Berman--Witt-Nystrom \cite{B-N}, Y. Han and C. Li~\cite{HL} have recently introduced and studied the general notion of a \emph{weighted $v$-soliton} on a smooth Fano variety $X$, as follows. In the setup explained above, we let $\alpha=2\pi c_1(X)$ and consider the natural action of $\T$ on $K^{-1}_X$,  which fixes the momentum polytope $\Pol$ of $(X, \alpha, \T)$ and normalizes the momentum map $\m_{\omega}$ for any $\T$-invariant K\"ahler metric $\omega \in \alpha$. For a (smooth) positive weight function $v(\x)$ on $\Pol$, one defines a $v$-soliton as a $\T$-invariant K\"ahler metric $\omega \in \alpha$, such that
\begin{equation}\label{v-KRS}
\rho_{\omega} - \omega = \frac{1}{2} dd^c \log v(\m_{\omega}),
\end{equation}
where $\rho_{\omega}$ denotes the Ricci form of $\omega$. Notice that when $v(\x)= e^{\langle \x, \xi \rangle}$ for some $\xi\in \tor$, one gets the well-studied class of \emph{K\"ahler--Ricci solitons}~\cite{TZ} whereas the case when $v(\x)$ is a positive affine-linear function on $\Pol$ corresponds to the \emph{Mabuchi solitons} studied in \cite{M1,M2}. As we shall see below other choices for $v$ are also geometrically meaningful. We make the following useful observation.
\begin{iProp}\label{tilde v} Let $X$ be a smooth Fano manifold and $\T\subset \Aut(X)$ a compact torus. A $\T$-invariant K\"ahler metric $\omega \in 2\pi c_1(X)$ is a $v$-soliton if and only if $\omega$ is $(v, w)$-cscK with 
$w(\x):= 2(m + \langle d\log v, \x\rangle)v(\x).$
\end{iProp} 
We use the above result in order to make connection with the recent paper \cite{HL} (where the authors obtain a  complete Yau-Tian-Donaldson type correspondence for the existence of $v$-Ricci solitons) which will play an important role in our present study of $(v, w)$-cscK metrics.

We also notice that $v$-solitons can be viewed as $(\bar v,\bar w)$-cscK metrics for \emph{different} choices of weights. This is for instance the case when $v(\x)=\ell(\x)^{-(m+2)}$, where $\ell(\x)=\langle \xi, \mu\rangle + a$ is a positive affine-linear on $\Pol$. Whereas Proposition~\ref{tilde v} identifies the $v$-soliton as a $(v,w)$-cscK metric with
\[ v= \ell^{-(m+2)}, \qquad  w= 2\ell^{-(m+3)}\left(-2\ell +(m+2)a\right), \]
 we also observe  that
\begin{iProp}\label{SE} Let $(X, \T)$ be a smooth Fano variety and $\ell(\mu)=(\langle \xi, \mu\rangle + a)$ a positive affine-linear function on its canonical polytope $\Pol$. A $\T$-invariant K\"ahler metric $\omega \in 2\pi c_1(X)$ is an $\ell^{-(m+2)}$-soliton if and only if 
the lift $\hat \xi$ of $\xi=d\ell$ to $K_X$ via $\ell$ is the Reeb vector field of a Sasaki--Einstein structure defined on the unit circle bundle  $\Sm \subset K_X$ with respect to the hermitian metric on $K_X$ with curvature $-\omega$.  The latter condition is also equivalent to $\omega$ be a $(\ell^{-m-1}, 2ma\ell^{-m-2})$-cscK metric.
\end{iProp}

\subsection{Main results} Similarly to the usual cscK case,  it is shown in \cite{lahdili2} that the solutions of \eqref{weighted-cscK-0} can be characterized as minimizers of a functional $\cE_{v, w}$ defined on the space of $\T$-invariant K\"ahler metrics in $\alpha$, extending the \emph{Mabuchi energy} to the weighted setting (see Section~\ref{s:weighted-setup} below for the precise definition). After  the deep works \cite{BDL,CC},  it is now well-understood that the coercivity of the Mabuchi energy is equivalent to the existence of a cscK metric in a given cohomology class.  Noting that, by  the results in \cite{lahdili2},  any $(v, w)$-extremal metric is  invariant under a \emph{maximal} compact torus in  $\Aut_{r}(X)$,  our  first main result is an extension of one direction of the correspondence in the cscK case to the weighted setting.
\begin{iTheorem}\label{main} Suppose $\T\subset {\rm Aut}_{r}(X)$ is a maximal torus in the reduced group of  automorphisms of $X$, and $\bom\in \alpha$ a $\T$-invariant $(v, w_0)$-extremal K\"ahler metric. Then  the weighted Mabuchi energy $\cE_{v,w}$ (with $w= \ell^{\rm ext}_{v, w_0}w_0$) is coercive relative to the complex torus $\T^{\C}$, in the sense of \cite{DaRu}, i.e. there exist positive real constants $\lambda, \delta$ such that for any $\T$-invariant K\"ahler metric $\omega \in \alpha$,
\[\cE_{v,w}(\omega) \ge \lambda \inf_{\sigma \in \T^{\C}}\cJ(\sigma^*\omega) - \delta, \]
where $\cJ$ denotes the Aubin functional on the space K\"ahler metrics, see Definition~\ref{d:I, I-Aubin, J} below.
\end{iTheorem} 
Our proof of Theorem~\ref{main} also adapts to the case when the torus $\T\subset \Aut_r(X)$ is not necessarily maximal,  but  instead of $\T^{\C}$ one takes the infimum of  $\cJ(\sigma^*\omega)$ over $\hat\G:={\Aut}_{r}^{\T}(X)$, the connected component of the identity of the centralizer of $\T$  in $\Aut_r(X)$ (which by \cite{lahdili2} is a reductive group if  $X$ admits a $(v, w_0)$-extremal $\T$-invariant K\"ahler metric, see Remark~\ref{r:hat-G} for more details). Furthermore,  we can also consider \emph{any} reductive connected subgroup group $\G= \Ko^{\C}\subset \hat\G$ with a compact form $\Ko$ containing $\T$,  and restrict $\cE_{v, w}$ to the space of $\Ko$-invariant K\"ahler metrics in $\alpha$ as in \cite{HL}.~\footnote{We are grateful to Chi Li for pointing this out to us.}

As noticed in \cite{BDL} (in the polarized case) and in \cite{dyrefelt2} (in the more general K\"ahler case), the coercivity of  the Mabuchi energy yields a sharp estimate of the sign of the Donaldson--Futaki invariant of a $\T$-equivariant test configuration.    In our weighted setting, we consider $\T$-equivariant (compactified) K\"ahler test configurations $(\tstX, \tstA)$ associated to $(X, \alpha, \T)$,  which have smooth total space. To any such test-configuration one can associate  a \emph{weighted Donaldson--Futaki} invariant by the formula (cf. \cite{lahdili2})
\[ \Fut_{v,w}(\tstX, \tstA) :=-\int_{\tstX}\left(\Scal_v(\Omega)- w(\m_{\Omega})\right)\Omega^{[m+1]} + (8\pi) \int_X v(\m_{\omega}) \omega^{[m]}, \]
where $\Omega \in \tstA, \omega \in \alpha$ are $\T$-invariant K\"ahler forms respectively  on $\tstX$ and $X$, with respective $\Pol$-normalized momentum maps $\m_{\Omega}, \m_{\omega}$,  and   $\Scal_v(\Omega)$ is the $v$-scalar curvature of $\Omega$ defined by \eqref{v-scal}. In the above formula,  for any $2$-form $\psi$ we use the convention $\psi^{[k]}:= \frac{\psi^k}{k!}$ so that  $\Fut_{v,w}(\tstX, \tstA)$ extends to the weighted setting the expression~\cite{odaka,wang} of the Donaldson--Futaki invariant of $(\tstX, \tstA)$ in terms of intersection numbers.

\begin{iCor}\label{stability} Under the hypotheses of Theorem~\ref{main}, for any $\T$-equivariant smooth K\"ahler test configuration $(\tstX, \tstA)$ of $(X, \alpha, \T)$ which  has a reduced central fibre,  we have the inequality
\[ \Fut_{v, w}(\tstX, \tstA) \ge 0,\] with equality if and only if $(\tstX, \tstA)$ is a product test configuration. Furthermore, if $\alpha= 2\pi c_1(L)$ corresponds to a polarization $L$ of $X$ and $(\tstX, \tstL, \T)$ is a $\T$-equivariant smooth polarized test configuration of $(X, L)$ as above,  we have the inequality
\[ \Fut_{v, w}(\tstX, \tstA) \ge \lambda  \, {\bf J}_{\T^{\C}}^{\rm NA}(\tstX, \tstA),\]
where $\tstA=2\pi c_1(\tstL)$, $\lambda>0$ is the constant appearing in Theorem~\ref{main},  and ${\bf J}_{\T^\C}^{\rm NA}(\tstX, \tstA)$ is the $\T^{\C}$-relative non-Archimedean ${\bf J}$-functional of the test configuration  introduced  in \cite{H, ChiLi}, see \eqref{def:J-NA-G}.
\end{iCor}
Corollary~\ref{stability} improves the ($\T$-equivariant) $(v,w)$-K-semistability established in \cite[Thm.~2]{lahdili3} to  ($\T$-equivariant)
$(v, w)$-K-polystability on the test configurations as above, and, in the projective case,  further to  $\T^{\C}$-uniform $(v,w)$-K-stability in the sense of \cite{H,ChiLi}. As we already mentioned, the fist part of Corollary~\ref{stability}  was proved in \cite{stoppa,BDL,dyrefelt2} in the cscK case ($v=1$ and $w$ is a constant), and in \cite{dervan2, stoppa-sz, dyrefelt2} in the unweighted extremal case ($v=1=w_0$). We however notice that in the extremal case our proof uses directly the coercivity of the \emph{relative} Mabuchi energy (which follows from Theorem~\ref{main}) whereas  the proofs in \cite{dervan2,stoppa-sz}  and \cite{dyrefelt2}  are based respectively on the Arezzo--Pacard existence results of extremal metrics on blow-ups~\cite{APS},  and on the coercivity of the unweighted  Mabuchi energy $\cE_{1,c}$ established in \cite{BDL, dyrefelt2}.  The  $\T^{\C}$-uniform $(v,w)$-K-stability statement in the second part of Corollary~\ref{stability} is established in the cscK case in \cite{H,ChiLi}, and in the case of a $v$-soliton in \cite{HL}.  Our proof of Corollary~\ref{stability} in the general weighted case follows easily from Theorem~\ref{main} by the established techniques in the cscK case, see Section~\ref{s:stability}.

Another notable special case where our results apply is when  $\alpha = c_1(L)$ for an ample line bundle $L$ over $X$,  and $v=\ell^{-m-1}, w_0=\ell^{-m-3}$ for a positive affine-linear function on $\Pol$. It is observed in \cite{AC} that  in this case a $(v, w_0)$-extremal K\"ahler metric in $\alpha$ describes an \emph{extremal Sasaki} metric on the total space $\Sm$ of the unit circle bundle in $L^{-1}$ with respect to the hermitian metric with curvature $-\omega$,  and  Reeb vector field corresponding to the lift of $d\ell$  to $L^{-1}$ via $\ell$. In this special case, the first part of Corollary~\ref{stability} above was obtained in \cite{ACL}  for \emph{polarized} test configurations (see Theorem 1, Conjecture 5.8 and Remark 5.9 in \cite{ACL}), by using the results in~\cite{He-Li} which establish an analogue of Theorem~\ref{main} in the Sasaki case. Thus, our proofs of Theorem~\ref{main} and Corollary~\ref{stability} presented in this paper allow one to recast and further generalize \cite[Thm.~1]{ACL} entirely within the framework of weighted K\"ahler geometry of $X$.

\subsection{Method of proof} We now discuss briefly the method of proof of Theorem~\ref{main} above. It  is an application of the general coercivity principle~\cite[Thm.~3.4]{DaRu}, see Section~\ref{s:coercivity} below. The latter  is used in the cscK case in \cite{BDL},  and our approach is mainly inspired by these two references. Noting that in the weighted extremal case $\cE_{v,w}$ is  $\G$-invariant and $\G:=\T^{\C}$  is reductive, by the results of \cite{DaRu}, in order to obtain Theorem~\ref{main}, one needs to accomplish the following steps:  (1) extend $\cE_{v,w}$ to the space $\E^1(X, \bom)$ of $\bom$-relative pluri-subharmonic functions of maximal mass and finite energy; (2) show that the extension is convex and continuous along weak $d_1$-geodesics in $\E^1(X, \bom)$; (3) establish a compactness result for the extension of $\cE_{v, w}$, and (4) show the uniqueness modulo the action of $\G$ (and in particular the regularity) of the weak minimizers of $\cE_{v,w}$, under the assumption that a $(v, w_0)$-extremal metric exists. The steps (1), (2) and (3) in the unweighted cscK case are obtained in \cite{BDL0} and  follow from the Chen--Tian formula of $\cE_{1,1}$. The analogous  formula for $\cE_{v,w}$ is obtained in \cite{lahdili2}, but the presence of weights does not allow for a straightforward generalization of the arguments in \cite{BDL0}.  Similar difficulty arises in \cite{B-N},  in the framework of $v$-solitons on a Fano variety,  where the authors were able to obtain a suitable extension of the \emph{weighted Ding functional} to the space $\E^1(X, \bom)$. The latter functional has  milder dependence on the weights than the weighted Mabuchi functional we consider.  Indeed, the arguments of \cite{B-N} yield the existence of a continuous extension  to $\E^1(X, \bom)$ of one of the three terms in the Chen--Tian decomposition of $\cE_{v,w}$, which  depend on the weight $w$. Building on \cite{B-N},  Han--Li~\cite{HL} proposed a new approach to the extension problem in the case of $v$-solitons,  based on an idea going back to Donaldson \cite{Do-05} (see in particular the proof of Proposition 3 in \cite{Do-05}), which amounts to consider suitable fibre-bundles $Y$ over a cscK bases $B$ and fibre $X$,  and show that the weighted quantities on $X$ correspond to the restrictions of unweighted quantities on the total space $Y$. This is the \emph{semi-simple principal  $(X,\T)$-fibration construction} which we review in the next subsection. Going further than \cite{HL}, we express  in general the scalar curvature of  a bundle-compatible K\"ahler metric on $Y$  in terms of the  weighted scalar curvature of $X$,  and show that the usual (unweighted) Mabuchi energy on $Y$ restricts to a suitably weighted Mabuchi energy on $X$. It thus follows that at least for suitable polynomial weights $v$, the remaining terms of the Chen--Tian decomposition of $\cE_{v, w}$ can be extended to $\E^1(X, \bom)$ simply by restricting  to the fibres the corresponding (unweighted)  extension of the Mabuchi energy of $Y$. The final crucial observation for obtaining the extension for \emph{any} weights is that $\cE_{v,w}$ depends \emph{linearly} and \emph{continuously} on $(v, w)$, so that one can further use (as in \cite{HL}) the Stone--Weierstrass approximation theorem over $C^0(\Pol)$.  
With this in place,  and using the weighted analogue of the uniqueness~\cite{BB} achieved in \cite{lahdili3}, we can adapt the arguments from \cite{BDL}.

\subsection{Applications to the semi-simple principal fibration construction} We briefly review here the semi-simple principal bundle construction, which is a key tool in our proof of Theorem~\ref{main}, but also provides a framework for further geometric applications of our results, extending the setting of the generalized Calabi construction in~\cite{HFKG4}.

We denote by $\T$ a compact $r$-dimensional torus with Lie algebra $\tor$ and lattice $\Lambda \subset \tor$ of generators of $S^1$-subgroups, i.e. $\T= \tor/2\pi \Lambda$. Let $B= B_1 \times \cdots \times B_k$ be a $2n$-dimensional cscK manifold which is a product of  compact cscK Hodge K\"ahler $2n_a$-manifolds  $(B_a, \omega_{B_a})$,   $a=1, \ldots, k$. We then consider a principal $\T$-bundle
$\pi : \bN \to B$  endowed with  a connection $1$-form $\theta \in \Omega^1(\bN, \tor)$ with curvature 
\[d\theta = \sum_{a=1}^k (\pi^* \omega_{B_a}) \otimes p_a, \qquad p_a \in \Lambda.\]
For any smooth compact K\"ahler  $2m$-manifold $(X, \omega_X, \T)$,  endowed with a hamiltonian isometric action of the torus $\T$ as in the setup above, we can construct the \emph{principal $(X, \T)$-fibration}
\[ Y:= (X \times \bN)/\T \to B, \]
where the $\T$-action on the product is $\sigma(x, p)= (\sigma^{-1}x, \sigma p), \, x\in X, \, p\in \bN, \,  \sigma \in \T$. Using the chosen connection on $\bN$,  the almost complex structures on $X$ and $B$  lift to define a  CR structure on the product $X\times \bN$, and thus endow $Y$  with the structure of a  $2(m+n)$-dimensional smooth complex manifold. Furthermore, $Y$ comes equipped with an induced holomorphic  fibration $\pi : Y \to B$,  with smooth complex fibres $X$, and induced fibre-wise $\T$-action. Fixing constants $c_a\in \R$ such that for each $a=1, \ldots, k$,  the affine linear function $\langle p_a, \x \rangle + c_a$ on $\tor^*$  is strictly positive on the momentum image $\Pol$ of $X$, one can define a lifted K\"ahler metric $\omega_Y$ on $Y$ which,  pulled-back  to $X \times \bN$,  has the form
\begin{equation*}\label{Y-Kahler-0}
\begin{split}
\omega_Y  &:= \omega_X  + \sum_{a=1}^k (\langle p_a, \m_{\omega}\rangle + c_a)\pi^*\omega_{B_a} +  \langle d m_{\omega} \wedge  \theta \rangle
                     \end{split}
\end{equation*}
where $\langle \cdot, \cdot \rangle$ stands for the natural pairing of $\tor$ and $\tor^*$:  thus $\langle p_a, \m_{\omega}\rangle$ is a smooth function  and $\langle d\m_{\omega} \wedge \theta \rangle$ is a $2$-form on $X\times \bN$. As we show in Section~\ref{s:geometric} below, when $\omega_X$ varies in a given K\"ahler class of $X$, the corresponding K\"ahler metric $\omega_Y$ will vary in a fixed K\"ahler class on $Y$.
We also notice that when $(X, \omega_X, \T)$ is a smooth \emph{toric} K\"ahler manifold,  the setup above reduces to the theory of \emph{semi-simple rigid}  toric fibrations studied in \cite{HFKG1, HFKG2, HFKG4}. Inspired by the results in the latter works, we show that the scalar curvature of $\omega_Y$ can be expressed in terms of the $p$-weighted scalar curvature of $(X, \omega_X)$, where the weight function $p(\x)$ is a polynomial depending on the fixed data $(p_a, c_a, n_a)$ of the construction. With this observation in mind, we show that (similarly to the case of semi-simple rigid toric fibrations recently studied in \cite{Jubert}) the recent results~\cite{CC,He-ext}  can be used to obtain a converse of Theorem~\ref{main} in the case of a semi-simple principal fibrations.
 
\begin{iTheorem}\label{extremal-bundle} Suppose $Y$ is a semi-simple principal $(X, \T)$-fibration, with a K\"ahler metric $\omega_Y$ induced by a $\T$-invariant K\"ahler metric $\omega_X$ on $X$.  We suppose, moreover, that $\T$ is a maximal torus in the reduced group of automorphisms $\Aut_r(X)$. Then,  the following conditions are equivalent
\begin{itemize}
\item[(i)] $Y$ admits an extremal K\"ahler  metric in the K\"ahler class $[\omega_Y]$;
\item[(ii)] $X$ admits a  $\T$-invariant $(p, \tilde w)$-cscK metric in the K\"ahler class $[\omega_X]$,  with weights 
\[p(\mu)=\prod_{a=1}^k(\langle p_a, \x\rangle+ c_a)^{n_a}, \qquad \tilde w(\mu)= p(\mu)\left(-\sum_{a=1}^k \frac{\Scal(\omega_{B_a})}{\left(\langle p_a, \x \rangle + c_a\right)} + \ell^{\rm ext}(\mu)\right),\]
where $\ell^{\rm ext}$ is an affine-linear function determined by the condition \eqref{Futaki};
\item[(iii)] The weighted Mabuchi energy $\cE_{p, \tilde w}^X$ of $(X, [\omega_X], \T)$ is coercive with respect to $\T^{\C}$, where $p, \tilde w$ are the weights defined in (ii).
\end{itemize}
\end{iTheorem}
Compared to the general setting of \cite{DS}, the semi-simple principal $(X, \T)$-fibration (trivially) satisfy the condition of optimal symplectic connection. Accordingly, one can conclude by \cite{DS} that $(Y, [\omega_Y])$ admits an extremal K\"ahler metric, provided that  $(X, \omega_X)$ is cscK,  and if we take large enough constants $c_a$. As a matter of fact,  the conclusion  also follows under the more general assumption that $(X, \omega_X)$ is extremal, by  the proof of \cite[Thm.~3]{HFKG2}. The novelty of Theorem~\ref{extremal-bundle} is therefore in the fact that it gives a precise condition (in terms of $X$) for the existence of an extremal K\"ahler metric in a \emph{given} K\"ahler class $[\omega_Y]$, also revealing that $(X, [\omega_X])$ needs not to be extremal in general. 
We finally note that in the case of toric fibre, \cite{Jubert} provides a further equivalence with a certain weighted notion of \emph{uniform K-stability} of the corresponding Delzant polytope.

\smallskip
If all the factors $(B_a, \omega_{B_a})$ of the base are positive K\"ahler--Einstein manifolds,  and  the fibre $(X, \T)$ is a smooth Fano variety,  the semi-simple principal $(X, \T)$-fibration construction can produce a smooth Fano variety  $Y$ for suitable choice of the principal $\T$-bundle over $B$  (see Lemma~\ref{Fano} below). In this case, combining  \cite[Thm.~3.5]{HL} with the results in this paper,  we get
\begin{iTheorem}\label{v-soliton-bundle} Suppose $Y$ is a Fano semi-simple principal $(X, \T)$-fibration, obtained from the product of positive K\"ahler--Einstein Hodge manifolds $(B_a, \omega_{B_a})$  and a smooth Fano fibre $(X, \T)$ via Lemma~\ref{Fano}. Suppose also   that $\T$ is a maximal torus in the automorphism group $\Aut(X)$.  Then  $Y$ admits a $v$-soliton in $2\pi c_1(Y)$, provided that the weighted Mabuchi functional $\cE^X_{pv, \tilde w}$ of $(X, \T, 2\pi c_1(X))$ is coercive with respect to $\T^{\C}$, where $p$ is the weight defined in Theorem~\ref{extremal-bundle}-(ii) and
\[ \tilde w = 2pv\Big(m + \langle d\log v, \x \rangle + \langle d\log p, \x\rangle \Big). \]

If, furthermore,  the fibre  $(X, \T)$  is a smooth toric Fano variety, then the latter condition is equivalent to the vanishing of the  Futaki invariant \eqref{Futaki} associated to the weights $(pv, \tilde w)$ on $X$. In particular, any Fano semi-simple principal $(X, \T)$-fibration with smooth toric Fano fibre $(X, \T)$ admits a K\"ahler--Ricci soliton,  and the corresponding affine cone $(K_Y)^{\times}$  admits a Calabi--Yau cone metric, given by  a Sasaki--Einstein structure on a unit circle bundle associated to the canonical bundle $K_Y$.
\end{iTheorem}
The existence of a K\"ahler--Ricci soliton in the above setting is essentially known even though we didn't find it explicitly stated  in the literature.  In the toric case (i.e. when $Y=X$ and $B$ is a point) the result follows by  \cite{ZW} (see also \cite{DaSz}), and for  $\PP^1$-bundles by \cite{KS,DW,HFKG5}. In the general case,  the result  can be obtained from \cite{SP},  which in turn extend \cite{ZW} to  the framework of multiplicity-free manifolds, but the arguments can be also adapted to the case of semi-simple principal $(X,\T)$-fibrations (see \cite[Rem.7]{HFKG4} and \cite{Do-survey}). Our approach, however, builds on the idea of \cite{DaSz}. There are also related existence results for K\"ahler--Ricci solitons on  spherical manifolds, see \cite{Delcroix, Delgove}. On the other hand, the existence  of  Sasaki--Einstein metrics seems to be new in the above stated generality. Indeed, in the toric case the claim follows from \cite{FOW}, and there are known existence results \cite{GMSW, BT, MN} on $\PP^1$-bundles. We expect our arguments to extend to spherical manifolds too.

\subsection{Structure of the paper}  In Section~\ref{s:weighted-setup},  we recall the setup of weighted cscK metrics  and  state the main results we shall need from \cite{lahdili2, lahdili3}.  In Section~\ref{s:v-KRS},  we recall the notion of $v$-solitons from \cite{M1,HL}, and establish the equivalences stated in Propositions~\ref{tilde v} and \ref{SE}. Sections \ref{s:coercivity} and \ref{s:stability} review and recast in the weighted setting respectively the coercivity principle of \cite{DaRu} and its application to stability~\cite{BDL,dyrefelt2}, thus outlining the main steps needed for the proofs of Theorem~\ref{main} and deriving Corollary~\ref{stability} from the latter. In Section~\ref{s:geometric}, we  introduce the semi-simple principal  $(X, \T)$-fibration construction, and establish the main geometric properties allowing us to extend the results from \cite{HFKG4}. In Section~\ref{s:extensions}, we  use an idea from \cite{HL} in order to define an extension of the weighted Mabuchi energy to the space $\E^1(X, \bom)$, and show its convexity and compactness properties. In Section~\ref{s:regularity}, we extend the arguments of \cite{BDL}  to show that weak minimizers of the weighted Mabuchi energy are smooth. Here, we complete   the proof of Theorem~\ref{main}. In Section~\ref{s:proofs2}, we detail the proofs of Theorems~\ref{extremal-bundle} and \ref{v-soliton-bundle}. In the Appendix, we present some technical computational results,  detailing the linearization of the scalar and the twisted scalar curvature of a semi-simple principal $(X,\T)$-fibre and re-casting the weighted Futaki invariant \eqref{Futaki}, which are needed for the proofs of Theorem~\ref{extremal-bundle} and \ref{v-soliton-bundle}.

\section{Preliminaries on the weighted cscK problem}\label{s:weighted-setup} We recall the setup from \cite{lahdili2}. 
Let $X$ be a smooth compact, connected K\"ahler manifold of (real) dimension $2m$, and let \[
\cK (X, \bom)=\{\varphi \in C^{\infty}(X) \, | \,  \omega_{\varphi} : = \bom + dd^c \varphi >0\}
\]
be the space of $\bom$-relative smooth K\"ahler potentials on $X$. We let $\T \subset \Aut_r(X)$ be a fixed compact torus in the \emph{reduced} group  of automorphisms of $X$, i.e. the connected closed subgroup $\Aut_r(X)$ of  the group of complex automorphisms $\Aut(X)$, whose Lie algebra is the space of holomorphic vector fields of $X$ with zeros (see e.g. \cite{gauduchon-book}). Equivalently, $\Aut_r(X)$ is the connected component of the identity of the kernel of the natural group homomorphism from 
$\Aut(X)$ to the Albanese torus,  and is known to be isomorphic to the linear algebraic group in the Chevalley-type decomposition of $\Aut(X)$,  cf. \cite{Fujiki}.  We denote by $C^{\infty}_\T(X)$ the space of $\T$-invariant smooth functions on $X$ and  introduce the space
\[ \cK_\T(X, \bom) := \cK(X, \bom) \cap C^{\infty}_\T(X), \]
of $\T$-invariant relative K\"ahler potentials, assuming also that $\bom$ is $\T$-invariant.

It is well-known that the action of $\T$ on $(X, \bom)$ is hamiltonian, and  we let $\m_{0} : X \to \tor^*$ be a momentum map, where $\tor$ is the Lie algebra of $\T$ and $\tor^*$ the dual vector space. By the convexity theorem \cite{Atiyah,GS}, the image $\Pol:=\m_{0}(X) \subset \tor^*$ is a compact convex polytope. For any $\varphi \in \cK_{\T}(X, \bom)$, the smooth $\tor^*$-valued function
\begin{equation}\label{momenta}
\m_{\varphi} = \m_{0} + d^c\varphi 
\end{equation}
is the $\T$-momentum map of $(X, \omega_{\varphi})$,   normalized by the condition $\m_{\varphi}(X)=\Pol$. In the above formula, $d^c\varphi$ is viewed as a smooth $\tor^*$-valued function via the identity $\langle d^c\varphi, \xi \rangle := d^c\varphi (\xi)$ for  any  $\xi \in  \tor \subset C^{\infty}(X, TX)$.

\subsection{The $(v,w)$-constant scalar curvature K\"ahler metrics}
Following \cite{lahdili2}, let $v(\x)>0$ and $w(\x)$ be smooth functions on $\Pol$. One can then consider the condition  \eqref{weighted-cscK-0} for a $\T$-invariant K\"ahler  metric $\omega_{\varphi}$  in  $\alpha$ (and the fixed polytope $\Pol$), called \emph{$(v,w)$-cscK metric}. We thus want to solve the following PDE for $\varphi\in \cK_{\T}(X, \bom)$:
\begin{equation}\label{weighted-cscK}
\begin{split}
\Scal_{v}(\omega_{\varphi}) &= w(\m_{\varphi}), \\
\Scal_v(\omega_{\varphi}) &:= v(\m_{\varphi}) \Scal(\omega_{\varphi}) + 2\Delta_{\omega_{\varphi}} v(\m_{\varphi})  + \big\langle g_{\varphi}, \m_{\varphi}^*\left(\Hess(v)\right)\big\rangle
\end{split}\end{equation}
As we explained in the Introduction,  the problem of finding $\omega_{\varphi} \in \alpha$ solving \eqref{weighted-cscK}  is obstructed by the condition \eqref{Futaki}, and in the case when $v, w_0$ are positive weights, this can be resolved (similarly to the approach in \cite{calabi}) by  finding a unique affine-linear function $\ell^{\ext}_{v,w_0}(\x)$ on $\tor^*$, called \emph{the extremal function}, such that for any $\omega_{\varphi}$
\[\int_X \Big(\Scal_v(\omega_{\varphi}) - \ell^{\ext}_{v, w_0}(\m_{\varphi}) w_0(\m_{\varphi})\Big) \ell(\m_{\varphi})\omega_{\varphi}^{[m]} =0, \qquad \forall \ell \in \Aff(\tor^*). \]
Geometrically, the above condition means that the weighted cscK problem with weights $(v, w)=(v, \ell^{\ext}_{v,w_0}w_0)$ is unobstructed in terms of \eqref{Futaki},  and  a  solution $\omega_{\varphi}$ of the $(v, \ell^{\ext}_{v,w_0}w_0)$-cscK problem is referred to as \emph{$(v,w_0)$-extremal metric}.

\subsection{The weighted Mabuchi energy} 
\begin{defn}\cite{lahdili2}\label{d:Mabuchi}  Let $v, w$ be weight functions on $\Pol$ with $v(\x)>0$. The weighted Mabuchi  energy $\cE_{v, w}$  on $\cK_{\T}(X, \bom)$ is defined by
\[(d_{\varphi} \cE_{v, w})(\dot\varphi)= -\int_X\left(\Scal_{v}(\omega_{\varphi}) - w(\m_{\varphi})\right) \dot \varphi \, \omega_{\varphi}^{[m]}, \qquad \cE_{v, w}(0)=0. \]
\end{defn}
\begin{rem}\label{r:potentials/metrics} It  follows from the above definition  and the results in \cite{lahdili2} that for  a constant $c$, $\cE_{v, w}(\varphi+c) = \cE_{v, w}(\varphi)$ if and only if $v, w$ satisfy the integral relation
\begin{equation}\label{c-relation}
\int_X \Scal_v(\bom) \omega_0^{[m]} = \int_X w(\m_{0}) \omega_0^{[m]}. 
\end{equation}
Furthermore, by the results in \cite{lahdili2}, \eqref{c-relation} is a necessary condition for the existence of a solution of  \eqref{weighted-cscK}  and it is incorporated in the definition of $\cE_{v,w}$ given in \cite{lahdili2},   via the constant $c_{v, w}(\alpha)$ in front of $w$,  but we do not  assume  a priori  this condition in the current article.  It is however automatically satisfied if $\alpha$ admits a $\T$-invariant $(v, w)$-cscK metric, or if we consider the weights $(v, w)=(v, \ell^{\ext}_{v,w_0}w_0)$ corresponding to $(v, w_0)$-extremal K\"ahler metrics.  In these cases, we shall write $\cE_{v, w}(\omega_{\varphi})$ to emphasize  that the weighted Mabuchi functional acts of the space of $\T$-invariant K\"ahler metrics  in $\alpha=[\bom]$. \end{rem}

The following result is established in \cite{lahdili3}, generalizing \cite{BB} to arbitrary weights $v>0, w$.
\begin{Theorem} If $\omega$ is a $\T$-invariant $(v, w)$-cscK metric on $(X, \alpha, \T, \Pol)$, then for any $\varphi\in \cK_{\T}(X, \bom)$, \,  $\cE_{v, w}(\omega_{\varphi}) \ge \cE_{v, w}(\omega)$. 
\end{Theorem}

\subsection{The automorphism group of a $(v,w_0)$-extremal K\"ahler manifold} In what follows we will consider \emph{connected} Lie groups. We recall that we have set $\Aut_r(X)$ to be the connected component of the identity of the kernel of the Albanese homomorphism and,  similarly, we denote by $\Aut_r^{\T}(X)$ the connected component of the identity of the centralizer of the torus $\T$ in $\Aut_r(X)$. We shall use the following result, established in \cite[Thm.~B.1]{lahdili2} (cf. also \cite{FO}) and \cite[Rem.~2]{lahdili3}:
\begin{prop}\label{T-max} If $(X, \alpha, \T)$ admits a $(v, w_0)$-extremal K\"ahler metric $\omega$, then the connected component of the identity  $\Aut_r^{\T}(X)$  of the subgroup  of $\T$-commuting automorphisms in $\Aut_r(X)$ is reductive, and $\omega$ is invariant under the action of a maximal compact connected subgroup of $\Aut_r^{\T}(X)$. In particular, the isometry group of $(X, \omega)$ contains a maximal torus $\T_{\rm max} \subset \Aut_r(X)$ with $\T \subset \T_{\rm max}$. If, furthermore,  $\T=\T_{\rm max}$, then $\Aut_r^{\T}(X)= \T^{\C}$.
\end{prop}
Because of this result, we shall often assume (without loss of generality for solving \eqref{weighted-cscK}) that $\T=\T_{\rm max} \subset \Aut_r(X)$ and thus $\Aut_r^{\T}(X)=\T^{\C}$.

\subsection{Uniqueness of the $(v,w_0)$-extremal K\"ahler metrics} Another key result in the theory is the extension in \cite{lahdili3} of the uniqueness results \cite{BB, CPZ} to the weighted setting.
\begin{Theorem}\label{thm:uniqueness-weighted} Suppose $\omega, \omega'$ are $\T$-invariant  $(v, w_0)$-extremal K\"ahler metrics. Then there exists $\sigma \in  \Aut_r^{\T}(X)$ such that  $\sigma^*(\omega')=\omega$. In particular, if $\T \subset \Aut_r(X)$ is maximal, then the uniqueness holds modulo $\T^{\C}$.
\end{Theorem}

\section{$v$-solitons as weighted cscK metrics}\label{s:v-KRS} We review here the definition of $v$-solitons on a Fano manifold, following \cite{B-N,HL}, and discuss their link with $(v,w)$-cscK metrics. 

We thus suppose throughout this section that that $X$ is a smooth Fano manifold,   $\alpha:=2\pi c_1(X)$ and $\T\subset\Aut(X)$ a fixed  compact torus. (We recall here  that on a Fano manifold, $\Aut_r(X)$ coincides with the connected component of the identity  of the full automorphism group.) We further consider the \emph{natural} action of $\T$ on the anti-canonical bundle $K_X^{-1}$ of $X$, which normalizes the momentum map $\m_{\omega}$ of each $\T$-invariant K\"ahler metric $\omega \in \alpha$, and fixes the momentum image $\Pol$.  We shall sometimes refer to this normalization as the \emph{canonical normalization} of $\Pol$. In this setup, we recall
\begin{defn}\label{d:v-soliton} Let $v>0$ be a positive smooth weight  function on $\Pol$. A $v$-soliton on $X$ is a $\T$-invariant K\"ahler metric $\omega\in 2\pi c_1(X)$ which satisfies the relation \eqref{v-KRS}:
\[\rho_{\omega} - \omega = \frac{1}{2} dd^c \log v(\m_{\omega}).\]
 In the special case $v=e^{\langle \xi, \x\rangle}$ we obtain a K\"ahler--Ricci soliton in the sense of \cite{TZ}. \end{defn}

\begin{lemma}\label{v-KRS=(v,w)-cscK} A $\T$-invariant K\"ahler metric $\omega \in 2\pi c_1(X)$ is a $v$-soliton if and only if 
$\omega$ is a $(v, w)$-cscK metric with weight $w(\x)= 2v(\x)\Big[m + \langle \d \log v (\x), \x \rangle\Big]$.
\end{lemma}
\begin{proof}
We start by showing that \eqref{v-KRS} implies that $\omega$ is $(v,w)$-cscK with the weight $w$ specified in the Lemma.
 Taking the trace in \eqref{v-KRS} with respect to $\omega$ gives
\begin{align}
\begin{split}\label{trace-v-sol}
\Scal(\omega)-2m &=-\Delta_\omega\left(\log v(\mu_\omega)\right)= -\frac{1}{v(\m_{\omega})}\Delta_{\omega} (v(\m_{\omega})) -\frac{1}{v(\m_{\omega})^2}g_{\omega}\Big(d v(\m_{\omega}), d v(\m_{\omega})\Big) \\
=&-\sum_{i=1}^{m}\frac{v_{,i}(\m_\omega)}{v(\mu_\omega)}(\Delta_\omega \m_\omega^{\xi_i})
  +\sum_{i,j=1}^{m}\frac{v_{,ij}(\m_\omega)}{v(\m_\omega)}g_{\omega}(\xi_i,\xi_j) \\& -\sum_{i,j=1}^{m}\frac{v_{,i}(\mu_\omega)v_{,j}(\mu_\omega)}{v(\mu_\omega)^{2}}g_{\omega}(\xi_i,\xi_j)
\end{split}
\end{align}
where $(\xi_i)_{i=1,\cdots,\dtor}$ is a basis of $\tor$ and $v_{, i}$ denotes the partial derivative in direction of $\xi_i$. On the other hand, by taking the interior product of \eqref{v-KRS} with $\xi_i$ and using that $\xi_i$ is Killing with respect to $\omega$,  we get 
\[
\begin{split}
-d\Delta_\omega \m_\omega^{\xi_i}+2d\m_\omega^{\xi_i}&= d\Big(d^c\big(\log v(\m_{\omega})\big)(\xi_i)\Big)
=d\left(\sum_{j=1}^{m}\frac{v_{,j}(\m_\omega)}{v(\m_\omega)}g_{\omega}(\xi_i,\xi_j)_\omega\right),
\end{split}\]
where $\m_{\omega}^{\xi}:=\langle \mu_{\omega}, \xi \rangle$ is the momentum of $\xi$. It follows that 
\begin{equation}\label{int-v-sol}
-\Delta_\omega \m_\omega^{\xi_i}+ 2\m_\omega^{\xi_i}=\sum_{j=1}^{m}\frac{v_{,j}(\m_\omega)}{v(\m_\omega)}g_{\omega}(\xi_i,\xi_j) +c
\end{equation}
for some constant $c$.  As we consider the canonical normalization of $\m_{\omega}$ (corresponding to the natural
lifted $\T$-action on $K_X^{-1}$), one can see that $c=0$. Indeed, the infinitesimal actions $A_i$ of the elements of the basis $(\xi_i)_i$ on smooth sections of $K^{-1}_X$ are given by $A_i(s) := \mathcal{L}_{\xi_i}s$. We denote by $H_g$ the induced hermitian metric on $K_X^{-1}$ through the Riemannian metric $g_{\omega}$ of $\omega$ (so that $H_g$ has curvature $\rho_{\omega}$) and by $H=v(\m_{\omega})H_g$ the induced hermitian metric with curvature $\omega$ (by using \eqref{v-KRS}); comparing the actions of the corresponding Chern connections, $\nabla^g_{\xi_i}$ and $\nabla^H_{\xi_i}=\nabla^g_{\xi_i}  -\frac{\sqrt{-1}}{2} d^c \log v(\m_{\omega})(\xi_i) {\rm id}$ on smooth sections of $K^{-1}_X$ with the infinitesimal actions  $A_i$ gives  (see e.g. \cite[Prop. 8.8.2 \& 8.8.3]{gauduchon-book})
\begin{equation}\label{canonical action}
A_i(s)=\nabla^{g}_{\xi_i}s+\frac{\sqrt{-1}}{2}(\Delta_\omega \m^{\xi_i}_\omega)s, \qquad A_i(s)=\nabla^{H}_{\xi_i}s+\sqrt{-1}\mu^{\xi_i}_\omega s. \end{equation}
We thus deduce $\frac{1}{2}\Delta_{\omega} \m_{\omega}^{\xi_i} = \m_{\omega}^{\xi_i} - \frac{1}{2}d^c\big(\log v(\m_{\omega})\big)(\xi_i)$, 
i.e.  $c=0$ in \eqref{int-v-sol}. 

Now letting $c=0$ in  \eqref{int-v-sol},  multiplying it by $\frac{v_{,i}(\m_\omega)}{v(\m_\omega)}$,  and taking the sum over $i$,  give
\begin{align*}\label{int-v-sol-2}
\sum_{i,j=1}^{m}\frac{v_{,i}(\m_\omega)v_{,j}(\m_\omega)}{v(\m_\omega)^{2}}g_{\omega}(\xi_i,\xi_j)=\sum_{i=1}^{\dtor} \frac{v_{,i}(\m_\omega)}{v(\m_\omega)}(\Delta_\omega \m^{\xi_i}_\omega)-2\sum_{i=1}^{\dtor}\frac{v_{,i}(\m_\omega)}{v(\m_\omega)}\m^{\xi_i}_\omega,
\end{align*}
which substituting back in \eqref{trace-v-sol} yields
\[
\begin{split}
\Scal(\omega)-2m =&-2\sum_{i=1}^{\dtor} \frac{v_{,i}(\m_\omega)}{v(\m_\omega)} \Delta_{\omega} \m^{\xi_i}_\omega +\sum_{i,j=1}^{m}\frac{v_{,ij}(\m_\omega)}{v(\m_\omega)}g_{\omega}(\xi_i,\xi_j) + 2\sum_{i=1}^{\dtor}\frac{v_{,i}(\m_\omega)}{v(\m_\omega)}\mu^{\xi_i}_\omega \\
= &-2\sum_{i=1}^{\dtor} \frac{v_{,i}(\m_\omega)}{v(\m_\omega)} \Delta_{\omega} \m^{\xi_i}_\omega +2\sum_{i,j=1}^{m}\frac{v_{,ij}(\m_\omega)}{v(\m_\omega)}g_{\omega}(\xi_i,\xi_j) \\
&-\sum_{i,j=1}^{m}\frac{v_{,ij}(\m_\omega)}{v(\m_\omega)}g_{\omega}(\xi_i,\xi_j) + 2 \langle d\log v, \m_{\omega}\rangle \\
= & -2\Delta_{\omega} \left(v(\m_{\omega})\right) -\sum_{i,j=1}^{m}\frac{v_{,ij}(\m_\omega)}{v(\m_\omega)}g_{\omega}(\xi_i,\xi_j) + 2 \langle d\log v, \m_{\omega}\rangle.
\end{split}
\]
Thus $\Scal_v(\omega)=w(\mu_\omega)$. 

\bigskip
Now we show the converse. To this end, let  $\omega \in 2\pi c_1(X)$  be a  $\T$-invariant K\"ahler metric, $v>0$ a positive smooth function on the canonically normalized polytope $\Pol$ and  $w= 2(m+ \langle d\log, \x \rangle)v$ the weight defined in Lemma~\ref{v-KRS=(v,w)-cscK}. Let $h\in C^{\infty}_\T(X)$ be an $\omega$-relative Ricci potential, i.e.  \[\rho_\omega-\omega=\frac{1}{2}dd^{c}h.\]
 Taking the trace with respect to $\omega$  and the interior product with $\xi\in \tor$ in the above identity we get
 \begin{equation}\label{basic-relations}
\Scal(\omega) = 2m - \Delta_{\omega} h, \qquad \Delta_{\omega}\m_\omega^{\xi} + \cL_{J\xi} h= 2\m_{\omega}^{\xi}, \end{equation}
where we have used the canonical normalization of $\m_{\omega}$ to determine the additive constant in the second inequality (as we did for \eqref{int-v-sol}). Similar  computations as  in the first part of the proof (using  \eqref{basic-relations}) give 
\begin{equation}\label{pain}
\begin{split}
 \Scal_v(\omega)- & w(\m_\omega) \\
= &-v(\m_\omega)(\Delta_\omega h)+2\sum_{i=1}^{\dtor}v_{,i}(\m_\omega)(\Delta_\omega \m_\omega^{\xi_i})-\sum_{i,j=1}^{\dtor}v_{,ij}(\m_\omega)(\xi_i,\xi_j)_\omega
+2\sum_{i=1}^\dtor v_{,i}(\m_\omega)\m_\omega^{\xi_i}\\
=&-v(\m_\omega)(\Delta_\omega h)+\sum_{i=1}^{\dtor}v_{,i}(\m_\omega)g_{\omega}(dh, d\m^{\xi_i}_\omega)\\
&+\sum_{i=1}^{\dtor}v_{,i}(\m_\omega)(\Delta_\omega \m_\omega^{\xi_i})-\sum_{i,j=1}^{\dtor}v_{,ij}(\m_\omega)g_{\omega}(\xi_i,\xi_j) \\
=&-v(\m_\omega)(\Delta_{\omega,v} h)+\sum_{i=1}^{\dtor}v_{,i}(\m_{\omega})(\Delta_\omega \m_\omega^{\xi_i})-\sum_{i,j=1}^{\dtor}v_{,ij}(\m_\omega)g_{\omega}(\xi_i,\xi_j),
\end{split}
\end{equation}
where $\Delta_{\omega,v}:=\frac{1}{v(\m_{\omega})}\delta_\omega v(\m_{\omega}) d$ is the weighted Laplacian, see Appendix~\ref{s:appendix}. Using the second equality in \eqref{trace-v-sol}, we compute
\[
\begin{split}
v(\m_\omega)\Delta_{\omega, v}(\log v(\m_\omega)) &=v(\m_\omega)(\Delta_\omega \log v(\m_\omega))-\sum_{i=1}^{m}v_i(\mu_\omega)g_{\omega}(d(\log v(\m_\omega)),d\m_\omega^{\xi_i}) \\
&=v(\m_\omega)(\Delta_\omega \log v(\m_\omega))-\sum_{i,j=1}^{m}\frac{v_{,i}(\mu_\omega)v_{,j}(\mu_\omega)}{v(\mu_\omega)}g_{\omega}(\xi_i,\xi_j)\\
&=\sum_{i=1}^{\dtor}v_{,i}(\m_\omega)(\Delta_\omega \m_\omega^{\xi_i})-\sum_{i,j=1}^{\dtor}v_{,ij}(\m_\omega)g_{\omega}(\xi_i,\xi_j).
\end{split}\]
Substituting back in \eqref{pain} we obtain
\begin{equation}\label{Scal-v-tilde v}
\Scal_v(\omega)-  w (\m_\omega) = v(\m_{\omega}) \Delta_{\omega, v}\big(\log v(\m_{\omega}) - h\big).
\end{equation}
It follows that if $\omega$ is $(v, w)$-cscK then $h=\log v(\m_\omega)+c$ by the maximum principle, showing that $\omega$ satisfies \eqref{v-KRS}. \end{proof}

\begin{rem}\label{r:canonical-normalization} Using the second relation in \eqref{basic-relations} it follows that under the canonical normalization of $\m_{\omega}$  we have
\begin{equation}\label{canonical-normalization}
\int_X \mu^{\xi}_{\omega} e^h  \omega^{[m]} =0, \qquad \xi \in \tor. \end{equation}
This is precisely the normalization of $\mu_{\omega}$ used in \cite[Sect.2]{TZ}.
\end{rem}

\begin{lemma}\label{SE=v-KRS} Let $v:= \ell^{-(m+2)}$ where $\ell(\x)=\langle \xi, \x\rangle + a$ is a positive affine-linear function on $\Pol$. Then  $\omega \in 2\pi c_1(X)$ is a $v$-soliton if and only if $\omega$  is $(\ell^{-(m+1)}, 2ma\ell^{-(m+2)})$-cscK metric.
\end{lemma}
\begin{proof} The proof is similar to the one of Lemma~\ref{v-KRS=(v,w)-cscK}.  

If $\omega$ is a $v$-soliton with $v:= \ell^{-(m+2)}$, specializing \eqref{trace-v-sol} and \eqref{int-v-sol} to the specific choice of $v$,  and letting $f:=\ell(\m_{\omega})= \m_{\omega}^{\xi} + a$, we get the identities
\begin{equation*}
\begin{split}
\Scal(\omega) &= 2m + (m+2) \Delta_{\omega} \log f,  \qquad -\Delta_{\omega} f  + 2f = \frac{(m+2)}{f}g_{\omega}(df, df) + 2a.
\end{split}\end{equation*}
Multiplying by $f^2$ the first equality and taking the sum with the second  equality multiplied by  $mf$ gives
\begin{equation}\label{(m+2,f)}
f^2\Scal(\omega) - 2(m+1) f\Delta_{\omega} f -(m+1)(m+2) g_{\omega}(df, df) = 2m a f.\end{equation}
The RHS is the $(m+2, f)$-scalar curvature (see \cite{AC}) and it is straightforward check that the above equality is equivalent with the condition that $\omega$ is an $(\ell^{-(m+1)}, 2ma \ell^{-(m+2]})$-cscK metric.

\bigskip In the other direction, for any $\T$-invariant K\.ahler metric $\omega \in 2\pi c_1(X)$ we let $f:= \ell(\m_\omega)= \m_{\omega}^{\xi} + a>0$ be the corresponding Killing potential and  let $h\in C^{\infty}_{\T}(X)$  be such that $\rho_{\omega}-\omega =\frac{1}{2} dd^c h.$
From  \eqref{basic-relations} we have
\[\Scal(\omega) = 2m - \Delta_{\omega} h,  \qquad -\Delta_{\omega} f  + 2f = -g_{\omega}(df, dh) + 2a.\]
Multiplying the first identity by $f^2$ and summing with the second identity multiplied by $mf$ gives
\begin{equation}\label{Scal-f}
\begin{split} &f^2\Scal(\omega) - 2(m+1) f\Delta_{\omega} f -(m+1)(m+2) g_{\omega}(df, df) -2m a f \\
                      &= -f^2\Big(\Delta_{\omega}(h + (m+2)\log f) + m g_{\omega}\big(d\log f, dh + (m+2)d\log f\big)\Big).                      \end{split}\end{equation}
If we suppose that \eqref{(m+2,f)} holds,   we conclude again by the maximum principle that $((m+2)\log f +h)$ must be constant.      
\end{proof}
\begin{rem}\label{l:equivalent-weights} Lemmas~\ref{v-KRS=(v,w)-cscK} and~\ref{SE=v-KRS}  give two different realizations  of the same $\ell^{-(m+2)}$-soliton as a weighted cscK metric,  with respective weights $\Big(\ell^{-(m+2)}, 2\big(-2\ell + (m+2)a\big)\ell^{-(m+3)}\Big)$ and $\Big(\ell^{-(m+1)}, 2am\ell^{-(m+2)}\Big)$.
\end{rem}

We derive from Lemma~\ref{SE=v-KRS} and the correspondence in \cite{AC} the following fact, which does not seem to have been noticed before.

\begin{lemma}\label{l:SE} On a Fano manifold $(X, \T)$,  a $\T$-invariant K\"ahler metric $\omega \in 2\pi c_1(X)$ is a $\ell^{-(m+2)}$-soliton with respect to a positive affine linear function $\ell= \langle \xi, \x \rangle +a$ if and only if the  lift $\hat \xi$ of the vector field $\xi$ to $K_X$,  via the hermitian connection $\nabla^h$ with curvature $-\omega$ and the $\omega$-momentum $\ell(\m_{\omega})$ of $\xi$, is a Reeb vector of  a Sasaki--Einstein (transversal) structure of transversal scalar curvature $2am$, defined on the unit circle bundle $\Sm$ of $(K_X, h)$.
\end{lemma}
\begin{proof} By Lemma~\ref{SE=v-KRS}, we need to show that  an $(\ell^{-(m+1)}, 2am\ell^{-(m+2)})$-cscK metric in $2\pi c_1(X)$ corresponds to a Sasaki--Einstein structure as defined in the statement. By \cite[Thm.~1]{AC}, the condition that $\omega$ is $(\ell^{-(m+1)}, 2am\ell^{-(m+2)})$-cscK is equivalent to the condition that the corresponding Sasaki structure has transversal scalar curvature equal to $2m a$ (notice that $a>0$ by the positivity of $\ell$ over the canonical polytope $\Pol$). Any Sasaki structure of constant transversal scalar curvature on $\Sm \subset K_X$ is transversally K\"ahler--Einstein  as $c_1(K_X^{\times})=0$,  and therefore the first Chern class of the CR distribution of $\Sm$ vanishes (see e.g.  \cite[Cor.~5.3]{BGS} and \cite[Prop.~4.3]{FOW}). This completes the proof.
\end{proof}
\begin{rem} The correspondence in Lemma~\ref{l:SE} is, in fact,  local and can be deduced directly from the relation between the transversal Ricci tensors of the two Sasaki structures on the CR manifold $\Sm\subset K_X$, respectively defined by $\hat \xi$ and the regular Reeb vector field $\hat \chi$ (cf. \cite{JL,gauduchon-privite}).
\end{rem}

\begin{proof}[\bf Proofs of Propositions~\ref{tilde v} and \ref{SE}] Propositions~\ref{tilde v}  and ~\ref{SE} from the introduction follow directly from  Lemmas~\ref{v-KRS=(v,w)-cscK}, \ref{SE=v-KRS} and \ref{l:SE} above.
\end{proof}

\section{The  coercivity principle: Plan of proof of Theorem~\ref{main}}~\label{s:coercivity} We consider the following  general setup, based on the results of \cite{Da,DaRu,Tian}.  As before, we let $\T \subset \Aut_r(X)$ be a fixed connected compact torus in the reduced group of automorphisms of $X$, and denote by $\G=\T^{\C} \subset \Aut_r(X)$ the corresponding complex torus.

\smallskip
Following \cite{Da}, we consider the $L_1$-length function on $\cK(X, \bom)$, introduced on a smooth curve $\psi_t, t\in [0, 1]$ by
\[ L_1(\psi_t) := \int_{0}^1 \left(\int_X |\dot \psi_t|\omega_{\psi_t}^{[m]}\right) ds, \]
and,  for $\varphi_0, \varphi_1 \in  \cK(X, \bom)$,  we let
\[ d_1(\varphi_0, \varphi_1) := \inf\{L_1(\psi_t) \, | \, \psi_t \in \cK(X, \bom), t\in [0, 1] \, \psi_0=\varphi_0, \psi_1=\varphi_1\}.\]
Similarly we define $d_1$ on $\cK_{\T}(M, \bom)$ by considering infimum over smooth curves in $\cK_{\T}(X, \bom)$. It is proved in \cite{Da} that $(\cK(X, \bom), d_1)$ is a metric space, and it is observed in \cite{DaRu} that $(\cK_{\T}(X, \bom), d_1)$  is a metric subspace of $(\cK(X, \bom), d_1)$.

\smallskip
Recall the following well-known functionals on $\cK(X, \bom)$.

\begin{defn}\label{d:I, I-Aubin, J} Let ${\cI}$ denote the functional on $\cK(X, \bom)$ defined by
\[ (d_{\varphi} {\cI})(\dot \varphi) = \int_X \dot \varphi  \omega_{\varphi}^{[m]}, \qquad {\cI}(0)=0,\]
and let $\cJ(\varphi):= \int_X \varphi \omega_0^{[m]} - \cI(\varphi)$.
\end{defn}
\begin{rem}\label{r:J-invariance}  For any constant $c$,  $\cI(\varphi + c) = \cI(\varphi) + c\V(X, \bom)$ (where $\V(X, \bom)=\int_X \bom^{[m]}$ stands for the total volume of $(X, \bom)$) whereas  $\cJ(\varphi+c) = \cJ(\varphi)$, i.e. we can  see  $\cJ$ as a functional on the space of  K\"ahler metrics in the K\"ahler class $\alpha = [\bom]$, which  motivates the notation ${\cJ}(\omega_{\varphi})$. One can further show that $\cJ(\omega_{\varphi}) \ge 0$ with equality iff $\omega_{\varphi}=\bom$.
\end{rem}
By the above remark, for any K\"ahler metric  $\omega_{\varphi}$ in the K\"ahler class $[\bom]$,  there exists a uniquely determined $\bom$-relative potential $\varphi \in \cK(X, \bom)$ satisfying 
\[ \cI(\varphi) =0.\]
We shall denote by $\mathring{\cK}(X, \bom)$ (resp. $\mathring{\cK}_{\T}(X, \bom)$) the subspaces of \emph{normalized} $\bom$-relative K\"ahler potentials  satisfying the above equality. We notice that the group $\G=\T^{\C}$ naturally acts on the space of K\"ahler metrics in $[\bom]$, preserving the subspace of $\T$-invariant K\.ahler metrics. This induces an action $[\G]$ on the spaces  $\mathring{\cK}(X, \bom)$  and $\mathring{\cK}_{\T}(X, \bom)$,  such that
\[ \omega_{\sigma[\varphi]}= \sigma^*(\omega_{\varphi}), \qquad \forall \sigma \in \G, \, \varphi \in \mathring{\cK}(X, \bom). \]
We introduce the $\G$-relative distance on $\mathring{\cK}(X, \bom)$ and $\mathring{\cK}_{\K}(X, \bom)$ by 
\[d_1^{[\G]}(\varphi_0, \varphi_1) = \inf_{\sigma_0, \sigma_1 \in \G} d_1(\sigma_0[\varphi_0], \sigma_1[\varphi_1]).\]
It is proved in \cite{DaRu} that $d_1^{[\G]}$ is $\G$-invariant, i.e. $d_1^{[\G]}(\sigma[\varphi_0], \sigma[\varphi_1])= d_1^{[\G]}(\varphi_0, \varphi_1)$ and thus
\[ d_1^{[\G]}(\varphi_0, \varphi_1) = \inf_{\sigma\in \G} d_1(\varphi_0, \sigma[\varphi_1]). \]
\begin{defn}\label{proper}
Let $\cF$ be a functional on $\cK_{\T}(X, \bom)$. We say that  $\cF$ is  \emph{$\G$-coercive}  if there exist uniform positive constants $(\lambda, \delta)$ such that 
\begin{equation}\label{d_1-proper}
\cF(\varphi) \ge \lambda d_1^{[\G]}(0, \varphi)- \delta, \qquad \forall \varphi \in \mathring{\cK}_{\T}(X, \bom). 
\end{equation}
\end{defn}
It is some times more natural to introduce $\G$-coercivity in terms of the functional $\cJ$, via the following result 
\begin{prop}\cite{DaRu}\label{p:J-proper} $\cF$ is $\G$-coercive if and only if there exist uniform positive constants $(\lambda', \delta')$ such that 
\begin{equation}\label{J-proper}
\cF(\varphi) \ge \lambda'  \inf_{\sigma\in \G}\cJ(\sigma^*\omega_{\varphi}) - \delta', \qquad \forall \varphi \in \cK_\T(X, \bom).
\end{equation}
\end{prop}
\begin{rem}\label{proper=>bounded}
If $\cF$ is $\G$-coercive,  then it is bounded below by \eqref{d_1-proper}. \end{rem}

\smallskip
Following \cite{Da}, one can consider the metric completion $(\E^1(X, \bom), d_1)$ of $(\cK(X, \bom), d_1)$, which can be characterized  by a suitable continuously embedded subspace in  $L^1(X, \bom)$; 
similarly we let $(\E^1_{\K}(X, \bom), d_1)$ be the metric completion of $(\cK_{\K}(X, \bom), d_1)$ which, again by the results in \cite{DaRu}, can be viewed as the closed subspace of  $\T$-invariant elements of $\E^1(X, \bom)$. It will be  important for us that $(\E^1_{\T}(X, \bom), d_1)$ is a \emph{geodesic space}, i.e. each two elements $\psi_0, \psi_1 \in \E^1_{\T}(X, \bom)$ can be connected with a curve $\psi_t, t\in [0,1]$ in $(\E^1_{\T}(X, \bom), d_1)$, called a \emph{weak geodesic}, obtained as the limit of $C^{1, \bar 1}$-geodesics between elements of $\cK_{\T}(X, \bom)$, see \cite{C,Da}.  The latter object is a curve $\varphi_t \in \E^1_{\T}(X, \bom)$, of regularity $C^{1, 1}([0,1] \times X)$, which is uniquely associated to each $\varphi_0, \varphi_1 \in \cK_{\T}(X, \bom)$ (see \cite{C,Bl,Tosatti-et-al} and the proof of Proposition~\ref{p:totally-geodesic} below for more details about the weak $C^{1, \bar 1}$-geodesics). 
\smallskip

\smallskip In \cite[Thm.~3.4]{DaRu},   the following  general principle is established.

\begin{Theorem}[\bf Coercivity Principle]\label{thm:coercivity-principle} Let $\cF:  \cK_{\T}(X, \bom) \to \R$  be a lower semicontinuous (lsc) functional with respect to $d_1$, and $\cF : \E^1_{\T}(X, \bom) \to \R \cup \{+\infty\}$ be its largest lsc extension. Suppose, furthermore, that
$\cF(\varphi+c)=\cF(\varphi)=:\cF(\omega_{\varphi})$  and $\cF(\sigma^*\omega_{\varphi})= \cF(\omega_{\varphi})$ for any $\varphi\in \cK_{\T}(X, \bom)$ and $\sigma \in \G$, and that $\cF$ satisfies the following properties
\begin{enumerate}
\item[\rm (i)] {\rm (Convexity)} For each $\varphi_0, \varphi_1 \in \cK_{\T}(X, \bom)$ and the $C^{1, \bar1}$-geodesic $\varphi_t$ joining $\varphi_0$ and $\varphi_1$, $t \to \cF(\varphi_t)$ is continuous and convex.
\item[\rm (ii)] {\rm (Regularity)} If $\psi\in \E^1_{\T}(X, \bom)$ is a minimizer of $\cF$, then $\psi \in \cK_{\T}(X, \bom)$.
\item[\rm(iii)] {\rm (Uniqueness)} $\G$ acts transitively on the set of minimizers of $\cF$.
\item[\rm(iv)] {\rm (Compactness)} If $\{\psi_j\}_j \in \E^1_{\T}(X, \bom)$  satisfies $\lim_{j\to \infty} \cF(\psi_j)=\inf_{\E_{\T}^1(X, \bom)} \cF$ and, for some $C>0$, $d_1(0, \psi_j) \le C$, then there exists a $\psi\in \E^1_{\T}(X, \bom)$ and a subsequence $\{\psi_{j_k}\}_k$ with $\psi_{j_k} \to \psi$ in $(\E^1_{\T}(X, \bom), d_1)$.
\end{enumerate}
Then, the  following two conditions are equivalent:
\begin{enumerate}
\item[$\bullet$] $\cF$  has minimizer in $\cK_{\T}(X, \bom)$;
\item[$\bullet$] $\cF$ is \emph{$\G$-coercive}.
\end{enumerate}
\end{Theorem}
The above result provides a clear framework for achieving the proof of Theorem~\ref{main}: we need to find a suitable largest lsc extension of the weighted Mabuchi functional $\cE_{v,w}$ to the space  $\E^1_{\T}(X, \bom)$,  and show it satisfies the properties (i)--(iv). Notice that the invariance of $\cE_{v,w}$ under the action of $\G=\T^{\C}$ is equivalent to the necessary condition \eqref{Futaki} for the existence of a $(v,w)$-cscK metric whereas (iii) will follow from Theorem~\ref{thm:uniqueness-weighted} once the regularity condition (ii) is established. Furthermore, the property (i) is proved in \cite[Thm.~1]{lahdili3}, so the core of our arguments is to define the extension of $\cE_{v,w}$ to  $\E^1_{\T}(X, \bom)$ and establish  the properties (ii) and (iv). These steps will be respectively detailed in Theorems~\ref{M-extension},~\ref{W&S} and~\ref{compact} below. 

\section{K-stability via coercivity: Deriving Corollary~\ref{stability} from Theorem~\ref{main}}\label{s:stability} We consider the following general setup, based on the results of \cite{Be,BDL,dyrefelt2, Tian, BHJ, H, ChiLi} which deal with the K-polystability and uniform K-stability in the unweighted cscK case. Let  $\T \subset \Aut_r(X)$ be a connected compact torus in the reduced group of  automorphisms of $X$.
\begin{defn} A  $\T$-equivariant  K\"ahler test configuration $(\tstX, \tstA)$ associated to $(X, \alpha, \T)$ is a  \emph{normal} compact K\"ahler space $\tstX$ endowed with
\begin{itemize}
\item a flat morphism $\pi : \tstX \to \PP^1$;
\item a $\C^*$-action $\rho$  covering the standard $\C^*$-action on $\PP^1$,  and a $\T$-action commuting with $\rho$ and preserving $\pi$;
\item 
 a $\T\times \C^*$-equivariant biholomorphism  $\Pi_0 : (\tstX, \setminus \pi^{-1}(0)) \cong (X\times (\PP^1\setminus \{0\}))$;
\item a K\"ahler class $\tstA\in H^{1,1}(\tstX, \R)$ such that $(\Pi_0^{-1})^*(\tstA)_{|_{X \times \{\tau\}}} = \alpha$.
\end{itemize}
We say that $(\tstX, \tstA)$ is \emph{smooth} if $\tstX$ is smooth and \emph{dominating} if $\Pi_0$  extends to
a $\T\times \C^*$-equivariant  morphism 
\begin{equation}\label{dominate}
 \Pi : \tstX \to X \times \PP^1.
 \end{equation}
$(\tstX, \tstA)$  is called \emph{trivial} if  it is dominating and $\Pi$ is an isomorphism; $(\tstX, \tstA)$ is called \emph{product} if $\pi^{-1}(0) \cong X$. If $(X, L)$ is a smooth polarized variety and $\alpha = 2\pi c_1(L)$, a  \emph{polarized} test configuration is a  normal polarized variety $(\tstX, \tstL)$ such that  for some $r\in \N^*$, $(\tstX, \frac{1}{r}2\pi c_1(\tstL))$ defines a K\"ahler  test configuration  of $(X, \alpha)$ and, under $\Pi_0$,  $(X, \tstL_{|_{X \times \{\tau\}}}) \cong (X, L^r)$.
\end{defn}

\subsection{Non-Archimedean functionals}
We recall that any $\T\times \Sph^1$-invariant K\"ahler metric $\Omega \in \tstA$ on $\tstX$ gives rise to a smooth ray  of  $\T$-invariant K\"ahler metrics $\omega_t \in \alpha$ on $X$ defined by
\[ \omega_t := \rho(e^{-t + i s})^*(\Omega)_{|_{X\times\{1\}}}.\]
\begin{defn}\label{def:NA}
Let  $\cF$ be a functional defined on the space of $\T$-invariant K\"ahler metrics on $X$  in the class $\alpha$.   We say that ${\bf F}$ admits a \emph{non-Archimedean version} ${\bf F}^{\rm NA}$, defined on a subclass  $C$ of $\T$-equivariant K\"ahler test configurations  $(\tstX, \tstA)$ associated to $(X, \alpha, \T)$, if for any  $(\tstX, \tstA) \in C$, and any induced smooth ray of  $\T$-invariant K\"ahler metrics $\omega_t \in \alpha$ on $X$, the slope $\lim_{t\to \infty} \frac{{\bf F}(\omega_t)}{t}$ is well-defined and given by a quantity $\cF^{\rm NA}(\tstX, \tstA)$ which is independent of the choice of the $\T\times \Sph^1$-invariant K\"ahler form $\Omega \in \tstA$.
\end{defn}
We give below  two key examples of non-Archimedean versions of known functionals. The first one is established in the polarized case in \cite{BHJ} and in the generality we consider in \cite{dervan-ross, dyrefelt1}:
\begin{ex}\label{J-NA} The functional ${\bf J}$ introduced in Definition~\ref{d:I, I-Aubin, J} admits a non-Archimedean version defined, up to a positive dimensional multiplicative constant,  on the class of smooth $\T$-equivariant dominating K\"ahler test configurations $(\tstX, \tstL)$ by 
\[ {\bf J}^{\rm NA}(\tstX, \tstA) = \frac{\left((\Pi^*\alpha)^m \cdot \tstA\right)_{\tstX}}{(\alpha^m)_X} -\frac{1}{m+1}\frac{(\tstA^{m+1})_{\tstX}}{(\alpha^m)_X}, \]
where $\Pi$ is the morphism \eqref{dominate} and $\alpha$ denotes both the K\"ahler class on $X$ and its pull back to $X\times \PP^1$.
\end{ex}
The above expression generalizes to dominating smooth test configurations which are only \emph{relatively nef} (in the terminology of \cite{dyrefelt2}), thus also providing a non-Archimedean version of ${\bf J}$ for any K\"ahler test configuration: indeed,  by the equivariant Hironaka resolution, any  $\T$-equivariant test configuration can be dominated by a smooth  relatively nef  K\"ahler dominating test configuration, and the computation  of ${\bf J}^{\rm NA}$ on the latter does not depend on the choice made.

The non-Archimedean functional ${\bf J}^{\rm NA}$ defined above is always non-negative and equals to zero precisely when $(\tstX, \tstA)$ is the \emph{trivial} test configuration. The latter statement is established in  \cite[Thm.~7.9]{BHJ} in the polarized case,  and follows from the results in  \cite{dyrefelt2} in the K\"ahler case: see in particular  \cite[Lemma~4.8]{dyrefelt2} with $G$ trivial and recall that the ${\bf J}$-norm is Lipschitz equivalent to the $d_1$-distance, so that the unique weak geodesic ray associated to a test configuration with vanishing ${\bf J}^{\rm NA}$-norm must be constant,  and hence the test configuration must be trivial by \cite[Cor. 3.12]{dyrefelt2}. Thus, ${\bf J}^{\rm NA}$ can be thought of as a ``norm'' on the space of K\"ahler test configurations.

\smallskip
In order to obtain a norm which is zero  for more general \emph{product} test configurations, in \cite{dervan2,H,ChiLi}  the authors consider smooth rays $\tilde \omega_t \in \alpha$ of $\T$-invariant K\"ahler metrics on $X$ which are obtained by composing an induced ray $\omega_t$ from a $\T\times \Sph^1$-invariant K\"ahler metric $\Omega \in \tstA$ on $\tstX$ with the flow of a vector field $J\xi$,  where $\xi \in \tor$, i.e. $\tilde \omega_t = \exp(tJ\xi)^*(\omega_t)$. They show that the slope 
\[\lim_{t\to \infty} \frac{{\bf J}(\tilde \omega_t)}{t}=: {\bf J}^{\rm NA}(\tstX_{\xi}, \tstA_{\xi})\] is well-defined and  independent of the choice of   induced ray $\omega_t$.  We notice that when $\xi \in 2\pi \Lambda$ is a lattice element (or more generally is rational), $\xi$ induces an $\C^*$-action $\rho_{\xi}$ on $\tstX$ and   $\tilde \omega_t$ is an induced smooth ray from another K\"ahler test configuration $(\tstX_{\xi}, \tstA_{\xi})$,  called the \emph{$\xi$-twist} of $(\tstX, \tstA)$, obtained from $\tstX$ by composing the initial $\C^*$-action $\rho$ with $\rho_{\xi}$,  and compactifying trivially at infinity. (For instance, the product test configurations are precisely the $\xi$-twists of the trivial test configuration.) In  this case,  ${\bf J}^{\rm NA}(\tstX_{\xi}, \tstA_{\xi})$ is just the non-Archimedean $\bf J$-functional computed  as in Example~\ref{J-NA} on $(\tstX_{\xi}, \tstA_{\xi})$.  For a general $\xi$, the quantity $(\tstX_{\xi}, \tstA_{\xi})$ in the notation 
is not a test configuration in the usual sense (it is sometimes refereed to as a \emph{$\R$-test configuration})  but  the value ${\bf J}^{\rm NA}(\tstX_{\xi}, \tstA_{\xi})$ can be obtained as a continuous extension of the corresponding quantity for rational $\xi$'s. Following \cite{H,ChiLi}, we let
\begin{equation}\label{def:J-NA-G}
{\bf J}^{\rm NA}_{\T^{\C}}(\tstX, \tstA):=\inf_{\xi \in \tor} {\bf J}^{\rm NA}(\tstX_{\xi}, \tstA_{\xi}) \ge 0.
\end{equation}
A key observation~\cite{BHJ, H,ChiLi}  in the polarized case  is that the equality in \eqref{def:J-NA-G}  holds 
if and only if $(\tstX, \tstL)$ is  a \emph{product} test configuration. 
Furthermore, according to \cite[Thm.~B]{H} and \cite[Thm.~3.14]{ChiLi}, we have

\smallskip
\smallskip
\noindent
{\it Example}~\ref{J-NA}.-bis.  In the polarized case, the quantity ${\bf J}^{\rm NA}_{\T^{\C}}(\tstX, \tstA)$ introduced in \eqref{def:J-NA-G} defines a non-Archimedean version of the functional
\[{\bf J}_{\T^{\C}}(\omega):= \inf_{\sigma\in \T^{\C}} {\bf J}(\sigma^*(\omega)),\]
on the class of $\T$-equivariant polarized test configuration of $(X, L, \T)$.
\smallskip
\noindent

\smallskip
\noindent
Our second example is established in  \cite[Thm.~7]{lahdili2}\label{M-NA}:

 \begin{ex}\label{M-NA} Consider the weighted Mabuchi functional $\cE_{v,w}$ introduced in Definition~\ref{d:Mabuchi} and assume that the relation \eqref{c-relation} holds, see Remark~\ref{r:potentials/metrics}. Then  $\cE_{v,w}$ 
  admits a non-Archimedean version defined on \emph{smooth} $\T$-equivariant K\"ahler test configurations  with \emph{reduced central fibre}, given by the formula
 \begin{equation}\label{slope1}
 \begin{split}
 \Fut_{v,w}(\tstX, \tstA) := & -\int_{\tstX}\left(\Scal_v(\Omega)- w(\m_{\Omega})\right)\Omega^{[m+1]}
  + (8\pi) \int_X v(\m_{\omega}) \omega^{[m]}, \end{split}
\end{equation}
where  $\Omega\in \tstA$ is any $\T$-invariant K\"ahler metric on $\tstX$ with $\Pol$-normalized $\T$-momentum map $\m_{\Omega} : \tstX \to \Pol$ and  $v$-scalar curvature $\Scal_v(\Omega)$,  and $\omega\in \alpha$ is any  $\T$-invariant K\"ahler metric on $X$ with $\Pol$-normalized $\T$-momentum map $\m_{\omega} : X \to \Pol$. 
\end{ex}
\begin{defn}\label{d:Futaki}  The  RHS of \eqref{slope1} is independent of $\Omega \in \tstA$   and $\omega \in \alpha$ (see \cite{lahdili2}) and is referred to as the \emph{$(v,w)$-weighted Donaldson-Futaki invariant} of  a smooth $\T$-equivariant K\"ahler test configuration $(\tstX, \tstA)$.
 \end{defn}
\begin{rem} In the unweighted case (i.e. $v=1, w= 4m\pi \frac{c_1(X) \cdot \alpha^{m-1}}{\alpha^m}$), $\Fut_{v,w}(\tstX, \tstA)$ admits an equivalent expression in terms of intersection cohomology numbers on $\tstX$, see \cite{odaka, wang}. This allows one to extend the definition of the (unweighted) Donaldson--Futaki invariant to \emph{any} normal K\"ahler test configuration. For arbitrary weight functions $v>0$ and $w$,  we don't have as yet  a general definition for $\Fut_{v, w}$  but  \eqref{slope1} can be readily extended to orbifold test configurations. We also notice that the assumption on the central fibre in Example~\ref{M-NA} is necessary in order to ensure the equality $\Fut_{v,w} = \cE_{v,w}^{\rm NA}$ (see \cite{dyrefelt1} for a general formula of the non-Archimedean version of the unweighted Mabuchi energy). It will be interesting to obtain a non-Archimedean version of $\cE_{v,w}$ for any orbifold $\T$-equivariant  K\"ahler test configuration.
\end{rem}
\subsection{${\bf F}^{\rm NA}$-K-stability}
\begin{defn}\label{F-NA-stability}
Let  $\cF$ be  a functional defined on the space of $\T$-invariant K\"ahler metrics on $X$ in the K\"ahler class $\alpha$,  and suppose $\cF$ admits a non-Archimedean version  $\cF^{\rm NA}(\tstX, \tstA)$ (see Definition~\ref{def:NA}),  defined on a class  $C$ of $\T$-equivariant  K\"ahler test configurations  $(\tstX, \tstA)$ associated to $(X, \alpha, \T)$. We suppose that $C$ contains the product test configurations. We say that:
\begin{enumerate}
\item[(i)] $(X, \alpha, \T)$ is  \emph{$\T$-equivariant $\cF^{\rm NA}$-K-semistable} (on test configurations of the class $C$)  if  for any  $(\tstX, \tstA) \in C$ we have  $\cF^{\rm NA}(\tstX, \tstA) \ge 0$.
\item[(ii)] $(X, \alpha, \T)$ is  \emph{$\T$-equivariant $\cF^{\rm NA}$-K-polystable} (on test configurations of the class $C$) if it is $\K$-equivariant $\cF^{\rm NA}$-K-semistable,  and, furthermore,  $\cF^{\rm NA}(\tstX, \tstA) = 0$ if and only if $(\tstX, \tstA)$ is a product test configuration.
\item[(iii)]  $(X, \alpha, \T)$ is \emph{$\T^{\C}$-uniform $\cF^{\rm NA}$-K-stable} (on test configurations of the class $C$) if there exists a uniform positive constant $\lambda>0$ such that for any test configuration  $(\tstX, \tstA) \in C$ 
\[ \cF^{\rm NA}(\tstX, \tstA)  \ge \lambda {\bf J}^{\rm NA}_{\T^{\C}}(\tstX, \tstA), \]
where ${\bf J}^{\rm NA}_{\T^{\C}}(\tstX, \tstL)$ is introduced in \eqref{def:J-NA-G}.
\end{enumerate}
\end{defn}

\begin{rem}\label{proper=>semistable}
If $\cF$ is bounded below  then $(X, \alpha, \T)$ is  $\T$-equivariant $\cF^{\rm NA}$-K-semistable; furthermore both (ii) and (iii) imply (i) and, in the polarized case, (iii) implies (ii) by  the results in \cite{BHJ,H,ChiLi}.
\end{rem}

\begin{Theorem}\cite{BDL,H,  ChiLi, dyrefelt2}\label{proper=>stable} Suppose  $\cF$ is a functional defined on the space of $\T$-invariant K\"ahler metrics in $\alpha$,   which is $\K$-relatively $\T^{\C}$-proper.  Suppose, furthermore, that  $\cF$ admits a non-Archimedean version ${\bf F}^{\rm NA}$ defined for  a class $C$ of $\K$-equivariant K\"ahler test configurations of $(X, \alpha, \T)$. Then  $(X, \alpha, \T)$ is $\K$-equivariant $\cF^{\rm NA}$-K-polystable on $C$. If, moreover, $(X, L)$ is a polarized variety and $\alpha=2\pi c_1(L)$, then $(X, \alpha, \T)$ is  $\T^\C$-uniform $\cF^{\rm NA}$-K-stable on polarized test configurations in $C$. \end{Theorem}
\begin{proof}    For the first part, we follow  \cite{dyrefelt2} with some minor modifications.  We want to show that if  $\cF^{\rm NA}(\tstX, \tstA)=0$,  then $(\tstX, \tstA)$ is a product test configuration.

We fix a  $\K\times \Sph^1$-invariant K\"ahler form $\Omega \in \tstA$ and let $\omega_t$ be the corresponding ray of smooth $\T$-invariant K\"ahler forms in $\alpha$, and $\psi_t\in \cH_{\K}(X, \bom)$ the \emph{normalized} smooth ray of K\"ahler potentials  satisfying $\cI(\psi_t)=0$.   According to \cite{dyrefelt1},  the K\"ahler test configuration  $(\tstX, \tstA)$ also determines a unique $C^{1, \bar 1}$ weak geodesic ray  $\varphi_t$ in $\cH^{1,{\bar 1}}(X, \bom)$,  emanating from $\psi_0$. Furthermore, $\varphi_t$ is invariant under $\K$ (by its uniqueness) provided that we have $\psi_0 \in \cH_{\K}(X, \bom)$.   According to  \cite[Prop.4.2]{dyrefelt2}, we can consider instead of $\tstA$  the  \emph{relative K\"ahler class}  $\tstA_c = \tstA - c [X_0]=\tstA - c\pi^*(\cO_{\PP^1}(1))$ (for a constant $c$ determined from $\tstA$ and where $[X_0]$ denotes the divisor corresponding to the central fibre $X_0$ of $\tstX$),  such that  the  $C^{1,\bar 1}$ weak geodesic ray $\varphi^c_t$ corresponding to $(\tstX, \tstA_c)$ is  the projection of $\varphi_t$ to the slice $\cH^{1, \bar 1}_{\K}(X, \bom)\cap \cI^{-1}(0)$. Notice that the smooth $(1,1)$-form $\Omega - c\pi^{*}\omega_{\rm FS} \in \tstA_c$  defines  the same smooth ray $\omega_t$ of $\K$-invariant K\"ahler metrics,  and thus the same ray of  smooth potentials $\psi_t \in \cH_{\K}(X, \bom)\cap \cI^{-1}(0)$ and  $\cF^{\rm NA}(\tstX, \tstA_c)= \cF^{\rm NA}(\tstX, \tstA)=0$. The key point is that   \eqref{d_1-proper} and $\lim_{t\to \infty} \frac{\cF(\omega_{\psi_t})}{t}=\cF^{\rm NA}(\tstX, \tstA_c)=0$  yield an estimate $0\le d_1^{[\G]}(0, \psi_t) \le o(t)$, which  is shown in \cite[Lemma 4.8]{dyrefelt2} to be equivalent to  $0 \le d_1^{[\G]}(0, \varphi_t^c) \le o(t)$. We can now apply the arguments in the proof of the implication `$(2) \Rightarrow (5)$' of \cite[Thm.~4.4]{dyrefelt2}, by replacing  the Mabuchi energy with the abstract functional $\cF$ and the group $\Aut_{0}(X)$ with $\T^{\C}$,  noting that
 in our $\K$-relative situation instead of the cscK potential  $\psi_0$ in \cite[Prop. 4.10]{dyrefelt2}  we can take \emph{any} K\"ahler potential in $\cH_{\K}(X, \bom)$ (as $\omega_{\psi_0}$ is $\K$-invariant and  $\T^{\C}$ is reductive). We thus deduce the implication (5) of \cite{dyrefelt2}, namely,  that the geodesic ray $\varphi^c_t$ associated to $(\tstX, \tstA_c)$ is given by the $\bom$-relative K\"ahler potentials of $\exp(tJ\xi)^*(\omega_{\psi_0})$ in $\cI^{-1}(0)$, where $\xi$ is a vector field in the Lie algebra of $\T$;  it follows from   \cite[Thm.~A.6]{dyrefelt2} that  $(\tstX, \tstA_c)$  and hence also $(\tstX, \tstA)$ is a product test configuration. 
 
 \smallskip
 The second part follows immediately from \eqref{J-proper} and Example~\ref{J-NA}-bis.\end{proof}

We next apply Theorem~\ref{proper=>stable} to $\cF=\cE_{v,w}$ and $\cF^{\rm NA} = \Fut_{v,w}$.
\begin{defn}  Let ${\bf F}^{\rm NA}= \Fut_{v, w}$, where $\Fut_{v,w}$ is defined on any smooth $\T$-equivariant test configuration  via  the formula \eqref{slope1}, see Definition~\ref{d:Futaki}. We then refer to the ${\bf F}^{\rm NA}$-K-stability notions introduced in the Definition~\ref{F-NA-stability} (i)--(iii) respectively as $\T$-equivariant $(v,w)$-K-semistability, $\T$-equivariant $(v,w)$-K-polystability, and $\T^\C$-uniform $(v,w)$-K-stability on $\T$-invariant dominating smooth K\"ahler test configurations with reduced central fibre.  \end{defn}
\begin{proof}[\bf Proof of Corollary~\ref{stability} modulo Theorem~\ref{main}] 
By the definition of $\cE_{v, w}$ (see Definition~\ref{d:Mabuchi}), we have  
\[ \cE_{v, w}(\varphi+ c) = \cE_{v, w}(\varphi) + c\int_{X}\big(\Scal_{v}(\omega_{\varphi}) - w(\m_{\varphi})\big)\omega_{\varphi}^{[m]}, \]
showing that if $\cE_{v, w}$ is bounded below on $\cK_{\T}(X, \bom)$ (in particular if  $\cE_{v, w}$ is  $\T$-relatively $\T^{\C}$-proper), then the relation \eqref{c-relation} holds  and  $\cE_{v, w}$ defines a functional on the space of $\T$-invariant K\"ahler metrics in $\alpha$ (see Remark~\ref{r:potentials/metrics}). In this case, Example~\ref{M-NA} tels us  that $\Fut_{v,w}(\tstX, \tstA)$ defines a non-Archimedean version of $\cE_{v,w}$.   
We can now apply Theorem~\ref{proper=>stable}.
\end{proof}

\section{Semi-simple principal fibrations}\label{s:geometric}

Let $(X, \omega)$ be a compact K\"ahler  $2m$-manifold,  endowed with a hamiltonian isometric action of an $\dtor$-dimensional torus $\T$.  As $\T$ will act on various spaces, we shall use at times upper  and under scripts to emphasize the space, on which $\T$ acts. For instance, $\T_X$ will denote the $\T$-action on $X$. Let  $\tor$ be the Lie algebra of $\T$ and $\Lambda \subset \tor$ the lattice of generators of circle groups in $\T$ (i.e. $\T = \tor /2\pi \Lambda$). We denote by  $\m_\omega : X \to \Pol  \subset \tor^*$ the normalized $\T_X$-momentum map of $\omega$, i.e. whose image is a fixed compact convex polytope $\Pol \subset \tor^*$.

Let $B= B_1 \times \cdots \times B_k$ be a $2n$-dimensional cscK manifold,  where each $(B_a, \omega_{B_a}), a=1, \ldots, k$ is a compact cscK Hodge K\"ahler $2n_a$-manifold  (i.e. $\frac{1}{2\pi} [\omega_{B_a}] \in H^2(B_a, \Z)$), and  $\pi_B : \bN \to B$  a principal $\T$-bundle  endowed with a connection $1$-form $\theta \in \Omega^1(\bN, \tor)$ with curvature
 \begin{equation}\label{d-theta}
d\theta = \sum_{a=1}^k (\pi_{B}^* \omega_{B_a}) \otimes p_a, \qquad p_a \in \Lambda. 
\end{equation}
\begin{rem}\label{r:Chern-class} The principle $\T$-bundle $P$  above can be described in terms of $r$ complex line bundles over $B$ as follows.
Fixing a lattice basis $\{\xi_1, \ldots, \xi_r\}$ of $\tor$,  and writing $p_a = \sum_{i=1}^r p_{ai} \xi_i, \, p_{ai} \in \Z, \, a=1, \ldots k$,  \eqref{d-theta} yields that $P$ is the (fiber-wise) product of $r$ principle ${\rm U}(1)$-bundles $P_i \to B$, where each $P_i$ is associated to a complex line bundle $L_i^{*}$ on $B$ with Chern class $2\pi c_1(L_i^*)= -\sum_{a=1}^r p_{ai}  \pi_B^*[\omega_{B_a}]$, i.e. we have
\begin{equation*}\label{Chern-Class}
 2\pi c_1(P) := -2\pi \sum_{i=1}^r c_1(L_i^*) \otimes \xi_i=\sum_{a=1}^k \pi_{B}^* [\omega_{B_a}] \otimes p_a. \end{equation*}
Fixing a connection $1$-form $\theta$  on $P$ as in \eqref{d-theta} amounts to introducing a hermitian metric $h_i^*$ on each $L_i^*$,   with curvature $-\sum_{a=1}^r p_{ai} \pi_B^*(\omega_{B_a})$, and identifying $P_i \subset L_i^{*}$ with the corresponding unitary $\Sph^1$-bundle.
\end{rem}

Let $\Ds = {\rm ann}(\theta) \subset T\bN$  be the horizontal distribution defined by $\theta$,   leading to a splitting
\[ T\bN = \Ds \oplus \tor_\bN, \]
where $\tor_\bN$ denotes the Lie algebra of $\T_\bN$ inside $C^{\infty}(\bN, T\bN)$,  corresponding to the $\T$-action $\T_\bN$ on $\bN$. The lift  $J_B$  of the integrable almost complex structure of $B$  to $\Ds$  gives rise to a  CR structure $(\Ds, J_B)$ on $\bN$ (of co-dimension $r$).

We further let $Z:=X \times \bN$ and consider the  induced  $\T$-action, denoted $\T_{Z}$,  generated by  $(-\xi_i^X + \xi_i^\bN)$ for any basis of $\Lambda$ as above. We thus  define
\[ Y := Z/\T_Z.\]
It follows that  $Y$ is a $2(m+n)$-dimensional smooth manifold,  and $\pi_Y : Z=X\times \bN \to Y$ is a principal $\T$-bundle over $Y$  whereas  $\pi_B: \bN \to B$  defines a  fibration $\pi_B : Y \to B$ with smooth fibres $X$, as summarized in the diagram below.

\begin{displaymath}
\begin{tikzcd}[arrows=rightarrow]
 & Z=X\times \bN \arrow{ld}{/\T_\bN}  \arrow{rd}{/\T_Z} \arrow{dd}{\pi_B}& \\ 
X\times B \arrow{rd}{\pi_B}  & &    Y \arrow{ld}{\pi_B} \\ 
&  B   & 
\end{tikzcd} 
\end{displaymath}

The $\T_X$-action on the factor $X$ in $Z=X\times \bN$ descends to a $\T$-action on $Y$, denoted $\T_Y$, which preserves each fibre (and thus coincides with the action of $\T_X$). Notice that the $1$-form $\theta$ also defines a connection $1$-form on $Z$ with horizontal distribution $\Hs$:
\begin{equation}\label{split-T(X x N)}
 T(X \times \bN) = \Hs \oplus \tor_{Z}, \qquad \Hs = TX \oplus \Ds= \mathrm{ann}(\theta), \end{equation}
giving rise to an induced CR structure $(\Hs, J=J_X \oplus J_B)$ of co-dimension $r$ on $Z$, which is clearly invariant under the $\T_Z$-action, and therefore defines a $\T_Y$-invariant complex structure $J_Y$ on $Y$.

We now consider K\"ahler metrics on $Y$, compatible with the  fibre-bundle construction of the above form. To simplify the notation, we denote by $\omega_a:=\omega_{B_a}$ the (fixed) cscK metric on each factor $B_a$,  by $\omega$ a $\T$-invariant K\"ahler structure in the class $\alpha$ on $X$,  and by $\tilde \omega$ the resulting K\"ahler structure on $Y$, which is defined in terms of a basic $2$-form on $Z=X\times \bN$,  depending on  $k$ real constants $c_a\in \R$ (which will be fixed) such that for each $a=1, \ldots, k$,  the affine linear function $\langle p_a, \x \rangle + c_a$ on $\tor^*$  is strictly positive on the momentum image $\Pol$:
\begin{equation}\label{Y-Kahler}
\begin{split}
\tilde \omega  &:= \omega  + \sum_{a=1}^k (\langle p_a, \m_{\omega}\rangle + c_a)\pi_B^*\omega_a +  \langle d \m_{\omega} \wedge  \theta \rangle \\
                     &= \omega + \sum_{a=1}^k c_a (\pi_B^*\omega_a) + d \left(\langle \m_{\omega}, \theta \rangle\right). 
                     \end{split}
\end{equation}
In the above expression, $\langle \cdot, \cdot \rangle$ stands for the natural pairing between $\tor$ and $\tor^*$:  thus $\langle p_a, \m_{\omega}\rangle$ is a smooth function, $\langle \m_{\omega}, \theta \rangle$ is a $1$-form,  and $\langle d\m_{\omega} \wedge \theta \rangle$ is a $2$-form on $Z$. One can directly check from the above expression that $\tilde \omega$ is closed, $\T_{Z}$-basic, and is positive definite on $(\Hs, J_X \oplus J_B)$, so it is the pullback of a K\"ahler form on $Y$. We shall tacitly identify in the sequel the K\"ahler form on $Y$ with its pullback \eqref{Y-Kahler} on $Z=X\times \bN$. Notice that $\tilde \omega$ is $\T_Y$-invariant and $\m_\omega$, seen as a smooth $\T_Z$-invariant  function  on $Z$, is the $\Pol$-normalized momentum map.

\begin{rem}\label{r:warped-geometry} The  horizontal part $\tilde \omega_h := {\tilde \omega}_{|_{\Hs}}$ of  $2$-form $\tilde \omega$ on $Z= X\times \bN \stackrel{\pi_B}{\longrightarrow} X \times B$ is invariant and basic with respect to the action $\T_\bN$ on the factor $\bN$, and thus induces a Hermitian (non-K\"ahler in general) metric on $X\times B=X\times \prod_{a=1}^k B_a$, given by
\[\tilde\omega_h = \omega + \sum_{a=1}^k (\langle p_a, \m_{\omega}\rangle + c_a)\omega_a, \]
which is an instance of warped geometry. On can thus think of $(X\times B, \tilde \omega_h)$ and $(Y, \tilde \omega)$  as being related by the twist construction of \cite{swann} applied to $(Z, \tilde \omega, \T_Z)$ and $(Z,  \tilde \omega, \T_\bN)$.
\end{rem}

\begin{defn}\label{d:(X, T)-principal} The K\"ahler manifold $(Y, \T_Y)$ constructed as above will be called a  \emph{semi-simple $(X, \T)$-principal fibration} associated to the K\"ahler manifold $(X, \T)$ and  the product cscK manifold $B= B_1\times \cdots \times B_k$. The $\T_Y$-invariant K\"ahler metric $\tilde \omega$ on $Y$ constructed from a $\T_X$-invariant K\"ahler metric $\omega$ on $X$ (and fixed cscK metrics $\omega_a$ on $B_a$) will be called \emph{bundle-compatible}. \end{defn}

\begin{rem} In the case when $(X, \T, \omega)$ is a toric K\"ahler manifold,  a semi-simple $(X, \T)$-principal fibration endowed with a bundle-compatible K\"ahler metric is an example of a semi-simple rigid toric fibration in the sense of \cite{HFKG4}, and is thus described by the \emph{generalized Calabi construction} with a global product structure on the base and no  blow-downs. \end{rem}

\subsection{The space of functions}\label{s:functions} The above bundle construction  gives rise to a natural embedding of the space ${C}_{\T}^{\infty}(X)$ of $\T_X$-invariant smooth functions on $X$ to the space $C^{\infty}_{\T} (Y)$ of $\T_Y$-invariant smooth functions on $Y$: for any $\varphi \in {C}_{\T}^{\infty}(X)$ we consider the induced function on $Z= X \times \bN$, which is clearly $\T_Z$-invariant, and thus descends to a smooth $\T_Y$-invariant function on $Y$. We shall  tacitly identify $\varphi$ and its image in $C^{\infty}_{\T} (Y)$, i.e. we shall consider 
\[ {C}_{\T}^{\infty}(X) \subset C^{\infty}_{\T} (Y). \]
Notice that the above embedding is closed in the Fr\'echet topology,  as we can identify a smooth $\T_X$-invariant function on $X$  with a smooth $\T_Y$-invariant function $\varphi$ on $Y$,  which has the property 
\[d_\bN(\pi_Y^*\varphi)=0\]
on $Z= X \times \bN$.

More generally, for any $\T_Y$-invariant smooth function $\psi \in C^{\infty}_{\T}(Y)$ its lift $\pi_Y^*\psi$ to $Z=X\times \bN$  a smooth function which is both $\T_Z$ and $\T_X$-invariant, or equivalently $\T_X$ and $\T_\bN$ invariant.  It thus follows  that $\pi^*_Y \psi$  can be equivalently viewed as a $\T_X$-invariant smooth function on $X\times B$, i.e. we have an identification
\begin{equation}\label{function-splitting}
C^{\infty}_{\T}(Y) \cong C^{\infty}_{\T}(X \times B).
\end{equation}
In particular, for any fixed point $x\in X$, we shall denote by  $\psi_x \in C^{\infty}(B)$ the induced smooth function on $B$, and for any fixed point $b\in B$ by $\psi_b \in C^{\infty}_{\T}(X)$ the induced function on $X$. We thus have the identification
\[{C}_{\T}^{\infty}(X)  \cong \{ \psi \in C^{\infty}_{\T} (Y) \, |  \, d_B \psi_x =0 \,  \forall x \in X\}.\]

\subsection{The space of bundle-compatible K\"ahler metrics}\label{s:geometric-construction} We shall next use the construction of \eqref{Y-Kahler} in order to identify the space $\cK_{\T}(X,\bom)$ of $\T_X$-invariant $\bom$-relative K\"ahler potentials on $X$   as a subset of  the space $\cK_{\T}(Y, \tilde \omega_0)$ of $\T_Y$-invariant $\tilde \omega_0$-relative K\"ahler potentials on $Y$. 

\begin{lemma}\label{l:Y-embedding} Let $\omega_{\varphi}=\bom + d_Xd^c_X \varphi$ be an  $\T_X$-invariant K\"ahler form on $X$ in the K\"ahler class $\alpha =[\bom]$, where $\varphi \in \cK_{\T}(X, \bom)$ is a $\T_X$-invariant smooth function on $X$. Denote by $\m_{\varphi}$ the momentum map of $\T_X$ with respect to $\omega_{\varphi}$, satisfying the normalization $\m_{\varphi}(X)=\Pol$, and by $\tilde \omega_{\varphi}$ the induced K\"ahler metric on $Y$,   via \eqref{Y-Kahler}. Then,
\[ \tilde \omega_{\varphi} = {\tilde \omega}_0 + d_Yd^c_Y \varphi, \]
where $\varphi$ stands for the induced smooth function on $Y$.
\end{lemma}
\begin{proof} Recall that $\m_{\varphi}=\m_0 + d^c\varphi$ (see \eqref{momenta}).
By \eqref{Y-Kahler}, the pullback of $\tilde \omega_{\varphi}$ to $Z=X \times \bN$ is 
\[ 
\begin{split}
\tilde \omega_{\varphi} &= \omega_{\varphi} + \sum_{a=1}^k c_a (\pi^*_B \omega_a) + { d}\langle \m_{\varphi},  \theta \rangle  \\
                                     &= \bom + \sum_{a=1}^k c_a (\pi^*_B \omega_a) + d_Xd^c_X \varphi + {d} \langle \m_{\varphi}, \theta \rangle  \\
                                     &= \tilde \omega_0 + {d}d^c_X \varphi + {d}\big(\langle d^c_X \varphi, \theta \rangle \big), \end{split}\]
so it is enough to check that
\begin{equation}\label{d^c}
d^c_Y \varphi  =  d^c_X \varphi + \langle d^c_X \varphi, \theta \rangle,
\end{equation}
                                     for any $\T_X$-invariant smooth function  $\varphi$ on $X$. 
To this end, let us choose a basis $\{\xi_1, \ldots, \xi_{\dtor}\}$ of $\tor$, with dual basis $\{ \xi^1, \ldots, \xi^{\dtor}\}$ of $\tor^*$, and write  $d^c_X \varphi = \sum_{j=1}^{\dtor} (d^c_X \varphi)(\xi_j^X)\xi^{j}$ and $\theta = \sum_{j=1}^{\dtor} \theta_j \xi_j$ for $1$-forms $\theta_j$ on $Z$ such that $\theta_j$ is zero on $\Hs$  and $\theta_j(\xi_i^\bN)= \theta_j(-\xi_i^X + \xi_i^\bN)=\delta_{ij}$. Thus, \eqref{d^c} is equivalent to 
  \[d^c_Y \varphi = d^c_X \varphi  + \sum_{j=1}^{\dtor} (d^c_X \varphi )(\xi_j^X) \theta_j.\]
 Evaluating the RHS of the above equality on the generators $(-\xi_j^X + \xi_j^\bN)$ of $\tor_Z$, we see that it is a $\pi_Y$-basic $1$-form on $Z$, and thus is the pullback of a $1$-form on $Y$ via $\pi_Y$. The claim follows easily. \end{proof}

Thus, Lemma~\ref{l:Y-embedding} defines an embedding ${\cK}_{\T}(X, \bom) \subset {\cK}_{\T}(Y, \tilde \omega_0)$ and we have also identified in Sect.~\ref{s:functions} a natural embedding of the space of $\T_X$-invariant functions on $X$ into the space of $\T_Y$-invariant functions on $Y$,   through their pull-backs to $Z= X \times \bN$.  

Letting  $\theta:= \sum_{j=1}^{\dtor} \theta_j \otimes \xi_j^\bN$ be the decomposition of the connection $1$-form $\theta$ on $P$ in a  basis  $\{\xi_1, \ldots, \xi_{\dtor}\}$ of the lattice $\Lambda \subset \tor$,   and  $\theta^{\wedge{\dtor}}:= \theta_1 \wedge \cdots \wedge \theta_{\dtor}$, it follows from  \eqref{Y-Kahler} and Lemma~\ref{l:Y-embedding} that  for any $\varphi \in {\cK}_{\T}(X, \bom) \subset {\cK}_{\T}(Y, \tilde \omega_0)$, the measure $\tilde \omega_{\varphi}^{[m+n]}$ on $Y$ is the push-forward  of the measure on $Z$:
\begin{equation}\label{Z-measure}
\frac{1}{(2\pi)^r} \tilde \omega_{\varphi}^{[m+n]}\wedge\theta^{\wedge\dtor}=\frac{1}{(2\pi)^{r}} \left( p(\mu_\varphi) \omega_{\varphi}^{[m]}\wedge  \bigwedge_{a=1}^{k} \pi^{*}_B\omega_a^{[n_a]}\right) \wedge\theta^{\wedge \dtor}, \end{equation}
where 
\begin{equation}\label{p-weight}  p(\x):= \prod_{a=1}^k (\langle p_a, \x\rangle + c_a)^{n_a}, \qquad n_a= {\rm dim}_{\C}(B_a)\end{equation}
is a positive polynomial on $\Pol$,  determined by the semi-simple $(X, \T)$-principal fibration $Y$ and the given bundle-compatible K\"ahler class on it.  It thus follows that any $\T_X$-invariant  integrable function $f$ on $X$ defines an integrable $\T_Y$-invariant function on $Y$ and, for any $\varphi \in {\cK}_{\T}(X, \bom) \subset {\cK}_{\T}(Y, \tilde \omega_0)$, we have
\begin{equation}\label{measures}
\int_{Y} f  \, \tilde \omega_{\varphi}^{[n+m]} = \vol(B, \omega_B) \int_X p(\m_{\varphi}) f \omega_{\varphi}^{[m]}.
\end{equation}

\begin{cor}\label{c:bundle-compatible} There exists an embedding ${\cK}_{\T}(X, \bom) \subset {\cK}_{\T}(Y, \tilde \omega_0)$ such that, for any smooth curve $\psi_t \in {\cK}_{\T}(X, \bom) \subset {\cK}_{\T}(Y, \tilde \omega_0)$, we have 
\[ L^Y_{1}(\psi_t)=  \V(B, \omega_B) L^X_{1, \p}(\psi_t), \]
where  $p(\x)$ is the positive weight function on $\Pol$ defined in \eqref{p-weight}, $L_{1, \p}^X$ is the $\p(\x)$-weighted length  function on $\cK_{\T}(X, \bom)$ given by  \[ L_{1, \p}^X (\psi_t) := \int_{0}^{1} \left(\int_X |\dot \psi_t| \p(\m_{{\psi_t}})\omega_{{\psi_t}}^{[m]}\right) dt,\]  and $L_1^Y$ is the length function on $\cK_{\T}(Y, \tilde \omega_0)$ corresponding to the weight $\p=1$. In particular, for any $\varphi_0, \varphi_1 \in \cK_{\T}(X, \bom)\subset {\cK}_{\T}(Y, \tilde \omega_0)$, $d_1^Y(\varphi_0, \varphi_1) = \V(B, \omega_B) d_{1,\p}^X(\varphi_0, \varphi_1)$, where $d_{1, \p}^X$ is the induced distance via the length functional $L_{1,\p}^X$.
\end{cor}
\begin{proof}  A direct consequence of \eqref{measures}. \end{proof}

  \begin{lemma}\label{l:Y-norm} Let $\varphi$ be a  smooth  $\T_X$-invariant function on $X$, also considered as a smooth $\T_Y$-invariant function on $Y$, and $\omega$  be  an $\T_X$-invariant K\"ahler metric on $X$ with $\tilde \omega$ the corresponding $\T_Y$-invariant K\"ahler metric on $Y$ given by \eqref{Y-Kahler}. Then 
\[ || d\varphi ||^2_{\omega} =||d \varphi ||^2_{\tilde \omega}.\]
\end{lemma}
\begin{proof} We use that 
\[
\begin{split} 
||d\varphi||^2_{\omega} &= \frac{d_X\varphi \wedge d^c_X \varphi \wedge \omega^{[m-1]}}{\omega^{[m]}}=  \frac{d_X\varphi \wedge d^c_X \varphi \wedge \omega^{[m-1]}\wedge \p(\m_{\omega}) (\pi_B^* \omega_B)^{[n]}\wedge \theta^{\wedge\dtor}}{\omega^{[m]}\wedge p(\m_{\omega})(\pi_B^* \omega_B)^{[n]} \wedge \theta^{\wedge\dtor}} \\
|| d\varphi||^2_{\tilde \omega} &= \frac{d_Y\varphi \wedge d^c_Y \varphi \wedge \tilde \omega^{[m+n-1]}}{\tilde \omega^{[m+n]}} = \frac{d_Y\varphi \wedge d^c_Y \varphi \wedge \tilde \omega^{[m+n-1]}\wedge \theta^{\wedge\dtor}}{\tilde \omega^{[m+n]}\wedge \theta^{\wedge\dtor}} \end{split}\]
(where the RHS are written on $X\times \bN$)  together with $d_X \varphi = d_Y \varphi$ and \eqref{Z-measure}. 
\end{proof}
\begin{prop}\label{p:totally-geodesic} The embedding in Corollary~\ref{c:bundle-compatible} is totally geodesic with respect to the weak $C^{1, \bar 1}$ geodesics.
\end{prop}
\begin{proof} Let $\varphi_0, \varphi_1 \in \cK_{\T}(X, \bom)$. If $\varphi_0$ and $\varphi_1$  can be connected with a \emph{smooth} geodesic $\varphi_t$, i.e. with a smooth curve in $\cK_{\T}(X, \bom)$ such that
\begin{equation}\label{real-geodesic}
 \ddot{\varphi}  = ||d \dot \varphi||^2_{\omega_{\varphi}},
 \end{equation}
then, by Lemma~\ref{l:Y-norm}, it follows that  $\varphi_t$  is also a smooth geodesic in $\cK_{\T}(Y, \tilde \omega_0)$ connecting $\varphi_0, \varphi_1 \in \cK_{\T}(Y, \tilde \omega_0)$. 

In general, by the results in \cite{C}, $\varphi_0, \varphi_1$ can be connected only with a weak $C^{1, \bar 1}$-geodesic in $ \cK^{1, \bar 1}_{\T}(X, \bom)$, where $\cK^{1, \bar 1}(X, \bom)$ stand for the space of $C^1(X)$ functions  $\varphi$ on $X$ such that $\bom + dd^c\varphi \ge 0$  and has bounded coefficients as a  $(1,1)$-current. More precisely, letting $\Sigma:=\{1< z< e\} \subset \C$, 
it is shown in \cite{C} that there exists a unique weak solution (i.e.  a positive $(1,1)$-current in the sense of Bedford--Taylor) of the homogeneous Monge--Ampère equation
\begin{equation}\label{weak-geodesic}
\begin{split}
&(\pi_X^*\bom + d_Xd_X^c \Phi)^{m+1}=0, \qquad  \pi_X^*\bom + d_Xd_X^c \Phi \ge  0, \, \Phi \in C^{1, \alpha}(X\times \overline{\Sigma}),\\
& \Phi(x, 1) = \varphi_0(x), \qquad \hspace{1.5cm} \Phi(x, e)= \varphi_1(x).
\end{split}
\end{equation}
 It was later shown in \cite{Tosatti-et-al} that $\Phi$ is actually of regularity $C^{1,1}(X\times \overline{\Sigma})$. Note that, by the uniqueness, $\Phi$ is $\T$-invariant as soon as $\varphi_0$ and $\varphi_1$ are. The link with \eqref{real-geodesic} is (see \cite{Se}) that if $\Phi$ were actually smooth, we can recover the smooth geodesic $\varphi_t$ joining $\varphi_0$ and $\varphi_1$ by letting $t:= \log|z|$ and $\varphi_t(x) :=\Phi(x, \log |z|)$. In the general case, the curve $\varphi_t$ of (weak)  $\bom$-relative pluri-subharmonic potentials  (of regularity $C^{1,1}(X\times [0,1])$) is referred to as the \emph{weak}  $C^{1, \bar 1}$-geodesic joining $\varphi_0$ and $\varphi_1$.
 
 \smallskip
 We are thus going to check that any weak $C^{1, \bar 1}$-geodesic on $X$ (invariant under $\T_X$) defines, via Lemma~\ref{l:Y-embedding},  a  $C^{1, \bar 1}$-geodesic on $Y$. To this end, we need to show that $\Phi$  satisfies
 \begin{equation}\label{Y-geodesic}
 (\pi_Y^*\tilde \omega_0 + d_Yd_Y^c \Phi)^{m + n+1}=0, \qquad  \pi_Y^*\tilde \omega_0 + d_Yd_Y^c \Phi \ge  0, \end{equation}
 the regularity statements being automatically satisfied on $Y$.
 
 By the results in  \cite{C} and \cite{Bl}, $\Phi$ can be approximated as $\varepsilon \to 0$,  both it the weak sense of currents and in $C^{1, \alpha}( X\times \overline{\Sigma})$ (for a fixed $\alpha \in (0, 1)$),  by smooth functions  $\Psi^{\varepsilon}(x, z)$ on $X\times \overline{\Sigma}$ which solve
 \begin{equation}\label{epsilon-geodesic}
 \begin{split}
 &(\pi^*_X \bom + d_Xd_X^c \Psi^{\varepsilon})^{[m+1]} = \varepsilon \left((\pi_X^*\bom)^{[m]} \wedge (dx \wedge dy)\right),  \, \, \varepsilon >0, \\
 & \pi^*_X \bom + d_Xd_X^c \Psi^{\varepsilon}>0, \, \, \Psi^{\varepsilon}(x, 1)=\varphi_0, \, \, \, \Psi^{\varepsilon}(x, e)=\varphi_1.
 \end{split} 
 \end{equation}
 By the uniqueness of the smooth solution of \eqref{epsilon-geodesic} (and using that both $\varphi_0, \varphi_1$ are $\T_X$-invariant), we have that $\Psi^{\varepsilon}(x, z)$ is a $\T_X$-invariant smooth function on $X$ for any $z\in \overline{\Sigma}$; furthermore, the positivity condition on the second line yields that $\Psi^{\varepsilon}(x, z) \in \cK_{\T}(X, \bom)$ for any $z\in \overline{\Sigma}$. We can then also see $\Psi^{\varepsilon}(x, z)$, via its pull-back to $X\times \bN \times \overline{\Sigma}$,  as  a $\T_Y$-invariant smooth function on $Y \times \overline{\Sigma}$; the arguments in the proof of Lemma~\ref{l:Y-embedding} yield that  $\pi^*_Y \tilde \omega _0+ d_Yd_Y^c \Psi^{\varepsilon}>0$ on $Y\times \overline{\Sigma}$. Furthermore, by the same proof, we have the following equality of volume forms on $X\times \bN\times \overline{\Sigma}$:
 \begin{equation}\label{total-measure}
 \begin{split}
(\pi^*_Y \tilde \omega_0 + d_Yd_Y^c \Psi^{\varepsilon})^{[m+n+1]}\wedge \theta^{\wedge \dtor}  &= \p(\m_{\Psi^{\varepsilon}})(\pi^*_X \bom + dd^c \Psi^{\varepsilon})^{[m+1]} \wedge (\pi^*_B\omega_B)^{[n]} \wedge  \theta^{\wedge\dtor} \\
&= \varepsilon \p(\m_{\Psi^{\varepsilon}})(\pi^*_X \bom)^{[m+1]} \wedge (\pi^*_B\omega_B)^{[n]} \wedge  \theta^{\wedge\dtor},
\end{split}
\end{equation}
where, we recall, $p(\x): = \prod_{a=1}^k(\langle p_a, \x \rangle + c_a)^{n_a}$, $\theta^{\wedge \dtor} := \theta_1 \wedge \cdots \wedge \theta_{\dtor}$ (for  $\theta = \sum_{i=1}^{\dtor} \theta_i \otimes \xi^\bN_i$  with respect to a basis $\{\xi_1, \ldots, \xi_{\dtor}\}$ of $\Lambda \subset \tor$), and, for any fixed $z\in \overline{\Sigma}$, $\m_{\Psi^{\varepsilon}}$ denotes the normalized $\T_X$-momentum map \eqref{momenta} of $\bom + d_X d^c_X \Psi^{\varepsilon}$.
 Notice that, as  $\p$ is uniformly bounded on $\Pol$ by positive constants, it follows by \eqref{total-measure} that 
\[ \lim_{\varepsilon \to 0}  \left((\pi^*_Y \tilde \omega_0 + d_Yd_Y^c \Psi^{\varepsilon})^{[m+n+1]}\wedge \theta^{\wedge \dtor}\right)  =0, \]
weakly (as measures on $Z\times \overline{\Sigma}$). The push-forward measure of $(\pi^*_Y \tilde \omega_0 + d_Yd_Y^c \Psi^{\varepsilon})^{[m+n+1]}\wedge \theta^{\wedge \dtor}$ to $Y$ is the measure $(\pi^*_Y \tilde \omega_0 + d_Yd_Y^c \Psi^{\varepsilon})^{[m+n+1]}$, so we obtain on $Y$:
\[\lim_{\varepsilon \to 0}  \left((\pi^*_Y \tilde \omega_0 + d_Yd_Y^c \Psi^{\varepsilon})^{[m+n+1]}\right) =0.\]
 Furthermore, using the $C^{1, \alpha}$-convergence of $\Psi^{\varepsilon}$ to $\Phi$, we get the weak convergences (of positive $(1,1)$-currents):
\[ 
\begin{split}
& \lim_{\varepsilon \to 0} \left(\pi^*_Y \tilde \omega_0 + d_Yd_Y^c \Psi^{\varepsilon}\right) = \pi^*_Y \tilde \omega_0 + d_Yd_Y^c \Phi \ge 0;  \\
0=& \lim_{\varepsilon \to 0} \left(\pi^*_Y \tilde \omega_0 + d_Yd_Y^c \Psi^{\varepsilon}\right)^{[m+n+1]} = \left(\pi^*_Y \tilde \omega_0 + d_Yd_Y^c \Phi\right)^{[m+n+1]}.
\end{split} \]
Thus, \eqref{Y-geodesic} follows.
\end{proof}

\begin{lemma}\label{l:Y-Scal}   Let $v$ be a smooth positive weight function on $\Pol$ and $\omega$, $\tilde \omega$ be $\T$-invariant K\"ahler metrics respectively on $X$ and $Y$, given by \eqref{Y-Kahler}, and suppose $(B_a, \omega_a)$  has constant scalar curvature $\Scal(\omega_{a})=s_a$.  Then,  the $v$-scalar curvature  $\Scal(\tilde \omega)$, considered as smooth function on $X \times \bN$,  is given by
\begin{equation}\label{Y-Scalv}
\Scal_v(\tilde \omega) = \frac{1}{\p(\m_{{\omega}})}\Scal_{\p v}(\omega) + v(\m_\omega)\q(\m_{\omega}) 
\end{equation}
with $\p(\x)= \prod_{a=1}^k(\langle p_a, \x\rangle + c_a)^{n_a}$ and $\q(\x)=\sum_{a=1}^k \frac{s_a}{\langle p_a, \x\rangle + c_a}$. In particular,  $\omega$ is $(\p v, \tilde w)$-cscK metric on $X$  if and only if $\tilde \omega$ is a $(v, w)$-cscK metric on $Y$,   with 
\[\tilde w(\x)= \p(\x)(w(\x)-v(\x)\q(\x)). \]
\end{lemma}
\begin{proof} We apply the arguments in the proof  \cite[Prop.~7]{HFKG1}  to both $(X, \T_X)$ and $(Y, \T_Y)$ to compute the corresponding scalar curvatures,  and compare the results.

On $X$,  we consider the open dense subset $\mathring{X} \subset X$  of {stable} points of the  $\T_X$-action,  and take the quotient $\pi_S : \mathring{X} \to S :=\mathring{X}/\T_{X}^{\C}$ under the induced complexified action $\T^{\C}_{X} \cong (\C^*)^{\dtor}$ (thus $S$ is a complex $2(m-\dtor)$-dimensional orbifold).~\footnote{Our argument is actually local, around each point in  $\mathring{X}$,   so one can assume without loss that $S$ is smooth.}

Consider the point-wise $\omega$-orthogonal and $\T$-invariant decomposition 
\[ T \mathring{X} =\Hor \oplus \tor_X \oplus J \tor_X, \]
and write the K\"ahler structure $(g, J, \omega)$ on $X$ as
\[ 
\begin{split}
g &= g_{\Hor} + \sum_{i,j=1}^{\dtor} H_{ij} \left(\eta_i \otimes \eta_j + J\eta_i \otimes J\eta_j\right) ,  \\
\omega  &= \omega_{\Hor} +  \sum_{i,j=1}^{\dtor} H_{ij} \eta_i \wedge J\eta_j, \\
\end{split}\]
where,  for a fixed basis $\{\xi_1, \ldots, \xi_{\dtor}\}$ of $\tor$, the $1$-forms $\eta_j$ on $\mathring{X}$ are defined by $(\eta_j)_{|_{\Hor}}=0$,  $\eta_j(\xi_i^X)= \delta_{ij}$; $\eta_j(J\xi_i^X)=0$ and $H_{ij}=g(\xi_i^X, \xi_j^X)$.

We next fix a local volume form $\vol_S$ on $S$ in some holomorphic coordinates, and write point-wisely
\begin{equation}\label{Q}
\omega_{\Hor}^{[m-\dtor]} = Q \pi_S^*(\vol_S),
\end{equation}
for some positive (locally defined) smooth function $Q$ on $\mathring{X}$ (where both $\omega_H^{[m-1]}$ and $\pi_S^*(\vol_S)$ are seen as sections of $\Wedge^{m-1} \Hor^*$). According to \cite[Prop.~7]{HFKG1}, we have that
\begin{equation}\label{kappa}
 \kappa := -\frac{1}{2}\left(\log(Q) + \log \det (H_{ij}) \right)
\end{equation}
is a (local) Ricci potential of $\omega$, i.e. $\rho_{\omega} = d_Xd_X^c \kappa$, and thus
\[ \Scal(\omega) = -2\frac{d_Xd^c_X \kappa \wedge \omega^{[m-1]}}{\omega^{[m]}}.\]
We can now make a similar argument on $Y$, noting that the K\"ahler reduction of $\mathring{Y}$ by the induced $\T_Y$-action  is $S\times B$; taking a local volume form in holomorphic coordinates on $S\times B$ of the form $\vol_S \wedge \vol_{B_1} \wedge \cdots \wedge \vol_{B_k}$, and using \eqref{Y-Kahler}, we see that a Ricci potential on $Y$ (when pulled back to $X \times \bN$) is written as
\[\tilde \kappa = \sum_{a=1}^k \kappa_{a} -\frac{1}{2}\left(\log(\tilde Q) + \log \det(H_{ij})\right),\]
where $\kappa_a: = -\frac{1}{2}\log\left(\frac{\omega_a^{[n_a]}}{\vol_{B_a}}\right)$ is a Ricci potential of $(B_a, \omega_a)$ and 
\[\tilde Q = p(\mu_{\omega})Q.\]
Thus, we obtain
\begin{equation}\label{Ricci-potentials}
\tilde{\kappa}= \sum_{a=1}^k \kappa_{a} +\kappa -\frac{1}{2}\log p(\mu_{\omega}),\end{equation}
as functions on $X\times \bN$. Introducing a basis $(\xi_i)_i$  of $\Lambda$ and writing the connection $1$-form $\theta\in \Omega^{1}(\bN, \tor)$ as $\theta = \sum_{j=1}^{\dtor} \theta_j \otimes \xi_j^\bN$  (where the $1$-forms $\theta_j$ on $\bN$ are such that $\theta_j$ is zero on $\Ds$  and $\theta_j(\xi_i^\bN)= \delta_{ij}$), we compute for the scalar curvature of $\tilde{\omega}$ 
\begin{align}\label{scal-tomega}
\begin{split}
\Scal(\tilde \omega)=&-\frac{d_Yd^c_Y \tilde \kappa \wedge \tilde \omega^{[m+n-1]}}{\tilde \omega^{[m+n]}} \text{    (on }Y\text{)}\\
=&-\frac{d_Yd^c_Y \tilde \kappa \wedge \tilde \omega^{[m+n-1]}\wedge\theta^{\wedge\dtor}}{\tilde \omega^{[m+n]}\wedge \theta^{\wedge\dtor}} \text{    (on }X\times \bN\text{)}.
\end{split}
\end{align}
By  \eqref{d^c} and \eqref{Ricci-potentials}, the pullback of $d_Yd^c_Y \tilde \kappa$ to $X\times \bN$ is given by,
\begin{align}
\begin{split}\label{ddc-tkappa}
d_Yd^c_Y \tilde \kappa=&d_Yd^{c}_Y\left(\kappa-\frac{1}{2}\log p(\mu_{\omega})\right)+\sum_{a=1}^{k}d_Yd^{c}_Y\kappa_a\\
=&d_Xd^{c}_X\left(\kappa-\frac{1}{2}\log p(\mu_{\omega})\right)+\sum_{j=1}^{\dtor} d_X\left(d^{c}_X\left(\kappa-\frac{1}{2}\log p(\mu_{\omega})\right)(\xi^{X}_j)\right) \wedge \theta_j\\
&+\sum_{j=1}^{\dtor} d^{c}_X\left(\kappa-\frac{1}{2}\log p(\mu_{\omega})\right)(\xi_j)d_\bN\theta_j+\sum_{a=1}^{k} dd^c_{B_a}\kappa_a\\
=&d_Xd^{c}_X\kappa-\frac{1}{2}d_Xd^{c}_X\left(\log p(\mu_{\omega})\right)+\sum_{j=1}^{\dtor} d_X\left(d^{c}_X\left(\kappa-\frac{1}{2}\log p(\mu_{\omega})\right)(\xi^{X}_j)\right)\wedge \theta_j\\
&+\sum_{a=1}^{k} d^{c}_X\left(\kappa-\frac{1}{2}\log p(\mu_{\omega})\right)(p_a)(\pi_B^*\omega_a)+\sum_{a=1}^{k} dd^{c}_{B_a}\kappa_a,
\end{split}
\end{align}
where in the last equality we used \eqref{d-theta} and we have denoted by $p_a$ the induced vector field on $X$ by the element $p_a\in \tor$. We shall compute the term $d^{c}_X\kappa(p_a)$ on $\mathring{X}$: using \eqref{kappa} we get 
\begin{equation}\label{d-kappa}
d^{c}_X\kappa(p_a)=\frac{1}{2}\left(\frac{\mathcal{L}_{Jp_a}Q}{Q}+\tr\left(H_{ij}^{-1}(\cL_{Jp_a}  H_{ij}\right)\right).
\end{equation}
Taking the wedge product of both sides of \eqref{Q} with 
\[\left(\sum_{i,j=1}^{\dtor} H_{ij} \eta_i \wedge J\eta_j\right)^{[\dtor]}=\det(H_{ij})\bigwedge_{j=1}^{\dtor}  (\eta_j \wedge J\eta_j),\]
 gives
\[
\omega^{[m]}=Q \pi^{*}_S\vol_S\wedge \det(H_{ij})\bigwedge_{j=1}^{\dtor}  (\eta_j \wedge J\eta_j).
\]
Applying the Lie derivative $\mathcal{L}_{Jp_a}$ to the above equality yields
\[
\begin{split}
& \left(\Delta_{\omega} \mu_{\omega}^{\p_{a}}\right) \omega^{[m]} =(\mathcal{L}_{J \xi_{a}}Q)\pi^{*}_S\vol_S\wedge \det(H_{ij})\bigwedge_{j=1}^{\dtor}  (\eta_j \wedge J\eta_j)\\
&\hspace{3cm} +Q\pi^{*}_S\vol_S\wedge \mathcal{L}_{Jp_a}\left(\det(H_{ij}) \bigwedge_{j=1}^{\dtor}\eta_j \wedge J\eta_j\right)\\
& Q\pi^{*}_S\vol_S\wedge \mathcal{L}_{Jp_a}\left(\det(H_{ij}) \bigwedge_{j=1}^{\dtor}(\eta_j \wedge J\eta_j)\right)\\
& \hspace{2.5cm}= \Big(\tr\left(H_{ij}^{-1}(\cL_{Jp_a} H_{ij}) \right)\Big) Q\pi^{*}_S\vol_S\wedge(\det(H_{ij}))\bigwedge_{j=1}^{\dtor}\eta_j \wedge J\eta_j, 
\end{split}
\]
where we used that $\mathcal{L}_{Jp_a}\eta_j$ is a basic form (since $(\mathcal{L}_{Jp_a}\eta_j)(\xi_i)=-\eta_j([Jp_a, \xi_i])=0$).
We thus get  $\Delta_{\omega} \mu_{\omega}^{p_{a}} =\frac{\mathcal{L}_{Jp_a}Q}{Q}+\tr\left(H_{ij}^{-1}(\cL_{Jp_a}H_{ij})\right)$, or equivalently, in terms of \eqref{d-kappa} 
\begin{equation}\label{m-rho}
d^{c}_X\kappa(p_a)=\frac{1}{2}\left(\Delta_{\omega}\m_{\omega}^{p_{a}}\right).\end{equation}
Using the above equation in \eqref{ddc-tkappa}, we continue the computation
\begin{align}
\begin{split}\label{ddc-tkappa1}
d_Yd^c_Y \tilde \kappa=&d_Xd^{c}_X\kappa-\frac{1}{2}dd^{c}_X\left(\log p(\mu_{\omega})\right)+\sum_{j=1}^{\dtor} d_X\left(d^{c}_X\left(\kappa-\frac{1}{2}\log p(\mu_{\omega})\right)(\xi^{X}_j)\right)\wedge \theta_j\\
&+\frac{1}{2}\sum_{a=1}^{k} \left(\Delta_{\omega} \m_{\omega}^{p_{a}} + \frac{(\cL_{Jp_a} (p(\mu_{\omega}))}{p(\mu_{\omega})}\right)(\pi_B^*\omega_a)+\sum_{a=1}^{k} d_{B_a}d^{c}_{B_a}\kappa_a.
\end{split}
\end{align}
Recall that \eqref{Z-measure} on $Z$  we have $\tilde \omega^{[m+n]}\wedge\theta^{\wedge\dtor}= p(\mu_\omega) \omega^{[m]}\wedge  \bigwedge_{a=1}^{k} \pi^{*}_B\omega_a^{[n_a]} \wedge\theta^{\wedge \dtor}$. 
Similarly,  by \eqref{Y-Kahler},  
\begin{align}
\begin{split}\label{omega^(m+n-1)}
\tilde \omega^{[m+n-1]}\wedge\theta^{\wedge\dtor}=& \sum_{b=1}^{k}\left(\frac{p(\mu_\omega)}{(\langle \mu_\omega,p_b\rangle + c_b)}\omega^{[m]}\wedge(\pi^{*}_B\omega_b)^{[n_b-1]}\wedge \bigwedge_{\substack{a=1\\a\neq b}}^{k} (\pi^{*}_B\omega_a)^{[n_a]}\wedge\theta^{\wedge\dtor}\right)\\
&+ p(\mu_\omega) \omega^{[m-1]}\wedge \bigwedge_{a=1}^{k} (\pi^{*}_B\omega_a)^{[n_a]}\wedge\theta^{\wedge\dtor}.
\end{split}
\end{align}
Using \eqref{scal-tomega},  \eqref{ddc-tkappa1}, \eqref{Z-measure} and \eqref{omega^(m+n-1)},   we obtain
\begin{align}
\begin{split}\label{Scalp}
\Scal(\tilde{\omega})=&  \Scal(\omega)+ \Delta_\omega(\log p(\mu_{\omega}))\\
&+ \sum_{a=1}^{k} \left( \frac{n_a}{(\langle \m_\omega,p_a\rangle + c_a)}\left[\Delta_{\omega}\m_{\omega}^{p_a} + \frac{\cL_{Jp_a} p(\m_\omega)}{p(\m_{\omega})}\right]+\frac{s_a}{(\langle \mu_\omega,p_a\rangle + c_a)}\right)\\
=& \Scal(\omega)+\sum_{a=1}^{k} n_a \Delta_\omega(\log(\langle \mu_\omega,p_a\rangle + c_a))\\
&+ \sum_{a=1}^{k}\left(\frac{n_a}{(\langle \mu_\omega,p_a\rangle + c_a)}\left[\Delta_{\omega}(\langle \mu_{\omega},p_{a}\rangle)+ \frac{\cL_{Jp_a}(p(\mu_\omega))}{p(\mu_{\omega})}\right]+ \frac{s_a}{(\langle \mu_\omega,p_a\rangle + c_a)}\right)\\
=&\Scal(\omega)- \sum_{a,b=1}^k\frac{n_an_b g(p_a, p_b)}{(\langle \mu_\omega,p_a\rangle + c_a)(\langle \mu_\omega,p_b\rangle + c_b)}  \\
& +\sum_{a=1}^{k}\left( \frac{2n_a\Delta_{\omega}(\langle \mu_{\omega},p_{a}\rangle)}{(\langle \mu_\omega,p_a\rangle + c_a)}+ \frac{n_a|\xi_{a}|^{2}_{g_\omega}}{(\langle \mu_\omega,p_a\rangle + c_a)^{2}} + \frac{s_a}{(\langle\mu_\omega,p_a\rangle + c_a)}\right).
\end{split}
\end{align}
On the other hand, using a basis $(\xi_i)$ of $\tor$ with a dual basis $(\xi^{i})$ of $\tor^*$, we compute
\[\begin{split}
\Scal_p(\omega):=& p(\mu_\omega)\Scal(\omega)+2\sum_{i=1}^{\dtor}p_{,i}(\mu_\omega)\Delta_{\omega}(\langle \mu_\omega,\xi_i\rangle)-\sum_{i,j=1}^{\dtor}p_{,ij}(\m_\omega)g_{\omega}(\xi_i,\xi_j)\\
=&p(\mu_\omega)\Scal(\omega)+2\sum_{i=1}^{\dtor}\Delta_{\omega} (\langle \mu_\omega,\xi_i\rangle)\sum_{a=1}^{k}\frac{n_a \xi^{i}(p_a)  p(\mu_\omega)}{\langle \mu_\omega,p_a\rangle + c_a}\\
&+\sum_{i,j=1}^{\dtor}g_{\omega}(\xi_i,\xi_j)\left[\sum_{a=1}^{k}\frac{n_a \xi^{i}(p_a) \xi^{j}(p_a)  p(\mu_\omega)}{(\langle \mu_\omega,p_a\rangle + c_a)^{2}}-\sum_{a, b=1}^{k}\frac{n_a n_b \xi^{i}(p_a)\xi^j(p_a) p(\mu_\omega) }{(\langle \mu_\omega,p_a\rangle + c_a)(\langle \mu_\omega,p_b\rangle + c_b)}\right]\\
=&p(\mu_\omega)\Scal(\omega)+2\Delta_{\omega}(\langle \m_\omega,p_a\rangle)\sum_{a=1}^{k}\frac{n_a p(\mu_\omega)}{(\langle \mu_\omega,p_a\rangle + c_a)}\\
&+\left[\sum_{a=1}^{k}\frac{n_a |p_a|^{2}_{g_\omega} p(\m_\omega)}{(\langle \m_\omega,p_a\rangle + c_a)^{2}}-\sum_{a, b=1}^{k}\frac{n_a n_b g_{\omega}(p_a,p_b)p(\mu_\omega)}{(\langle \mu_\omega,p_a\rangle + c_a)(\langle \mu_\omega,p_b\rangle + c_b)}\right].
\end{split}\]
Comparing the above expression with \eqref{Scalp}, we obtain
\begin{equation}\label{Scalp/p}
\Scal(\tilde \omega)= \frac{1}{\p(\m_{\omega})}\Scal_{\p}(\omega) +\left(\sum_{a=1}^{k}\frac{s_a}{\langle\mu_\omega ,p_a\rangle + c_a}\right). \end{equation}
Using that  (as functions on $X\times P$)  $\mu_{\tilde{\omega}}=\mu_{\omega}$  and $g_{\omega}(\xi_i,\xi_j)=g_{\tilde \omega}(\xi_i,\xi_j)$ (see the proof Lemma~\ref{l:Y-norm}), we further compute from \eqref{Scalp/p}
\[
\begin{split}
\Scal_v(\tilde{\omega})=&v(\x_{\tilde{\omega}})\Scal(\tilde{\omega})+2\sum_{i=1}^{\dtor}v_{,i}(\x_{\tilde{\omega}})\Delta^{Y}_{\tilde{\omega}} (\langle\x_{\tilde{\omega}},\xi_i\rangle)-\sum_{i,j=1}^{\dtor}v_{,ij}(\x_{\tilde{\omega}})g_{\tilde \omega}(\xi_i,\xi_j)\\
=& \frac{v(\x_\omega)}{\p(\m_{{\omega}})}\Scal_{\p}(\omega) + v(\m_\omega) q(\m_{\omega})+2\sum_{i=1}^{\dtor}v_{,i}(\x_{\omega})\Delta^{Y}_{\tilde{\omega}} (\langle\x_{\omega},\xi_i\rangle)-\sum_{i,j=1}^{\dtor}v_{,ij}(\x_\omega)g_{\tilde \omega}(\xi_i,\xi_j)\\
=&\frac{v(\x_\omega)}{\p(\m_{{\omega}})}\Scal_{\p}(\omega) + v(\m_\omega) q(\m_{\omega})+2\sum_{i=1}^{\dtor}v_{,i}(\x_{\omega})\Delta^{X}_{\omega,p} (\langle\x_{\omega},\xi_i\rangle)-\sum_{i,j=1}^{\dtor}v_{,ij}(\x_\omega)g_{\omega}(\xi_i,\xi_j),
\end{split}
\]
where for passing to the last line we used the identity  $\Delta_{\tilde \omega}^Y=\Delta^{X}_{\omega, \p}$  established in Lemma~\ref{l:operators-tilde}.  As 
\begin{equation}\label{p-laplacian}
\Delta^{X}_{\omega, \p}(\psi):=\frac{1}{\p(\m_{\omega})} \delta_{\omega}\Big( p(\m_{\omega}) d \psi\Big) =\Delta_{\omega}^{X}(\psi)-\sum_{j=1}^{\dtor}\frac{\p_{,j}(\x_\omega)}{\p(\x_\omega)}g_{\omega}(d\mu^{\xi_j}_\omega, d\psi), \end{equation}
we further get 
\[
\begin{split}
&\frac{v(\x_\omega)}{\p(\m_{{\omega}})}\Scal_{\p}(\omega)+2\sum_{i=1}^{\dtor}v_{,i}(\x_{\omega})\Delta^{X}_{\omega,p} (\langle\x_{\omega},\xi_i\rangle)-\sum_{i,j=1}^{\dtor}v_{,ij}(\x_\omega)g_{\omega}(\xi_i,\xi_j)\\
=&v(\x_\omega)\Scal(\omega)+2\sum_{i=1}^{\dtor}\frac{v(\x_\omega)p_{,i}(\x_{\omega})}{p(\x_\omega)}\Delta^{X}_{\omega} (\langle\x_{\omega},\xi_i\rangle)-\sum_{i,j=1}^{\dtor}\frac{v(\x_\omega)p_{,ij}(\x_\omega)}{p(\x_\omega)}g_{\omega}(\xi_i,\xi_j)\\
&+2\sum_{i=1}^{\dtor}v_{,i}(\x_{\omega})\Delta^{X}_{\omega} (\langle\x_{\omega},\xi_i\rangle)-2\sum_{i,j=1}^{\dtor}\frac{v_{,i}(\x_\omega)p_{,j}(\x_\omega)}{p(\mu_\omega)}g^{X}(\xi_i,\xi_j) -\sum_{i,j=1}^{\dtor}v_{,ij}(\x_\omega)g_{\omega}(\xi_i,\xi_j)\\
=&\frac{1}{p(\mu_\omega)}\Big((pv)(\mu_\omega)\Scal_\omega+2\sum_{i=1}^{\dtor}(pv)_{,i}(\x_{\omega})\Delta^{X}_{\omega} (\langle\x_{\omega},\xi_i\rangle)-\sum_{i,j=1}^{\dtor}(pv)_{,ij}(\x_\omega)g_{\omega}(\xi_i,\xi_j)\Big)\\
=&\frac{1}{p(\x_\omega)}\Scal_{pv}(\omega).
\end{split}
\]
The expression \eqref{Y-Scalv} follows from the above formulae. \end{proof}

\begin{lemma}\label{l:Mabuchi} The restriction of the weighted Mabuchi energy $\cE_{v, w}^Y$ on $Y$ to the subspace ${\cK}_{\T}(X, {\bom})\subset {\cK}_{\T}(Y, {\tilde \omega_0})$ is equal to  $C{\cE}^{X}_{\p v, \tilde w}$, where $\p, w, \tilde w$ are given in Lemma~\ref{l:Y-Scal} and $C=\vol(B, \omega_B)$.
\end{lemma}
\begin{proof} A direct corollary of Lemma~\ref{l:Y-Scal} and Definition~\ref{d:Mabuchi}. \end{proof}

\bigskip
We now specialize to the case when each $(B_a, \omega_a)$ is a Hodge K\"ahler--Einstein manifold with positive scalar curvature $s_a =2n_a k_a,$  where $k_a \in \N$. Equivalently, $2\pi c_1(B_a)= k_a [\omega_a]$  for a positive  integer  $k_a$  and an integral K\"ahler class $\frac{1}{2\pi}[\omega_a]$.  Notice that $k_a$ must be a positive divisor of the Fano index  ${\rm Ind}(B_a)$ of $B_a$, which yields 
the a priori bound $1\le k_a \le {\rm Ind}(B_a)$. We also assume that $(X, \T)$ is Fano, with canonically normalized momentum polytope $\Pol$. We then have 
\begin{lemma}\label{Fano}  In the setting above, if  the affine linear functions $(\langle p_a, \x \rangle + k_a)>0$ on $\Pol$, then  the bundle-compatible K\"ahler metric $\tilde \omega$ on $Y$  corresponding to the constants $c_a=k_a$ belongs to deRham class $2\pi c_1(Y)$. Furthermore, $\tilde \omega$ is a $v$-soliton if and only if $\omega$ is a $pv$-soliton.
\end{lemma}
\begin{proof}  By using \eqref{Ricci-potentials} and rearranging the terms in \eqref{ddc-tkappa}, we have the following relation (written on $Z$):
\begin{equation}\label{ricci-principal}
\begin{split}
\rho_{{\tilde \omega}} =&  \rho_{\omega}  + \sum_{a=1}^k \left(\langle p_a, \m_{\rho_{\omega}}\rangle + c_a\right) \omega_a +   \langle d\mu_{\rho_{\omega}} \wedge \theta \rangle \\
                                     & + \sum_{a=1}^{k} \left(\rho_a - c_a\omega_a\right) -\frac{1}{2} d_Y d^c_Y \log p(\m_{\omega}),
                                  \end{split}\end{equation}
where $\rho_{\tilde \omega}$, $\rho_{\omega}$  and $\rho_a$ respectively denote the Ricci forms of $(Y,\tilde \omega)$, $(X,\omega)$ and $(B_a, \omega_{a})$,  pulled back to $Z$, and $\m_{\rho_{\omega}} := d^c_X \kappa$ is the ``momentum map'' with respect to the Ricci form $\rho_{\omega}$.
As in \eqref{m-rho}, we have $\m_{\rho_{\omega}} = \frac{1}{2} \Delta_{\omega} \m_{\omega}$. Suppose $\rho_{\omega} -\omega = \frac{1}{2} d_X d^c_X h$ for some $\T$-invariant smooth function on $X$; by using that the momentum polytope $\Pol$ is  canonically normalized, we have (see \eqref{basic-relations}) $\m_{\rho_{\omega}} -\m_{\omega} = d^ch$.  A closer look at the proof of Lemma~\ref{l:Y-embedding}  and the relation  \eqref{ricci-principal} (with $c_a=\frac{s_a}{2n_a}=k_a$) show that 
\[\rho_{\tilde \omega} - \tilde \omega = \frac{1}{2} d_Y d^c_Y \tilde h, \qquad \tilde h := h  - \log p(\m_{\omega}). \]
The claim follows from the above. \end{proof}

\begin{rem}\label{r:Fano} Lemma~\ref{Fano} provides a useful way to \emph{construct} semi-simple $(X, \T)$-principal Fano fibrations.  Indeed, for given positive Hodge K\"ahler--Einstein manifolds $(B_a, \omega_a)$ as above, with corresponding integer constants $k_a$,  and a given Fano manifold $(X, \T)$ with associated canonical polytope $\Pol$,  one can try to find the possible principal $\T$-bundles $P$ over $B= \prod_{a=1}^k B_a$,  for which the corresponding semi-simple $(X, \T)$-principal fibration is Fano.  Such principal $\T$-bundles $P$ are in correspondence with the choice of lattice elements $p_a \in \Lambda \subset \tor$ and Lemma~\ref{Fano} tels us that for a set of elements  $p_a$ to  determine a Fano semi-simple $(X, \T)$-principal fibration $Y$, it is sufficient to check that all the affine linear functions 
\[ \langle p_a,  \x \rangle + k_a >0 \, \, \textrm{on} \, \, \Pol. \]
For instance, if we take $B=B_1=\PP^1$ with a Fubini--Study metric $\omega_1$ of scalar curvature $4$ (so that $k_1=2$ and $\omega_1$ is primitive)  and $(X, \T)=(\PP^1, \Sph^1)$ with canonical polytope $\Pol=[-1,1]$, then the possible Fano $(\PP^1, \Sph^1)$-principal fibrations will correspond to $p_1\in \Z$ such that $p_1 \x + 2 >0$ on $[-1, 1]$, i.e. $p_1= \pm 1, 0$ are the only possible values. This gives rise to the Fano surfaces $\PP(\cO \oplus \cO(-1)) \cong \PP(\cO \oplus \cO(1))$ and  $\PP^1 \times \PP^1$. In general, the isomorphism class of  the principal $\T$-bundle $P$ over $B$, and hence also the semi-simple $(X, \T)$-principal Fano fibration constructed as above,  is encoded by the Hodge classes $\frac{1}{2\pi} [\omega_a]\otimes p_a = \frac{1}{k_a}c_1(B_a)\otimes p_a \in H^2(B, \Z)^{r}.$ The a priori bounds  $1\le k_a \le {\rm Ind}(B_a)$ for $k_a$  show that  for given base $B=\prod_{a=1}^k B_a$ and fibre $(X, \T)$, there are  only a finite number of  semi-simple $(X, \T)$-principal Fano fibrations constructed this way.
\end{rem}

\begin{rem}\label{canonical-consistency} The relationship between the Ricci potentials $\tilde h$ and $h$ established  in the proof of Lemma~\ref{Fano} and \eqref{measures}  yield, via Remark~\ref{r:canonical-normalization}, that  if the momentum map $\m_{\omega}$ of $(X, \omega, \T_X)$ is canonically normalized, then the momentum map $\m_{\tilde \omega} = \m_{\omega}$ of the corresponding bundle-compatible K\"ahler metric $\tilde \omega$ on $(Y, \T_Y)$ is also canonically normalized. 
\end{rem}

We end up this section with the following straightforward extension of \cite[Lemma 5]{HFKG4}.
\begin{lemma}\label{T} Suppose $Y$ is a semi-simple principal $(X, \T)$-fibration over $B$, such that $\T$ is a maximal torus in the reduced group  of automorphisms  $\Aut_r(X)$. Let  $\tilde \omega$ be a bundle-compatible K\"ahler metric on $Y$ corresponding to a $\T$-invariant K\"ahler metric $\omega$ on $X$,  and  $\Ko_B \subset \Aut_r(B)$ be a maximal compact torus in the reduced group of automorphisms of $B$  which (without loss by Lichnerowicz--Matsushima theorem) belongs to the isometry group of $\omega_B$. Then $\tilde \omega$ is invariant under the action of a maximal torus $\Ko_Y \subset \Aut_r(Y)$,  and we have an exact sequence {of Lie algebras}
\[ \{0\} \to {\rm Lie}(\T_Y) \to {{\rm Lie}}(\Ko_Y) \to {{\rm Lie}}(\Ko_B) \to \{0\}.\]
Furthermore, for any positive weight functions $v, w_0$ defined on $\Pol \subset \tor^*$, there exists a unique affine-linear function $\ell^{\rm ext}_{v, w_0}$ on $\tor^*$ such that,  when pulled-back to the dual Lie algebra $\ko_Y^*$ of $\Ko_Y$, $(v, w_0\ell_{v, w_0}^{\rm ext})$ satisfy \eqref{Futaki} with respect to $\tilde \omega$ on $Y$, and any affine-linear function $\ell$ on $\ko_Y^*$.
\end{lemma}
\begin{proof} The proof of the above result is not materially different than the proof of  \cite[Lemma~5]{HFKG4} (which is made in the case when $(X, \T)$ is toric and $v=w_0=1$). We only give a sketch.  A Killing potential  $f$ for a Killing vector field $K \in \ko_B:= {\rm Lie}(\Ko_B)$ is of the form $f=\sum_{a=1}^k f_a$, where $f_a$ is a Killing potential of $(B_a, \omega_a)$.  Letting $\tilde K$ be the horizontal lift of $K$ to $\bN$ (using the $\tor_\bN$-valued connection $1$-form $\theta$), one can check that the vector field on $\bN$
\[ \hat K = \tilde K + \sum_{a=1}^k f_a \xi^\bN_{p_a} \]
is a CR vector field on $(\bN, \Ds, J_B)$, hence also on  $(Z, \Hol, J_B \oplus J_X)$.  Furthermore, a direct verification in \eqref{Y-Kahler} reveals that
\begin{equation}\label{Killing-potential}
\imath_{\hat K} \tilde \omega = - d\left(\sum_{a=1}^k (\langle p_a, \mu_{\omega}\rangle + c_a)f_a \right)\end{equation}
showing that $\hat K$ also preserves $\tilde \omega$.
We thus obtain a lift $\hat \ko_B$ of the Lie algebra $\ko_B={\rm Lie}(T_B)$ to $Z$, which clearly commutes with the action $\T_Z$, and preserves both the CR structure of $(Z, \Hol)$ and the $2$-form $\tilde \omega$. The Lie algebra $\ko_Y$ of $\Ko_Y$ is then induced by $\tor_X \oplus \hat \ko_B \subset TZ$, which descend to an abelian Lie algebra of Killing fields on  $Y$. The maximality of $\Ko_Y\subset \Aut_r(Y)$  and the exactness of the sequence follow from the maximality of each $\Ko_B \subset \Aut_r(B)$ and  $\T \subset \Aut_r(X)$, and the fact that (recall that $Y$ is a locally trivial $X$-fibre bundle and therefore the fibres have trivial normal bundle) any holomorphic vector field on $Y$ projects under $\pi_B$ to a holomorphic vector field on $B$.  For the final claim in Lemma~\ref{T}, notice that  by  \eqref{Killing-potential}   the Killing potentials of all lifted Killing vector fields $\hat K$  from $B$ are of the form $\sum_{a=1}^k (\langle p_a, \mu_{\omega} \rangle + c_a) f_a$.  Thus, by Lemma~\ref{l:Y-Scal} and using \eqref{Z-measure},  the integral condition \eqref{Futaki} on $(Y, \tilde \omega)$  will be zero for any such Killing potential,  as soon as we normalize $\int_{B_a} f_a \omega_a^{n_a} =0$ and assume $\ell_{v, w_0}^{\rm ext} \in \Aff(\tor_X^*)$. On the other hand, examining \eqref{Futaki} on $(Y, \tilde \omega)$ for  the Killing potentials $\ell(\m_{\tilde \omega}), \, \ell \in \Aff(\tor^*)$ reduces (again by Lemma~\ref{l:Y-Scal} and \eqref{measures}) to an integral  relation on $(X, \omega)$ which defines a unique  element $\ell_{v,w_0}^{\rm ext} \in \Aff(\tor^*)$. \end{proof}

\section{Weighted functionals and distances and their extensions}\label{s:extensions}
Let $\bom$ a $\T$-invariant K\"ahler metric in the K\"ahler class $\alpha$. We denote by ${\PSH}_{\T}(X,\bom)$ the space of $\T$-invariant $\bom$-pluri-subharmonic functions in $L^1(X, \bom)$, and define the  class of potentials of full volume by
\[\begin{split}
\E_{\T}(X, \bom):=&\left\{\varphi\in\PSH_\T(X,\bom)\,\mid \,\int_X \MA(\varphi)=\int_{X} \omega_0^{[n]} \right\}
\end{split}\]
According to \cite{Da}, the $d_1$-completion of $\cK_{\T}(X, \bom)$ can be identified with the subspace of potentials of finite energy, i.e. 
\[\E^1_{\T}(X, \bom)= \left\{ \varphi \in \E_{\T}(X, \bom) \, | \, \int_X |\varphi | \MA(\varphi) < \infty \right\}. \]

Our main result in this section will be the existence of a lsc extension of the weighted Mabuchi functional (defined in Definition~\ref{d:Mabuchi} on the space $\cK_{\T}(X, \bom)$) to a  functional on the space $\E^{1}_{\T}(X, \bom)$. Our starting point is that the weighted Mabuchi energy $\cE_{v, w}$ admits a weighted Chen--Tian decomposition \cite[Thm.~5]{lahdili2} into energy and entropy parts as follows:
\begin{align}
\begin{split}\label{Chen-Tian}
\cE_{v, w}(\varphi)=\int_X\log\left(\frac{v(\m_\varphi)\omega^{m}_\varphi}{\omega_0^{m}}\right) v(\m_\varphi)\omega^{[m]}_\varphi &-2\ce^{\rho_{\bom}}_{v}(\varphi)+\cI_w(\varphi) \\
&- \int_X \log(v(\m_0))v(\m_0)\omega_0^{[m]},
\end{split}
\end{align}
where  $\rho_{\bom}$ is the  Ricci form of $\bom$ and the functionals $\cI_w$ and $\ce^{\rho_{\bom}}_{v}$ are introduced in Definition~\ref{d:I,J,v} below.   We want to show the following
\begin{Theorem}\label{M-extension}
For smooth weight functions $v(\mu)$, $w(\mu)$ such that $v(\mu)>0$ on $\Pol$, the weighted Mabuchi energy $\cE_{v, w}:\cK_\T(X,\bom)\to\R$ extends using \eqref{Chen-Tian} to the largest $d_1$ lsc functional $\cE_{v, w}:\E^{1}_\T(X,\bom)\to\R\cup \{\infty\}$ which is convex along the finite energy geodesics of $\E_\T(X,\bom)$. Additionally, the extended weighted Mabuchi energy $\cE_{v, w}$ is linear in $v,w$, uniformly continuous in $w$ in the $C^{0}(\Pol)$ topology and continuous with respect to $v$ in the $C^{1}(\Pol)$ topology.
\end{Theorem}
The above result is well-known for the unweighted case, by the work~\cite{BDL0}, and we will follow a similar path to get an extension in the weighted case. The proof of  Theorem~\ref{M-extension} will be given at the end of the section,  and we detail below the definition and extension of each component of \eqref{Chen-Tian}.

\subsection{The weighted Aubin--Mabuchi functionals}
\begin{defn}\cite{lahdili2}\label{d:I,J,v} For a smooth weight function $v(\x)$ on $\Pol$, we let ${\cI}_{v}$ denote the functional on $\cK_{\T}(X, \bom)$,  defined by
\[ (d_{\varphi} {\cI}_v)(\dot \varphi) = \int_X \dot \varphi  v(\m_{\varphi})\omega_{\varphi}^{[n]}, \qquad {\cI}_v(0)=0, \]
and let ${\cJ}_{v}:= \int_X \varphi v(\m_{0})\omega_0^{[m]} - \cI_v(\varphi)$. Furthermore, for a fixed $\T$-invariant closed $(1,1)$-form $\rho$ on $X$ with momentum $\mu_\rho: X\to\tor^{*}$,  we define the  \emph{$\rho$-twisted Aubin--Mabuchi functional} $\ce^{\rho}_v:\cK_\T (X,\bom)\to\R$ by 
\begin{equation*}\label{Energi-chi}
(d_\varphi\ce^{\rho}_{v})(\dot{\varphi}):=\int_X \dot{\varphi}\left(v(\m_\varphi)\rho\wedge\omega_{\varphi}^{[m-1]}+\langle (dv)(\m_\varphi),\mu_\rho\rangle\omega_{\varphi}^{[m]} \right),\quad \ce^{\rho}_{v}(0)=0.
\end{equation*}
For $v\equiv 1$,  we let $\cI_{1}=\cI$, $\cJ_1=\cJ$ and $\ce^{\rho}_v = \ce^{\rho}$, and notice that $\cI, \cJ$ are the functionals introduced in Definition~\ref{d:I, I-Aubin, J}
\end{defn}
\begin{rem}\label{r:I,J,v} It  follows from the above definition and the results in \cite{lahdili2} that for any weight $v(x)$ and a constant $c$, $\J_v(\varphi+c) = \cJ_v(\varphi)$,  allowing one to  see  $\cJ_v$ as a functional on the space of $\T$-invariant K\"ahler metrics in the K\"ahler class $\alpha = [\bom]$, and motivates the notation ${\cJ}_v(\omega_{\varphi})$.  Notice also that  $\cI_v$, $\cJ_v$  $\ce^{\rho}_v$ are linear in $v$. In the case when $v>0$, $\cJ_v$ is non-negative (see Lemma~\ref{Han-Li} below),   whereas $\cI_{v}$ is monotone in the sense that   for any 
$\varphi_0, \varphi_1 \in \cK_{\T}(X, \bom)$ with  $\varphi_1(x) \ge  \varphi_0(x)$
\[\cI_{v}(\varphi_1) -\cI_v(\varphi_0) \ge \inf_{\Pol}(v)  \int_X(\varphi_1-\varphi_0)\omega_{\varphi_0}^{[m]}. \]
The above inequality follows by Definition~\ref{d:I,J,v} and integrating the derivative of $\cI_v$ along the path $t\varphi_1 + (1-t)\varphi_0 \in \cK_{\T}(X, \bom)$ and integrating by parts.
\end{rem}
The following is established in \cite[(2.37)]{HL}.
\begin{lemma}\label{Han-Li} Let $v>0$. There exists a uniform constant $C=C(X, \bom, v)>0$ such that
\[\frac{1}{C} \cJ(\varphi) \le \cJ_v(\varphi) \le C\cJ(\varphi). \]
\end{lemma}
\begin{proof} 
Let $\varphi_t := \varphi_0+t\varphi$ with $\varphi:=\varphi_1-\varphi_0$ and $\omega_{\varphi_{t}} = \omega_{\varphi_0}  +tdd^{c}\varphi$, $t\in [0, 1]$. We compute
\begin{align*}
\cJ_v(\varphi)=& \cJ_v(\varphi_1)-\cJ_v(\varphi_0)\\
=&\int^{1}_0\int_X \varphi \left(
v(\m_{\omega})\omega^{[m]}- v(\m_{\varphi_s})\omega^{[m]}_{\varphi_s}\right)ds\\
=&-\int^{1}_0\int_X\varphi \left(\int_0^{s}  \frac{d}{dt}[v(\m_{\varphi_t})\omega^{[m]}_{\varphi_t}]dt\right) ds\\
=& -\int^{1}_0\int_X\varphi \left(\int_0^{s}  \left( g_{\varphi_t}(d[\log\circ v(\m_{\varphi_t})], d\varphi) -\Delta_{\varphi_t}(\varphi)\right) v(\m_{\varphi_t})\omega^{[m]}_{\varphi_t}dt\right) ds\\
=& -\int^{1}_0\int_0^{s}\left( \int_X\varphi d[ v(\m_{\varphi_t})]\wedge d^{c}\varphi\wedge \omega^{[m-1]}_{\varphi_t} +\varphi dd^{c}\varphi \wedge v(\m_{\varphi_t})\omega^{[m-1]}_{\varphi_t}\right) dtds\\
=& \int^{1}_0\int_0^{s}\left( \int_X v(\m_{\varphi_t}) d\varphi\wedge d^{c}\varphi\wedge \omega^{[m-1]}_{\varphi_t}\right) dtds\\
=&\int^{1}_0\int_0^{s}\left( \int_X v(\m_{\varphi_t}) d\varphi\wedge d^{c}\varphi\wedge (t\omega_\varphi+(1-t)\omega )^{[m-1]}\right) dtds\\
=&\sum_{j=0}^{m-1}\int^{1}_0\int_0^{s}\left( \int_X t^{j}(1-t)^{m-j-1}v(\m_{\varphi_t}) d\varphi\wedge d^{c}\varphi\wedge\omega^{[j]}\wedge\omega_\varphi^{^{[m-j-1]}}\right) dtds
\end{align*}
where, in the fourth equality,  we have used that 
\begin{equation}\label{v-variation} 
\frac{d}{dt}[v(\m_{\varphi_t})]=\sum_{i=1}^{\dtor}v_{,i}(\m_{\varphi_t})(d^{c}\varphi)(\xi_i)=g_{\varphi_t}(d[ v(\m_{\varphi_t})], d\varphi) \end{equation}
for any basis  $(\xi_i)_{i=1,\cdots,\dtor}$ of $\tor$. It follows that
\[\frac{1}{C} \cJ(\varphi) \le \cJ_v(\varphi) \le C \cJ(\varphi),\]
where $C=C(X,\alpha,v)$ is a constant such that $\frac{1}{C}\leq v\leq C$ on $\Pol_\alpha$. \end{proof}

\begin{lemma}\label{l:I,J-uniform} Suppose $v,w$ are smooth functions on $\Pol$.    Then  
\[
\begin{split}
\Big| \cJ_{v}(\varphi) - \cJ_w(\varphi) \Big|  & \le ||v-w||_{C^0(\Pol)} \cJ_1(\varphi); \\
\Big|\cI_v(\varphi) - \cI_w(\varphi)\Big| & \le ||v -w||_{C^0(\Pol)}\Big( ||\varphi ||_{L^1(X, \bom)} + \cJ_1(\varphi)\Big).
  \end{split} \]
In particular, for a fixed $\varphi\in \cK_{\T}(X, \bom)$, $\cI_{v}(\varphi)$ and $\cJ_{v}(\varphi)$ are uniform continuous in $v$.
\end{lemma}
\begin{proof} The first relation follows from Lemma~\ref{Han-Li} above whereas the second inequality follows from the first and Definition~\ref{d:I,J,v}.
\end{proof}

\begin{lemma}\label{l:I_p} The restrictions of ${\cI}_1^{Y}$,  $\cJ_1^Y$  to the subspace ${\cK}_{\T}(X, \bom)\subset {\cK}_{\T}(Y, {\tilde \omega_0})$ are respectively equal to  $C{\cI}^X_{\p}$  and  $C\cJ^X_{\p}$, where $\p(\x)$ is the weight function defined in Lemma~\ref{l:Y-Scal} and $C=\V(B, \omega_B)$. Furthermore, if $\tilde{\rho}$ is a  K\"ahler form on $Y$,  induced by a K\"ahler form $\rho$ on $X$ using \eqref{Y-Kahler}, then the restriction of $(\ce_1^{\tilde \rho})^Y$
to the subspace ${\cK}_{\T}(X, \bom)$ equals $C(\ce_p^{\rho})^X$.
\end{lemma}
\begin{proof} The first part follows from the definition of $\cI^Y_1$, using that ${\tilde \omega}_{\varphi}^{[n+m]}\wedge \theta^{\wedge r} =\p(\m_{\varphi})\omega_{\varphi}^{[m]} \wedge \omega_B^{[n]}\wedge \theta^{\wedge r}$ on $Z$.

Similarly, if $\tilde \rho$ is a $(1,1)$-form $Y$ whose pull-back to $Z$ is 
\begin{equation}\label{tilde-rho}
\tilde{\rho}:=\rho+\sum_{a=1}^{k}(\langle p_a,\mu_\rho\rangle+c_a\rangle)\pi^{*}_B\omega_a+\langle d\mu_\rho\wedge\theta\rangle, \end{equation}
 we compute
 \begin{align*}
(d_\varphi\ce^{\rho}_{p})^X(\dot{\varphi})=&\int_X \dot{\varphi}\left[p(\m_\varphi)\rho\wedge\omega_{\varphi}^{[m-1]}+\langle (dp)(\m_\varphi),\mu_\rho \rangle\omega_{\varphi}^{[m]} \right]\\
=&\frac{1}{\vol(B,\omega_B)}\int_Y\dot{\varphi}\tilde{\rho}\wedge {\tilde \omega}_{\varphi}^{[n+m-1]} =\frac{1}{\vol(B,\omega_B)}(d_\varphi\ce^{\tilde{\rho}})^Y(\dot{\varphi}).
\end{align*}
The claim follows as $(\ce^{\rho}_{p})^X(0)=0=(\ce^{\tilde \rho})^Y(0)$. \end{proof}

\subsection{The weighted $d_1$-distance}
\begin{defn}\label{d:d_{1,v}} Let $v>0$ be a positive function on $\Pol$. For  $\varphi_0, \varphi_1 \in \cK_{\T}(X, \bom)$ we let
\[d_{1, v}(\varphi_0, \varphi_1) := \inf_{\psi(t)}\{L_{1, v}(\psi(t)) \, | \, \psi(t, x)\in C_{\T}^{\infty}([0, 1]\times X), \, \psi(t)\in \cK_{\T}(X, \bom)\}\]
where
\[L_{1, v}(\psi(t)):=\int_{0}^{1}\left(\int_X |\dot{\psi}(t)|v(\m_{\psi(t)})\omega_{\psi(t)}^{[m]}\right) dt. \]
For $v\equiv 1$ we have $d_{1, 1}=d_1$ where $d_1$ is the distance introduced in Section~\ref{s:coercivity}.
\end{defn}
\begin{lemma}\label{l:d_{1,v}} For any weight $v>0$, there exists  uniform constant $C=C(X, \bom, v)>0$ such that
\begin{equation}\label{d_1-equivalence}
\frac{1}{C} d_{1}(\varphi_0, \varphi_1)  \le d_{1, v}(\varphi_0, \varphi_1) \le C d_{1}(\varphi_0, \varphi_1), \qquad \forall \varphi_0, \varphi_1 \in \cK_{\T}(X, \bom), 
\end{equation}
where $d_1 := d_{1, 1}$ is the distance introduced in \cite{Da}.
In particular, $d_{1, v}$ is a distance on $\cK_{\T}(X, \bom)$ which is quasi-isometric with  $d_1$. 
\end{lemma}
\begin{proof} The relation \eqref{d_1-equivalence} follows from the fact that $v(\x)$ is positive and uniformly bounded on $\Pol$. This yields that $d_{1, v}$ is a distance, as $d_1$ is a distance according to \cite{Da}.  \end{proof}

\begin{lemma}\label{l:I-Lipschitz} For any smooth weight $v>0$ we have
\[ \left| \cI_{v}(\varphi_0) - \cI_{v}(\varphi_1) \right| \le d_{1,v}(\varphi_0, \varphi_1) \le Cd_1(\varphi_0, \varphi_1), \qquad \forall \varphi_0, \varphi_1 \in \cK_{\T}(X, \bom).\]
\end{lemma}
\begin{proof} For any smooth curve  $\varphi_t$ between $\varphi_1$ and $\varphi_2$, using Definition~\ref{d:I,J,v},  we have
\[\left| \cI_{v}(\varphi_0) - \cI_{v}(\varphi_1) \right| = \left| \int_{0}^1 (d_{\varphi_t}\cI_v)(\dot \varphi_t) dt \right| \le L_{1, v}(\varphi_t).\]
The claim follows from the above and Lemma~\ref{l:d_{1,v}}. \end{proof}

\subsection{Extensions to $\E_{\T}^1(X, \bom)$}

\begin{lemma}\label{l:I,J-extension} For any smooth weight $v$, the functionals $\cI_v$ and $\cJ_v$  continuously extend to the space $\E^1_{\T}(X, \bom)$. Furthermore,  for any $\psi\in \E^1_{\T}(X, \bom)$, the extended functionals are linear and uniform continuous in $v$, in the topology $C^0(\Pol)$.
\end{lemma}
\begin{proof} $\cI_{v}$ is $d_1$-Lipschitz by  Lemma~\ref{l:I-Lipschitz};  for $\cJ_v$ we get from Definition~\ref{d:I,J,v}
\[ 
\Big| \cJ_{v}(\varphi_0) - \cJ_v(\varphi_1) \Big| \le \int_X |\varphi_0-\varphi_1|\omega_0^{[m]} + \Big| \cI_{v}(\varphi_0) - \cI_v(\varphi_1) \Big|. \]
Combining the above inequality with  Lemma~\ref{l:I-Lipschitz} and  \cite[Cor.~5.7]{Da},  there exists a uniform positive constant $C=C(X, \bom, v)$ and, for any fixed positive real number $R>0$,  an increasing continuous function $F_R: \R_+ \to \R_+, F(0)=0$, 
defined in terms of  $(X, \bom, R)$, such that for any $\varphi_0, \varphi_1 \in \cK_{\T}(X, \bom)$ with  $d_1(0, \varphi_i) \le R$, we have 
\[\Big| \cJ_{v}(\varphi_0) - \cJ_v(\varphi_1) \Big| \le C d_1(\varphi_0, \varphi_1) + F_R(d_1(\varphi_0, \varphi_1)), \]
showing that $\cJ_v$ is locally uniform continuous on $(\cK_{\T}(X, \bom), d_1)$ and thus extends continuously to $(\E^1_{\T}(X, \bom), d_1)$.

The $v$-linearity of $\cI_{v}$  and $\cJ_v$ is clear by continuity, see Remark~\ref{r:I,J,v}.   The continuity with respect to $v$ follows from the  continuous extensions of the inequalities in Lemma~\ref{l:I,J-uniform}, noting that we have already shown that $\cJ_{v}, \cJ_{w}, \cJ,  \cI_{v}, \cI_{w}$ all extend continuously,  whereas $||\cdot||_{L^1(X, \bom)}$ extends continuously by  \cite[Thm.~5.8]{Da}.  \end{proof}

\begin{cor}\label{K0-completion} The metric completion of $(\cK_\T(X,\bom)\cap\cI_v^{-1}(0),d_1)$ is the complete geodesic metric space $(\E^{1}_\T(X,\bom)\cap \cI_v^{-1}(0),d_1)$. \end{cor} \begin{proof} Similarly to \cite[Lemma 5.2]{DaRu}, one can show that  $\cI_v$ is linear along finite energy geodesics.  As $\cI_v:\E^{1}_\T(X,\bom)\to\R$ is $d_1$-continuous, it follows that $\E^{1}_\T(X,\bom)\cap \cI_v^{-1}(0)$ is a $d_1$-closed subspace. \end{proof}

\begin{lemma}\label{ext-E^rho}
Let $v$ be a smooth weight function and $\rho$ a $\T$-invariant closed $(1,1)$-form. The functional $\ce^{\rho}_v:\cK_\T (X,\bom)\to\R$ extends to a $d_1$-continuous functional on $\E^{1}_\T(X,\bom)$,  which is bounded on $d_1$-bounded subsets of $\E^{1}_\T(X,\bom)$. Furthermore, the extended functional is linear and uniformly continuous in $v$, in the $C^{1}(\Pol)$ topology.
\end{lemma}

\begin{proof}
Following the proof of \cite[Prop.~4.4]{BDL}, we show that $\ce^{\rho}_v$ is locally uniformly $d_1$-continuous and bounded on $d_1$-bounded subsets of $\cK_\T(X,\bom)$.  Let $\varphi_{0},\varphi_1\in\cK_\T(X,\bom)$, we put $\varphi_s:=s\varphi_1+(1-s)\varphi_0$, $s\in[0,1]$ and compute
\begin{align}
\begin{split}\label{E1-E0}
\ce_{v}^{\rho}(\varphi_1)-&\ce_{v}^{\rho}(\varphi_0)=\int^{1}_0\frac{d}{ds}\ce_{v}^{\rho}(\varphi_s)ds\\
=&\int^{1}_0\int_X(\varphi_1-\varphi_0)\left(v(\m_{\varphi_s})\rho\wedge\omega_{\varphi_s}^{[m-1]}+\langle (dv)(\m_{\varphi_s}),\mu_\rho\rangle\omega_{\varphi_s}^{[m]} \right)ds\\
=&\int_X(\varphi_1-\varphi_0)\sum_{j=0}^{m-1}v_{j,m-1}(\mu_{\varphi_0},\mu_{\varphi_1})\rho\wedge\omega_{\varphi_1}^{[j]}\wedge\omega_{\varphi_0}^{[m-j-1]}\\
&+\int_X(\varphi_1-\varphi_0)\sum_{j=0}^{m} \langle (dv)_{j,m}(\m_{\varphi_0},\m_{\varphi_1}),\m_\rho\rangle\omega_{\varphi_1}^{[j]}\wedge\omega_{\varphi_0}^{[m-j]}
\end{split}
\end{align}
where $v_{j,k}(\m_0,\m_1)$, $(dv)_{j,k}(\m_{0},\m_{1})$ are defined on $\Pol\times\Pol$ by
\[\begin{split}
v_{j,k}(\m_0,\m_1) :=&\int_{0}^{1} s^{j}(1-s)^{k-j} v(s\m_1+(1-s)\m_0),\\
(dv)_{j,k}(\m_{0},\m_{1})=&\int_{0}^{1} s^{j}(1-s)^{k-j} (dv)(s\m_1+(1-s)\m_0).
\end{split}
\] 
Using the computation \eqref{E1-E0}, we get
\begin{align}
\begin{split}\label{A(E1-E0)}
|\ce_{v}^{\rho}(\varphi_1)-\ce_{v}^{\rho}(\varphi_0)|\leq & C\int_X|\varphi_1-\varphi_0|\sum_{j=0}^{m-1}\bom\wedge\omega_{\varphi_1}^{[j]}\wedge\omega_{\varphi_0}^{[m-j-1]}\\
&+C\int_X|\varphi_1-\varphi_0|\sum_{j=0}^{m} \omega_{\varphi_1}^{[j]}\wedge\omega_{\varphi_0}^{[m-j]}\\
\leq & C\int_X|\varphi_1-\varphi_0|\omega_{\frac{\varphi_0+\varphi_1}{4}}^{[m]} 
\end{split}
\end{align}
in the first inequality we use that the functions $\langle (dv)_{j,k}(\m_{\varphi_0},\m_{\varphi_1}),\mu_\rho\rangle$ and $v_{j,k}(\m_{\varphi_0},\m_{\varphi_1})$ are bounded on $\Pol\times \Pol$ and $-C\bom<\rho<C\bom$ for some constant $C>1$ and in the second inequality we use the observation $\omega_{\frac{\varphi_0+\varphi_1}{4}}=\bom/2+\omega_{\varphi_0}/4+\omega_{\varphi_1}/4$. Using the estimate \eqref{A(E1-E0)} we can show, similarly to \cite[Prop.~4.4]{BDL}, that for any $R>0$ there is an increasing continuous function $F_R:\R\to\R$ with $F_R(0)=0$ such that
\[|\ce_{v}^{\rho}(\varphi_1)-\ce_{v}^{\rho}(\varphi_0)|\leq F_R(d_1(\varphi_0,\varphi_1))\]
for any $\varphi_0,\varphi_1\in\cK_\T(X,\bom)\cap\{\varphi,\,\,d_1(0,\varphi)<R\}$. It follows that $\ce^{\rho}_v$ extends to a $d_1$-continuous functional on $\E^{1}_\T(X,\bom)$ which is bounded on $d_1$-bounded subsets of $\E^{1}_\T(X,\bom)$. 

For the last statement, let $v,w$ be two (smooth) positive weight functions and $\varphi\in\cK_\T(X,\bom)$. Taking $\varphi_1=\varphi$ and $\varphi_0=0$ in the computation \eqref{E1-E0} 
\[
\begin{split}
\ce_{v}^{\rho}(\varphi)=&\int_X\varphi\sum_{j=0}^{m-1}v_{j,m-1}(\mu_{0},\mu_{\varphi})\rho\wedge\omega_{\varphi}^{[j]}\wedge\omega_0^{[m-j-1]}\\
&+\int_X\varphi\sum_{j=0}^{m} \langle (dv)_{j,m}(\m_{0},\m_{\varphi}),\mu_{\rho}\rangle\omega_{\varphi}^{[j]}\wedge\omega_0^{[m-j]}
\end{split}
\]
Let $C>1$ such that $-C\bom<\rho<C\bom$, using the above formula we obtain
\[
\begin{split}
|\ce_{v}^{\rho}(\varphi)-\ce_{w}^{\rho}(\varphi)|=&|\ce_{v-w}^{\rho}(\varphi)|\\
\leq&C\int_X|\varphi|\sum_{j=0}^{m-1}|(v-w)_{j,m-1}(\mu_{0},\mu_{\varphi})|\omega_{\varphi}^{[j]}\wedge\omega_0^{[m-j]}\\
&+C\int_X|\varphi|\sum_{j=0}^{m} |\langle (d(v-w))_{j,m}(\m_{0},\m_{\varphi}),\mu_{\rho}\rangle|\omega_{\varphi}^{[j]}\wedge\omega_0^{[m-j]}\\
\leq& C\parallel v-w\parallel_{C^{1}(\Pol)}\int_X\sum_{j=0}^{m}|\varphi|\omega_{\varphi}^{[j]}\wedge\omega_0^{[m-j]}\\
\leq&C\parallel v-w\parallel_{C^{1}(\Pol)} \int_X|\varphi|(2\bom+dd^{c}\varphi)^{[m]}\\
\leq&C\parallel v-w\parallel_{C^{1}(\Pol)}\int_X|\varphi|\omega_\varphi^{[m]}.
\end{split}
\]
Using approximation by decreasing sequences in $\cK_\T(X,\bom)$, the above estimate holds for $\E^{1}_\T(X,\bom)$.
\end{proof}

Following Berman--Witt-Nystr\"om \cite{B-N} and the recent work of Han--Li \cite{HL},  we now define the  extension of weighted Monge--Ampère measures to the space $\E_{\T}(X, \bom)$.
\begin{prop}~\label{weight-MA}
Let $v>0$ be a smooth weight function. For any $\varphi\in \cK_{\T}(X,\bom)$,  we let
\begin{equation*}
\MA_v(\varphi):=v(\m_\varphi)\omega^{[m]}_\varphi.
\end{equation*}
Then $\MA_v(\varphi)$ extends to a well-defined Radon measure defined  for any $\varphi\in \E_{\T}(X, \bom)$, such that,  for any decreasing sequence $(\varphi_j)_j$ of elements in $\cK_{\T}(X,\bom)$ converging to $\varphi$ (which exists by \cite{BK2007}),  we have  $\lim_{j\to \infty} \MA_v(\varphi_j) = \MA_v(\varphi)$.
\end{prop}
\begin{proof}
The result is established in \cite{B-N, HL} for $\bom\in \alpha=c_1(L)$ a K\"ahler Hodge class on a projective variety $X$. The method of Han--Li \cite[Prop.~2.2]{HL}, which uses the  semi-simple principle fibration construction  and polynomial approximations,  extends to the case of arbitrary K\"ahler class $\alpha =[\bom]$.   Below we give details of this construction,  for Reader's convenience.

Let $\varphi\in \E_{\T}(X,\bom)$. Following the proof of \cite[Prop. 2.2]{HL}, we first define $\MA_p(\varphi)$ for a positive polynomial  weight of the form $p(\mu):=\prod_{a=1}^{k}(\langle p_a,\mu\rangle+c_a)^{n_a}$,  and extend the definition linearly on $p$ for finite sums of such polynomials.  We can then use the  Bernstein approximation theorem of an arbitrary positive $v$  with polynomials of the above form in order to obtain $\MA_v(\varphi)$. 

We start with a semi-simple principal $(X, \T)$-fibration $Y$ (see Section~\ref{s:geometric}) with corresponding polynomial weight $p(\mu):=\prod_{a=1}^{k}(\langle p_a,\mu\rangle+c_a)^{n_a}$ (see \eqref{p-weight}). As  the choice of the base $B= B_1 \times \cdots \times B_k$ does not matter, we can simply take (as in \cite{HL}) $B$ to be the product of projective spaces $(B_a, \omega_a)= (\PP^{n_a}, \omega_a)$ endowed with Fubini--Study metrics  of scalar curvatures $2n_a(n_a+1)$, and $P$ be the principal  ${\rm U}(1)^r$ bundle over $B$, obtained  from  the tensor products $P_i$ of  (the pull-backs to $B$ of) the natural principle ${\rm U}(1)$-bundles of degrees $p_{ai}$ over $\PP^{n_a}$ (see Remark~\ref{r:Chern-class}).

 Using \cite[Thm.~1]{BK2007}, there is a decreasing sequence $\varphi_{j}\in \PSH_{\T}(X,\bom)\cap C^{\infty}(X)=\mathcal{K}_{\T}(X,\bom)$ converging towards $\varphi$. By Lemma \ref{l:Y-embedding} we have $\varphi_{j}\in \mathcal{K}_{\T}(Y,\tilde \omega_0)$ and, by \eqref{measures},  for any $\T_X$-invariant continuous function $f$ on $X$ we have
\[
\int_X f  p(\m_{\varphi_{j}})(\bom+d_Xd^{c}_X\varphi_{j})^{[m]}=\frac{1}{\vol(B,\omega_B)}\int_Y f (\tilde \omega_0+d_Yd^{c}_Y\varphi_{j})^{[m+n]}.\]
Passing to the limit in both sides of the above equation, we can define $\MA^{X}_p(\varphi)$ on $\T$-invariant continuous functions $f$ by
\begin{equation}\label{MA-f-T}
\int_X f \MA^{X}_p(\varphi):=\underset{j\to\infty}{\lim}\frac{1}{\vol(B,\omega_B)}\int_Y f (\tilde{\omega}_0+d_Yd^{c}_Y\varphi_{j})^{[m+n]}
\end{equation}
Notice that by \cite[Thm.~1.9]{GZ} the limit  exists and is well-defined on $Y$ (independent of the chosen sequence).

For a continuous function $f$ on $X$ which is not necessarily $\T_X$-invariant, we define
\[\int_X f \MA^{X}_p(\varphi):=\int_X f^{\T} \MA^{X}_p(\varphi)\]
where $f^{\T}$ is the $\T_X$-invariant function given by average of $f$ over the $\T_X$-action. It follows that $\MA^{X}_p(\varphi)$ is a well-defined Radon measure by Riesz representation theorem. 

We can extend the above definition  by linearity in $p$ on polynomials  which are   linear combinations with positive coefficients of polynomials of the above special form:  thus, for $\varphi\in\PSH_{\T}(X,\bom)$ and for two polynomials $p,q$ on $\Pol$, we will have
\begin{equation}\label{MA-uniform}
\left|\int_X f\MA^{X}_p(\varphi)-\int_X f\MA^{X}_q(\varphi)\right|\leq \parallel p-q\parallel_{C^{0}(\Pol)}\int_X |f|\MA^{X}(\varphi)\end{equation}
for any $f\in C^{0}(X)$.

For an arbitrary smooth positive function $v$ on $\Pol$,  we can approximate $v$ in $C^0(\Pol)$ by polynomials $p_i$ as above (e.g. by using Bernstein's Approximation Theorem)  and thus,  for any continuous function $f$,  the limit
 \[\underset{i\to\infty}{\lim}\underset{j\to\infty}{\lim}\int_X f\MA^{X}_{p_i}(\varphi_{j})\] 
 exists  independently of the chosen approximation.   We then define
\[\int_X f\MA^{X}_v(\varphi):=\underset{i\to\infty}{\lim}\underset{j\to\infty}{\lim}\int_X f\MA^{X}_{p_i}(\varphi_{j}).\]
By the Riesz representation theorem,  $\MA^{X}_v(\varphi)$ is a well-defined Radon measure. 
\end{proof}

\begin{rem}
Notice that for any $\varphi \in \E_\T(X,\bom)$,  the measure $\MA_v(\varphi)$ is absolutely continuous with respect to $\MA(\varphi)$ since $v$ is bounded on $\Pol$. In particular, for any positive weight $v$, we have that 
\[\E^1_{\T}(X, \bom) = \left\{ \varphi \in \E_{\T}(X, \bom) \, \Big| \, \int_X |\varphi | \MA_v(\varphi) < \infty \right\}. \]
\end{rem}

\begin{lemma}\label{lim-MA-E1}
Let $v$ be a positive weight function and $\varphi_j,\varphi\in\E_{\T}^1(X, \bom)$ such that $d_1(\varphi_j,\varphi)\rightarrow 0$. Then, $\MA_v(\varphi_j)\rightarrow\MA_v(\varphi)$ weakly.
\end{lemma}
\begin{proof}
Let $v(\mu)$ be a polynomial of the form $p(\mu):=\prod_{a=1}^{k}(\langle p_a,\mu\rangle+c_a)^{n_a}$, $\varphi_j\in \cK_{\T}(X, \bom)$, and $f$ any continuous $\T$-invariant function on $X$.  We then have by the construction in Section~\ref{s:geometric}
\[
\int_X f  p(\m_{\varphi_{j}})(\bom+d_Xd^{c}_X\varphi_{j})^{[m]}=\frac{1}{\vol(B,\omega_B)}\int_Y f (\tilde{\omega}_0+d_Yd^{c}_Y\varphi_{j})^{[m+n]}.\]
It follows that for each $\varphi_j\in \E_{\T}^1(X, \bom)$ (using an approximation with a decreasing sequence of smooth potentials \cite{BK2007}), we have
\[
\int_X f  \MA^{X}_p(\varphi_j)=\frac{1}{\vol(B,\omega_B)}\int_Y f \MA^{Y}(\varphi_j).\]
By \cite[Thm.~5]{Da},  $\MA^{Y}(\varphi_j)\to\MA^{Y}(\varphi)$ weakly as $j\to \infty$. It follows that
\[
\underset{j\to \infty}{\lim}\int_X f  \MA^{X}_p(\varphi_j)=\frac{1}{\vol(B,\omega_B)}\int_Y f \MA^{Y}(\varphi)=\int_X f  \MA^{X}_p(\varphi).\]
Using \eqref{MA-f-T}, we conclude that $\MA^{X}_p(\varphi_j)\to \MA^{X}_p(\varphi)$ weakly as $j\to \infty$.

For an arbitrary weight function $v\in C^{0}(\Pol)$,  we take a sequence of polynomials $p_{i}$ of the above form  converging to $v$ in $C^{0}(\Pol)$. For any continuous function $f$ on $X$, using \eqref{MA-uniform},  we have
\[\begin{split}
&\left|\int_X f  \MA_{v}(\varphi_j)-\int_X f  \MA_{v}(\varphi)\right|\\
& \leq \left|\int_X f  \MA_{v}(\varphi_j)-\int_X f  \MA_{p_i}(\varphi_j)\right|+\left|\int_X f  \MA_{p_i}(\varphi_j)-\int_X f  \MA_{p_i}(\varphi)\right|\\
& \, \, \, \, +\left|\int_X f  \MA_{p_i}(\varphi)-\int_X f  \MA_{v}(\varphi)\right|\\
&\leq \left|\int_X f  \MA_{p_i}(\varphi_j)-\int_X f  \MA_{p_i}(\varphi)\right|+\parallel p_i-v\parallel_{C^{0}(\Pol)}\left(\int_X|f|\MA(\varphi_j)+\int_X|f|\MA(\varphi)\right).
\end{split}
\]
Letting $j\to\infty$, we get
\[
\underset{j\to\infty}{\lim}\left|\int_X f  \MA_{v}(\varphi_j)-\int_X f  \MA_{v}(\varphi)\right|\leq  2\parallel p_i-v\parallel_{C^{0}(\Pol)}\int_X|f|\MA(\varphi)
\]
where we used the existence of the weak limits $\MA_{p_i}(\varphi_j)\to \MA_{p_i}(\varphi)$ and $\MA(\varphi_j)\to\MA(\varphi)$ as $j\to\infty$ (by \cite[Thm.~5]{Da}). Taking the limit  $i\to\infty$ in the above inequality, we obtain
\[
\underset{j\to\infty}{\lim}\left|\int_X f  \MA_{v}(\varphi_j)-\int_X f  \MA_{v}(\varphi)\right|=0.
\]
It follows that $\MA_{v}(\varphi_j)\to \MA_{v}(\varphi)$ weakly as $j\to\infty$.
\end{proof}

For a finite measure $\chi$ on $X$ we define the entropy of $\chi$ with respect to $\omega^{[m]}$ by
\[
\Ent(\omega^{[m]},\chi):=\int_X \log\left(\frac{\chi}{\omega^{[m]}}\right)\chi.
\]

In the following lemma we show that the elements of $\E^{1}_\T(X,\bom)$ can be approximated in the $d_1$ distance by smooth potentials with converging entropy of the corresponding weighted Monge--Amp\`ere measures. This is the weighted analogue of \cite[Lemma~3.1]{BDL}.

\begin{lemma}\label{Ent-approx}
If $v>0$,  $\E^{1}_\T(X,\bom) \ni\varphi\mapsto {\rm Ent}(\omega_0^{[m]},\MA_v(\varphi))$ is $d_1$ lsc. Furthermore, for any $\varphi\in \E^{1}_\T(X,\bom)$ there exist a sequence of smooth potentials $\varphi_j\in\cK_\T(X,\bom)$ such that $d_1(\varphi_j,\varphi)\to 0$ and ${\rm Ent}(\omega_0^{[m]},\MA_v(\varphi_j))\to {\rm Ent}(\omega_0^{[m]},\MA_v(\varphi))$ as $j\to \infty$.
\end{lemma}
\begin{proof}
The proof follows closely the arguments of \cite[Lemma 3.1]{BDL}. By Lemma \ref{lim-MA-E1} and the fact that the entropy $\chi\mapsto \Ent(\omega_0^{m},\chi)$ is lsc on the space of finite measures,  with respect to the weak convergence of measures, (cf. \cite[Prop. 3.1]{BB}),  it follows that the entropy $\varphi\mapsto {\rm Ent}(\omega_0^{[m]},\MA_v(\varphi))$ is $d_1$ lsc. Let $\varphi\in \E^{1}_\T(X,\bom)$. If $\Ent(\omega_0^{[m]},\MA_v(\varphi))=\infty$ then any sequence $\varphi_j\in\cK_\T(X,\bom)$ such that $d_1(\varphi_j,\varphi)\to 0$ satisfies ${\rm Ent}(\omega_0^{[m]},\MA_v(\varphi))\to \infty$ as $j\to \infty$. We suppose that ${\rm Ent}(\omega_0^{[m]},\MA_v(\varphi))<\infty$ and we put $g:=\frac{\MA_v(\varphi)}{\omega_0^{[m]}}\geq 0$ the density function of the measure $\MA_v(\varphi)$. From the proof of \cite[Lemma 3.1]{BDL}, there exist a sequence of positive functions $g_j\in C_{\T}^{\infty}(X)$ such that $\parallel g-g_j\parallel_{L^{1}}\to 0$ and
\[
\int_X g_j\log g_j\, \omega_0^{[m]}\to {\rm Ent}(\omega_0^{[m]},\MA_v(\varphi)).
\]
Using \cite[Prop. 3.7]{HL}, we can find a smooth potential $\varphi_j\in \cK_{\T}(X,\bom)$ (which is unique up to adding a constant) such that $\MA_v(\varphi_j)=\left(\frac{\int_X v(\mu_0)\omega_0^{m}}{\int_X g_j\omega_0^{m} } \right) g_j\omega_0^{[m]}$. By \cite[Lemma 2.16]{HL}, up to a  passing to a subsequence of $\varphi_j$, there exists a $\psi\in \E_{\T}^{1}(X,\bom)$ such that $d_1(\psi,\varphi_j)\to 0$. Lemma \ref{lim-MA-E1} together with  $\parallel g-g_j\parallel_{L^{1}}\to 0$ gives $\MA_v(\psi)=\underset{j\to \infty}{\lim} \MA_v(\varphi_j)=\MA_v(\varphi)$. It follows that $\varphi=\psi$ (up to a constant) by \cite[Thm. 2.18]{B-N}. Thus, $d_1(\varphi,\varphi_j)\to 0$ and ${\rm Ent}(\omega_0^{[m]},\MA_v(\varphi_j))\to {\rm Ent}(\omega_0^{[m]},\MA_v(\varphi))$ as $j\to \infty$.
\end{proof}

Now we are in position to prove Theorem~\ref{M-extension}

\begin{proof}[\bf Proof of Theorem~\ref{M-extension}]
By Lemmas \ref{l:I,J-extension} and \ref{ext-E^rho} the functionals $\cI_{w}$ and $\ce_{v}^{\rho_\bom}$ extend as  continuous functionals on $\E_\T^{1}(X,\bom)$.  On the other hand,  the entropy $\varphi\mapsto {\rm Ent}(\omega_0^{m},\MA_v(\varphi))$ is $d_1$ lsc by Lemma~\ref{Ent-approx}).  Thus, the weighted Chen--Tian decomposition \eqref{Chen-Tian} gives rise to an extension of  the $(v,w)$-Mabuchi energy to a $d_1$ lsc functional $\cE_{v, w}:\E^{1}_\T(X,\bom)\to\R\cup \{\infty\}$. Notice that (using the continuity of $\cI_{w}$ and  $\ce_{v}^{\rho_\bom}$)  the restriction of $\cE_{v, w}:\E^{1}_\T(X,\bom)\to\R\cup \{\infty\}$ on the subspace $\cK_{\T}^{1,\bar{1}}(X,\bom)$ is equal to the weighted $(v,w)$-Mabuchi energy on that space defined in \cite[Cor.~3]{lahdili3}. By Lemma \ref{Ent-approx}, for $\varphi\in \E^{1}_\T(X,\bom)$  we can find a sequence $\varphi_j\in \cK_\T(X,\bom)$, such that $d_1(\varphi_j,\varphi)\to 0$ and 
\[\underset{j\to\infty}{\lim} \cE_{v,w}(\varphi_{j})=\cE_{v,w}(\varphi).\]
It follows that the extension $\cE_{v, w}:\E^{1}_\T(X,\bom)\to\R\cup \{\infty\}$ using \eqref{Chen-Tian} is the largest $d_1$ lsc extension of $\cE_{v, w}:\cK_\T(X,\bom)\to\R$.

\smallskip
We now show  that $t\mapsto\cE_{v, w}(\varphi_t)$, $t\in [0,1]$ is convex and continuous along the finite energy geodesics $\varphi_t\in \E_\T(X,\bom)$. We follow closely the arguments of \cite[Thm. 4.7]{BDL}. Let $\varphi_t\in \E^{1}_\T(X,\bom)$, $t\in [0,1]$ be a finite energy geodesic. Suppose that $t_0,t_1\in [0,1]$ with $t_0\leq t_1$. Using Lemma \ref{Ent-approx}, we can find sequences $\varphi^j_{t_0},\varphi^j_{t_1}\in \cK_\T(X,\bom)$, such that $d_1(\varphi^j_{t_0},\varphi_{t_0})\to 0$ and $d_1(\varphi^j_{t_1},\varphi_{t_1})\to 0$ and 
\[\underset{j\to\infty}{\lim} \cE_{v,w}(\varphi^j_{t_0})=\cE_{v,w}(\varphi_{t_0}),\quad\underset{j\to\infty}{\lim} \cE_{v,w}(\varphi^j_{t_1})=\cE_{v,w}(\varphi_{t_1}).\]
Let $t\mapsto \varphi_{t}^{j}\in \cK_{\T}^{1,\bar{1}}(X,\bom)$, $t\in[t_0,t_1]$ the $C^{1,\bar{1}}$-weak geodesic segment connecting $\varphi^j_{t_0},\varphi^j_{t_1}$. By \cite[Thm.~5]{lahdili3}, the function $[t_0,t_1]\ni t\mapsto \cE_{v,w}(\varphi^{j}_t)$ is convex. Since $\cE_{v, w}:\E^{1}_\T(X,\bom)\to\R\cup \{\infty\}$ is $d_1$ lsc we have
\[
\begin{split}
\cE_{v, w}(\varphi_t)\leq \underset{j\to\infty}{\liminf}\,\,\cE_{v, w}(\varphi^{j}_t)
\leq&\left(\frac{t-t_0}{t_1-t_0}\right) \underset{j\to\infty}{\lim}\,\,\cE_{v, w}(\varphi^{j}_{t_0})+\left(\frac{t_1-t}{t_1-t_0}\right) \underset{j\to\infty}{\lim}\,\,\cE_{v, w}(\varphi^{j}_{t_1})\\
\leq&\left(\frac{t-t_0}{t_1-t_0}\right) \cE_{v, w}(\varphi_{t_0})+\left(\frac{t_1-t}{t_1-t_0}\right) \underset{j\to\infty}{\lim}\,\,\cE_{v, w}(\varphi_{t_1})
\end{split}
\]
where the second inequality uses the  convexity of  $ t\mapsto \cE_{v,w}(\varphi^{j}_t)$.
Thus, $ t\mapsto \cE_{v,w}(\varphi_t)$ is convex and continuous up to the boundary of $[t_0,t_1]$ since it is $d_1$ lsc. 

\smallskip
It remains to show that  $\cE_{v, w}:\E^{1}_\T(X,\bom)\to\R\cup \{\infty\}$ is linear and continuous in $v, w$. For smooth potentials $\varphi\in \cK_\T(X,\bom)$, we have
\begin{equation}\label{Ent-lin-v}
\Ent(\omega_0^{[m]},\MA_v(\varphi))- \int_X \log(v(\m_0))v(\m_0)\omega_0^{[m]}=\int_X\log\left(\frac{\MA(\varphi)}{\omega_0^{m}}\right)\MA_v(\varphi),
\end{equation}
which is manifestly linear in $v$. For $\varphi\in \E^{1}_\T(X,\bom)$, the above expression is  still linear in $v$ by Proposition~\ref{weight-MA}.
Substituting back in \eqref{Chen-Tian},  and using Lemma~\ref{l:I,J-extension} and \ref{ext-E^rho}, it follows that $\cE_{v, w}:\E^{1}_\T(X,\bom)\to\R\cup \{\infty\}$ is linear in $v,w$. From these two lemmas, we know that $\ce^{\rho}_v:\E^{1}_\T(X,\bom)\to\R$ and $\cI_w:\E^{1}_\T(X,\bom)\to\R$ are uniformly continuous in $v,w$.  For the remaining entropy part,  we notice that if  $\varphi\in \E^{1}_\T(X,\bom)$, $v,v^{\prime}\in C^{\infty}(\Pol)$ and $f\in C^{0}(X)$, then
\[
\left|\int_X f\MA_v(\varphi)-\int_X f\MA_{v^{\prime}}(\varphi)\right|\leq \parallel v-v^{\prime}\parallel_{C^{0}(\Pol)}\int_X|f|\MA(\varphi),
\]
which can be obtained again by approximating $\varphi$ with a monotone sequence of smooth relative potentials and using Proposition~\ref{weight-MA}. It follows that $C^{\infty}(\Pol) \times \E^{1}_\T(X,\bom)\ni(v, \varphi)\mapsto \MA_v(\varphi)$ is uniformly continuous with respect to $v$ for the weak topology on the space of measures. Since the entropy $\chi\mapsto \Ent(\omega_0^{m},\chi)$ is lsc on the space of finite measures with respect to the weak convergence of measures \cite[Prop. 3.1]{BB},  the term ${\rm Ent}(\omega_0^{[m]},\MA_v(\varphi))$ is lsc with respect to $v$. The linearity with respect to $v$ in the RHS of \eqref{Ent-lin-v} shows that ${\rm Ent}(\omega_0^{[m]},\MA_v(\varphi))$ is in fact continuous with respect to $v$.
\end{proof}

We derive the following weighted version of the key compactness result from \cite{BBEGZ,BDL0}:
\begin{Theorem}\label{compact}
Any sequence $\varphi_j\in \E^{1}_\T(X,\bom)$ such that
\[
d_1(0,\varphi_j)\leq C,\quad \cE_{v,w}(\varphi_j)\leq C
\]
 admits a $d_1$-convergent subsequence. 
\end{Theorem}
\begin{proof}
From the formula \eqref{Chen-Tian} and  Lemmas \ref{l:I-Lipschitz} and \ref{ext-E^rho},   we see that ${\rm Ent}(\omega_0^{[m]},\MA_v(\varphi))$ is uniformly bounded under the hypotheses in the Corollary. We conclude using \cite[Lemma 2.16]{HL}.
\end{proof}

\section{Regularity of the weak minimizers of the weighted Mabuchi energy}\label{s:regularity}
In this section, we establish the  regularity  of  the weak minimizers of $\cE_{v,w}$.
\begin{Theorem}\label{W&S} Suppose  $\T\subset \Aut_r(X)$ is  a maximal torus and $(X,\alpha,\T)$ admits a $(v,w)$-cscK metric $\omega$ with $w=\ell^{\ext}_{v,w_0}w_0$, where $v,w_0>0$ are two positive smooth weight functions on $\Pol$.  If $\psi\in\E^{1}_\T(X,\bom)$ is a minimizer of the extended $(v,w)$-Mabuchi energy $\cE_{v,w}:\E^{1}_\T(X,\bom)\to \R\cup \{\infty\}$, then $\psi\in \cK_{\T}(X, \bom)$ is a smooth potential. \end{Theorem}

The proof of this result, which is an adaptation of the arguments in \cite{BDL},  will occupy the reminder of the section.

\begin{defn}\label{twisted-weighted-minimizers}
Let $v(\mu)>0$, $w(\mu)$ be smooth weight functions on $\Pol$ and $\rho>0$ a $\T$-invariant K\"ahler form on $X$. We let
\begin{equation}
\cM_{v,w}:=\{\psi\in \E^{1}_\T(X,\bom)\cap \cI^{-1}(0)\mid \cE_{v,w}(\psi)=\underset{\varphi\in \E^{1}_\T}{\inf}\cE_{v,w}(\varphi)\}
\end{equation}
and $\cE_{v,w}^{\rho}:=\cE_{v,w}+\ce^{\rho}$, 
where $\ce^{\rho}$ is introduced via Lemma~\ref{ext-E^rho} and $v=1$. 
\end{defn}
By \cite[Lemma 5.2]{DaRu}  and Theorem~\ref{M-extension},   the set $\cM_{v,w}$ (when non-empty) is totally geodesic with respect to the finite energy geodesics of $\E^{1}_\T(X,\bom)$.  Furthermore, if there exists a $\psi_\rho\in \cM_{v,w}$ such that $\ce^{\rho}(\psi_\rho)=\underset{\psi\in \cM_{v,w}}{\inf} \ce^{\rho}(\psi)$,  then $\psi_\rho$ is unique by the strict convexity of $\ce^{\rho}$ established in \cite[Prop. 4.5]{BDL}. Furthermore, by  Theorem~\ref{M-extension}, the functional $\cE_{v,w}^{\rho}:\E^{1}_\T(X,\bom)\to \R\cup \{\infty\}$ will be also strictly convex along finite energy geodesics, showing the uniqueness of an element $\psi\in \E^{1}_\T(X,\bom)\cap \ce^{-1}(0)$ such that $\cE^{\rho}_{v,w}(\psi)=\underset{\varphi\in \E^{1}_\T}{\inf}\cE^{\rho}_{v,w}(\varphi)$ (assuming that such minimizer $\psi$ exists).

\bigskip
We then have the following weighted version of the continuity method of \cite[Prop.~3.1]{BDL}:

\begin{prop}\label{prop: BDL1}
Let $v>0$, $w$ be smooth weight functions on $\Pol$.  Suppose that $\cM_{v,w}$ is non-empty,  and $\varphi\in\cK_\T(X,\bom)\cap\cI^{-1}(0)$. For any $\lambda>0$, there exists a unique minimizer $\psi_{\lambda}\in \E^{1}_\T(X,\bom)\cap \cI^{-1}(0)$ of $\cE^{\lambda\omega_\varphi}_{v,w}:=\cE_{v,w}+\ce^{\lambda\omega_\varphi}$. The curve $[0,\infty)\ni\lambda\mapsto \psi_\lambda\in \E^{1}_\T(X,\bom)\cap \cI^{-1}(0)$ is $d_1$-continuous, $d_1$-bounded and  $\psi_0:=\underset{\lambda\to 0}{\lim}\,\psi_\lambda$ is the unique minimizer of $\ce^{\omega_\varphi}$ on $\cM_{v,w}$. Furthermore, for any $\psi\in \cM_{v,w}$ and $\lambda>0$,  we have
\begin{equation}\label{I_lambda<I}
I(\varphi,\psi_\lambda)\leq m(m+1)I(\varphi,\psi),
\end{equation}
where $I(\varphi,\psi):=\int_X(\varphi-\psi)\left(\omega^{m}_{\psi}-\omega^{m}_{\varphi}\right)$.
\end{prop}

\begin{proof}
The proof follows by a straightforward adaptation of the arguments in \cite[Prop.~3.1]{BDL}. 
\end{proof}
We next need a weighted analogue of \cite[Lemma 3.3]{BDL}.
\begin{lemma}\label{p:sub-slop} Let $v>0, w$ be smooth weight functions  on $\Pol$, and  $\rho>0$ a smooth $\T$-invariant K\"ahler form on $X$.  Let $\varphi_0\in\cK_\T(X,\bom)$, 
$\varphi_1\in \mathcal{E}^{1}_{\T}(X,\bom)$,  and $[0,1]\ni t\mapsto\varphi_t\in \mathcal{E}^{1}_{\T}(X,\bom)$ be a finite energy geodesic connecting $\varphi_0$ and $\varphi_1$. Then,
\[
\begin{split}
\underset{t\to 0^{+}}{\lim}\frac{\cE_{v,w}^{\rho}(\varphi_t)-\cE_{v,w}^{\rho}(\varphi_0)}{t}\geq &\int_X\left(w(\m_{\varphi_0})-\Scal_v(\varphi_0)\right)\dot{\varphi}_0\omega_{\varphi_0}^{[m]}+\int_X\dot{\varphi}_0\rho\wedge\omega_{\varphi_0}^{[m-1]}
\end{split}
\]
where $\cE_{v,w}^{\rho}:=\cE_{v,w}+\ce^{\rho}$.
\end{lemma}
\begin{proof} Using Theorem~\ref{M-extension} and the fact that $\cI^{\rho}$ is $d_1$-continuous (see \cite{BDL} or  Lemma~\ref{ext-E^rho}),  for any $t\in[0,1]$ there exists a sequence $(\varphi^{k}_t)_k\in\cK_\T(X,\bom)$ such that $\lim_{k\to \infty} d_1(\varphi^{k}_t,\varphi_t)= 0$ and $\cE_{v,w}^{\rho}(\varphi^{k}_t)\to \cE_{v,w}^{\rho}(\varphi_t)$. We let $[0,t]\ni s\mapsto \psi^{k}_s$ be the weak $C^{1,\bar{1}}$-geodesic joining $\varphi^{k}_0=\varphi_0$ with $\varphi^{k}_t$. By the proof of \cite[Cor.~1]{lahdili3}, we get
\[
\begin{split}
\underset{t\to 0^{+}}{\lim}\frac{\cE_{v,w}^{\rho}(\varphi^{k}_t)-\cE_{v,w}^{\rho}(\varphi_0)}{t}\geq &\int_X\left(w(\m_{\varphi_0})-\Scal_v(\varphi_0)\right)\dot{\psi}^{k}_0\omega_{\varphi_0}^{[m]}+\int_X\dot{\psi}^{k}_0\rho\wedge\omega_{\varphi_0}^{[m-1]}
\end{split}
\]
According to \cite[Lemma 3.4]{BDL},  we can use the dominated
convergence theorem on the RHS of the above inequality to conclude.
\end{proof}

The last step is to establish a weighted version of \cite[Prop.~3.2.]{BDL}.
\begin{prop}\label{prop: BDL2} Suppose $\T\subset \Aut_{r}(X)$ is a maximal torus, and let $v(\mu),w_0(\mu)>0$, $w=\ell^{\rm ext}_{v, w_0}w_0$.  Suppose that $\varphi^{*}\in \mathcal{K}_\T(X,\bom)\cap\cI^{-1}(0)$ is a $(v,w)$-cscK potential. Then, for any fixed K\"ahler form $\omega_{\varphi}$, $\varphi\in \mathcal{K}_{\T}(X,\bom)$,   there exists a $\sigma\in \G:=\T^{\C}$, such that
\[
\underset{\psi\in\mathcal{M}_{v,w}}{\inf}\ce^{\omega_\varphi}(\psi)=\ce^{\omega_\varphi}(\sigma[\varphi^{*}]).
\]
\end{prop}

\begin{proof} 
 As $\G$ is reductive,  there exists a unique $\sigma\in \G$ such that
\begin{equation}\label{min-I-orb}
\ce^{\omega_{\varphi}}(\sigma[\varphi^{*}])=\underset{\tau\in \G}{\inf} \, \ce^{\omega_{\varphi}}(\tau[\varphi^{*}]),
\end{equation}
(see e.g. \cite[Sect. 6]{DaRu} or \cite[Lemma 11]{lahdili3}) where, we recall,  the $\G$ action on potentials is introduced via  the slice $\cI^{-1}(0)$.  Let $\varphi_0:=\sigma[\varphi^{*}]\in \mathcal{K}_\T(X,\bom)\cap\cI^{-1}(0)$ and $\psi_0\in \mathcal{M}_{v,w}$ be the unique minimizer of $\ce^{\omega_{\varphi}}$. We want to show that $\varphi_0=\psi_0$. 

For $\lambda>0$, let $\psi_\lambda$ be the unique minimizer of $\cE^{\lambda\omega_{\varphi}}_{v,w}=\cE_{v,w}+\lambda\cI^{\omega_{\varphi}}$ on $\E_\T^{1}(X,\bom)\cap \cI^{-1}(0)$, given by Proposition~\ref{prop: BDL1}.  By this proposition, we know that $\lim_{\lambda \to 0}d_1(\psi_\lambda,\psi_0)= 0$. We denote respectively by $\dF_\lambda$ and $\dW$ the differentials of $\cE^{\lambda\omega_{\varphi}}_{v,w}$ and $\cI^{\omega_{\varphi}}$, viewed as $1$-forms on the Fr\'echet space $\cK(X, \bom)$. We thus have  $\forall \,  \psi\in\cK_\T(X, \bom),  \, \forall \, \dot\psi \in C^{\infty}_{\T}(X)$
\begin{equation}\label{linearizations}
\begin{split}
(\dF_0)_{\psi}(\dot{\psi}) &=-\int_X\left(\Scal_v(\omega_\psi)-w(\m_\psi)\right)\dot{\psi}\omega^{[m]}_\psi, \\
\dW_{\psi}(\dot{\psi}) &=\int_X \dot{\psi}\omega_\varphi\wedge\omega_{\psi}^{[m-1]}, \\
(\dF_{\lambda})_{\psi}(\dot{\psi}) &= (\dF_0)_{\psi}(\dot{\psi}) + \lambda \dW_{\psi}(\dot{\psi}). \end{split}
\end{equation}
Recall that the Mabuchi connection $\mathcal D$ on the Fr\'echet space $\cK_{\T}(X, \bom)$ is  introduced by
\begin{equation*}\label{Mabuchi-connection}
\left( {\mathcal D}_{\dot{\varphi}_t} {\dot \psi_t }\right)_{\varphi_t}:= {\ddot \psi}_t - \langle d\dot \psi_t, d\dot \varphi_t \rangle_{\omega_{\varphi_t}}, \end{equation*}
where $\varphi_t$ and $\psi_t$ are smooth paths in $\cK_{\T}(X, \bom)$. Using \cite[Lemma B.1]{lahdili2}, we compute the covariant derivative of $\dF_0$ with respect to the Mabuchi connection  to be
\[\begin{split}
\left( (\mathcal{D}_{\dot{\psi}_2}\dF_0)(\dot{\psi}_1)\right)_{\psi}=&\int_X\Big[2v(\m_\psi)\left((\nabla^{\omega_{\psi}}d\dot{\psi}_1)^{-},(\nabla^{\omega_{\psi}}d\dot{\psi}_2)^{-}\right)_{\omega_{\psi}} \\
\, \,  &+\left(\Scal_v(\omega_\psi)-w(\m_\psi)\right))(d\dot{\psi}_1,d\dot{\psi}_2)_{\omega_{\psi}}\Big]\omega^{[m]}_\psi,\end{split}
\]
where $(\nabla^{\omega_{\psi}}d\dot \psi)^{-}$ denotes the $(2,0)+(0,2)$ part of the Hessian of $\dot \psi$ with respect to the Levi--Civita connection $\nabla^{\omega_{\psi}}$ of $\omega_{\psi}$. Taking $\psi=\varphi_0$ to be  the  $(v,w)$-cscK potential, we get
 \begin{equation*}
 \begin{split}
\left( (\mathcal{D}_{\dot{\psi}_2}\dF_0)(\dot{\psi}_1)\right)_{\varphi_0}=&2\int_X\left((\nabla^{\omega_{\varphi_0}}d\dot{\psi}_1)^{-},(\nabla^{\omega_{\varphi_0}}d\dot{\psi}_2)^{-}\right)_{\omega_{\varphi_0}} v(\m_{\varphi_0})\omega^{[m]}_{\varphi_0}\\
=&2\int_X {\mathbb L}_{\omega_{\varphi_0}, v}(\dot{\psi}_1)\dot{\psi}_2 \omega^{[m]}_{\varphi_0}= 2\int_X {\mathbb L}_{\omega_{\varphi_0}, v}(\dot{\psi}_2)\dot{\psi}_1 \omega^{[m]}_{\varphi_0},
\end{split}
\end{equation*}
where the operator ${\mathbb L}_{\omega_{\psi}, v}(\dot{\psi}) :=\delta_{\omega_{\psi}}\delta_{\omega_{\psi}}\left(v(\m_\psi)(\nabla^{\omega_{\psi}}d\dot{\psi})^{-}\right)$ is a 4th order elliptic self-adjoint operator on $(X, \omega_{\psi})$, with kernel given by the space of Killing potentials in $C^{\infty}_{\T}(X)$, see Appendix~\ref{s:appendix}. 

As $\varphi_0$ is a  $(v,w)$-cscK potential which satisfies \eqref{min-I-orb},  we have by \cite[Lemma~10]{lahdili3} that $\dW_{\varphi_0}(\dot{\psi})=0$ for any  $\T$-invariant Killing potential $\dot{\psi}$ with respect to $\omega_{\varphi_0}$. It follows that we can solve the linear equation (for a function $\dot \psi \in C^{\infty}_{\T}(X)$)
\[ {\mathbb L}_{\omega_{\varphi_0}, v}(\dot{\psi}) = \frac{\omega_{\varphi}\wedge\omega_{\varphi_0}^{[m-1]}}{\omega_{\varphi_0}^{[m]}}\]
as the RHS is $L^2$-orthogonal (with respect to the measure $\omega_{\varphi_0}^{[m]}$) to the kernel of ${\mathbb L}_{\omega_{\varphi_0}, v}$.
Equivalently, there exits a $\dot \psi_0 \in C^{\infty}_{\T}(X)$, such that we have equality of $1$-forms on $\cK_{\T}(X, \bom)$:
\begin{equation}\label{Eq-DF=W}
(\mathcal{D}_{\dot{\psi}_0}\dF_0)_{\varphi_0}=-\dW_{\varphi_0}.
\end{equation}
Let $\lambda \to \dot{\phi}_\lambda \in C^{\infty}_\T(X)$ be a smooth curve in the tangent space to $(\varphi_0+\lambda\dot{\psi}_0)\in\mathcal{K}_\T(X,\bom)$, defined for $\lambda>0$ small enough. We compute
\begin{equation}\label{d-lambda-F}
\dfrac{d}{d\lambda}_{\mid \lambda=0} (\dF_\lambda)_{\varphi_0+\lambda\dot{\psi}_0}(\dot{\phi}_\lambda)=\dW_{\varphi_0} (\dot{\phi}_0)+\left((\mathcal{D}_{\dot{\psi}_0} \dF_0)(\dot{\phi}_0)\right)_{\varphi_0}+(\dF_0)_{\varphi_0}(\dfrac{d}{d\lambda}_{\mid \lambda=0}\dot{\phi}_\lambda) =0,
\end{equation}
where we have used \eqref{Eq-DF=W} and  that $(\dF_0)_{\varphi_0}=0$ since $\varphi_0$ is a $(v,w)$-cscK potential, see \eqref{linearizations}. On the other hand, letting
\[
f_\lambda:=-\Scal_v(\omega_{\varphi_0+\lambda\dot{\psi}_0})+w(\m_{\varphi_0+\lambda\dot{\psi}_0})+\Big\langle\omega_{\varphi_0+\lambda\dot{\psi}_0}, \omega_{\varphi}\Big\rangle_{\omega_{\varphi}} 
\]
it follows from \eqref{linearizations}  that for any $\dot \phi \in C^{\infty}_{\T}(X)$
\[
\begin{split}
(\dF_\lambda)_{\varphi_0+\lambda\dot{\psi}_0}(\dot{\phi})&=\int_X\dot{\phi} f_\lambda\omega_{\varphi_0+\lambda\dot{\psi}_0}^{[m]}.
\end{split}
\]
Thus \eqref{d-lambda-F} implies that $f_\lambda=O(\lambda^{2})$ and 
\[
|(\dF_\lambda)_{\varphi_0+\lambda\dot{\psi}_0}(\dot{\phi})|\leq C\lambda^{2}\underset{X}{\sup}|\dot{\phi}|.
\]
Let $\psi_\lambda(t)\in \E_\T^{1}(X,\bom)$ be a finite energy geodesic connecting $\psi_\lambda(0):=\psi_\lambda\in \E_\T^{1}(X,\bom)$ with $\psi_\lambda(1):=\varphi_0+\lambda\dot{\psi}_0\in\cK_\T(X,\bom)$ for $\lambda>0$ small enough. By Lemma \ref{p:sub-slop}, we get
\[
\dfrac{d}{dt}_{\mid t=1^{-}}\cE_{v,w}^{\lambda\omega_{\varphi}}(\psi_\lambda(t))\leq \int_X\dot{\psi}_\lambda(1) f_\lambda\omega_{\varphi_0+\lambda\dot{\psi}_0}^{[m]}.
\]
By Proposition~\ref{prop: BDL1},  $d_1(0,\psi_\lambda(0))$ is uniformly bounded. As  $\psi_\lambda(1):=\varphi_0+\lambda\dot{\psi}_0\in\cK_\T(X,\bom)$, it follows that $d_1(0,\psi_\lambda(1))$ is uniformly bounded for $\lambda$ small enough. We thus have that both $d_1(0,\psi_\lambda(0))$ and $d_1(0,\psi_\lambda(0))$ are uniformly bounded and by  \cite[Lemma 3.4(ii)]{BDL} we get
\[
\int_X |\dot{\psi}_\lambda(1)|\omega_{\varphi_0+\lambda\dot{\psi}_0}^{[m]} =d_1(\psi_\lambda(0),\psi_\lambda(0))\leq d_1(0,\psi_\lambda(0))+d_1(0,\psi_\lambda(1))\leq C.
\]
From $f_\lambda=O(\lambda^{2})$ we obtain
\[
\dfrac{d}{dt}_{\mid t=1^{-}}\cE_{v,w}^{\lambda\omega_{\varphi}}(\psi_\lambda(t))\leq O(\lambda^{2}).
\]
As the unique minimizer of the strictly convex functional $\cE_{v,w}^{\lambda\omega_{\varphi}}$ on $\E^{1}_{\T}(X,\bom)\cap \cI^{-1}(0)$ is $\psi_\lambda(0)=\psi_\lambda$, 
\[
\dfrac{d}{dt}_{\mid t=1^{-}}\cE_{v,w}^{\lambda\omega_{\varphi}}(\psi_\lambda(t))\geq \dfrac{d}{dt}_{\mid t=0^{+}}\cE_{v,w}^{\lambda\omega_{\varphi}}(\psi_\lambda(t))\geq 0.
\]
Using that the functions $t\mapsto \ce^{\omega_{\varphi}}(\psi_\lambda(t))$ and $t\mapsto \cE_{v,w}(\psi_\lambda(t))$ are both convex (this follows from \cite[Prop.~4.5]{BDL} and Theorem~\ref{M-extension}), we have
\[
0\leq \lambda\left(\dfrac{d}{dt}_{\mid t=1^{-}}-\dfrac{d}{dt}_{\mid t=0^{+}}\right)  \ce^{\omega_{\varphi}}(\psi_\lambda(t))   \leq    \left(\dfrac{d}{dt}_{\mid t=1^{-}}-\dfrac{d}{dt}_{\mid t=0^{+}}\right)  \cE_{v,w}^{\lambda\omega_{\varphi}}(\psi_{\lambda}(t))   \leq  O(\lambda^{2}).
\]
By the convexity of $t\mapsto \ce^{\omega_{\varphi}}(\psi_\lambda(t))$, the last estimate also gives
\[
\begin{split}
0\leq& t\ce^{\omega_{\varphi}}(\psi_\lambda(1))+(1-t)\ce^{\omega_{\varphi}}(\psi_\lambda(0))-\ce^{\omega_{\varphi}}(\psi_\lambda(t))\\
=&t(1-t)\left(\frac{\ce^{\omega_{\varphi}}(\psi_\lambda(1))-\ce^{\omega_{\varphi}}(\psi_\lambda(t))}{1-t}\right)-t(1-t)\left(\frac{-\ce^{\omega_{\varphi}}(\psi_\lambda(0))+\ce^{\omega_{\varphi}}(\psi_\lambda(t))}{t}\right)\\
\leq& t(1-t) O(\lambda).
\end{split}
\]
Letting $\lambda\to 0$, and using using the endpoint stability of the finite energy geodesic segments (see \cite[Prop. 4.3]{BDL}), together with the $d_1$-continuity of $\ce^{\omega_{\varphi}}$ established in \cite[Prop. 4.4]{BDL}, it follows that $t\mapsto\ce^{\omega_{\varphi}}(\psi(t))$ is linear along the finite energy geodesic $\psi(t)=\underset{\lambda\to 0^{+}}{\lim} \psi_\lambda(t)$ connecting $\psi_0(0)=\psi_0$ and $\psi_0(1)=\varphi_0$. The strict convexity of $\ce^{\omega_{\varphi}}$ along finite energy geodesics (\cite[Prop.~4.5]{BDL}) then  yields $\psi_0=\varphi_0=\sigma[\varphi^{*}]$.
\end{proof}
Now, we are in position to prove Theorem~\ref{W&S} by the arguments in \cite[Thm.~1.4]{BDL}.

\begin{proof}[\bf Proof of Theorem \ref{W&S}]
Without loss of generality, we can assume that the $(v,w)$-extremal metric $\omega^{*}=\bom$ is the initial metric, and we suppose $\psi_0\in\E_\T^{1}(X,\bom)\cap \cI^{-1}(0)$ is a weak minimizer of $\cE_{v,w}:\E_\T^{1}(X,\bom)\to\R\cup\{\infty\}$. We want to show that $\psi_0=\sigma[0]$ for some $\sigma\in \G=\T^{\C}$. It is well-known (see \cite{D} or Lemma~\ref{K0-completion}) that there exists a sequence $\varphi_j\in \cK_\T(X,\bom)\cap \cI^{-1}(0)$ such that $d_1(\varphi_j,\psi_0)\to 0$. We set $\rho_j=\bom+dd^{c}\varphi_j$ which is a $\T$-invariant K\"ahler form. 

Since $\bom$ is $(v,w)$-extremal metric, $\mathcal{M}_{v,w}$ is non-empty. By Proposition~\ref{prop: BDL1}, the functional $\cE_{v,w}^{\lambda\rho_j}=\cE_{v,w}+\lambda\ce^{\rho_j}$ admits a unique minimizer $\psi_{j,\lambda}\in \E^{1}_\T(X,\bom)\cap \cI^{-1}(0)$,  such that 
\[
I(\varphi_j, \psi_{j,\lambda})\leq m(m+1)I(\varphi_j,\psi_0).
\]
By the quasi-triangle identity \cite[(2.16)]{BDL}, we get
\begin{equation}\label{I-lambda-j}
I(\psi_0, \psi_{j,\lambda})\leq C\left(I(\psi_0,\varphi_j)+ I(\varphi_j, \psi_{j,\lambda})\right)
\leq C(m^{2}+m+1) I(\varphi_j,\psi_0),
\end{equation}
 where $C>0$ is a uniform constant depending only on $m$.

Let $j>0$ be fixed. According to Proposition~\ref{prop: BDL1},  $\psi_{j,0}:= \lim_{\lambda \to 0} \psi_{j, \lambda}$ is the unique minimizer of $\cI^{\lambda\rho_j}$ on $\mathcal{M}_{v,w}$ whereas  Proposition \ref{prop: BDL2} yields that there exists a $\sigma_j\in \G$ such that $\psi_{j,0}=\sigma_j[0]$. Letting $\lambda\to 0^{+}$ in \eqref{I-lambda-j} (and using the $d_1$-continuity of $I$, see e.g. \cite{BDL0} or Lemma~\ref{l:I,J-extension}), we have
\[
I(\psi_0, \sigma_j[0])\leq C(m^{2}+m+1)  I(\varphi_j,\psi_0).
\]
Taking  $j\to\infty$ (and using $d_1(\varphi_j, \psi_0) \to 0$), we get $I(\psi_0,\sigma_j[0])\to 0$. By \cite[ Prop. 2.3]{BBEGZ} and \cite[Prop. 5.9]{Da},  the latter limit is  equivalent to $d_1(\sigma_j[0],\psi_0)\to 0$. Using \cite[Lemma 3.7]{BDL}, there exists  a $\sigma\in \G$ such that $\sigma[0]=\psi_0$.
\end{proof}
\begin{rem}\label{r:Chi-Li} The arguments in the proofs of Proposition~\ref{prop: BDL2} and Theorem~\ref{W&S} extend  if we remove the maximality assumption for $\T\subset \Aut_{r}(X)$, but replace the group $\G=\T^{\C}$ with the connected component of the identity  $\hat \G=\Aut_r^{\T}(X)$ of the centralizer of $\T$ in $\Aut_r(X)$. The key points are that $\hat \G$ is reductive (see Proposition~\ref{T-max}), and $\hat \G$ acts transitively on the space of $\T$-invariant $(v, w_0)$-extremal K\"ahler metrics (see Theorem~\ref{thm:uniqueness-weighted}).
\end{rem}

\begin{proof}[\bf Proof of Theorem~\ref{main}] We apply the Coercivity Principle of \cite{DaRu}, see Theorem~\ref{thm:coercivity-principle}.  By Theorem~\ref{M-extension}, the extension of the weighted Mabuchi energy $\cE_{v,w}$  to the space $\E^1_{\T}(X, \bom)$ satisfies the hypotheses of Theorem~\ref{thm:coercivity-principle} (the invariance of $\cE_{v,w}$ under the action of $\G=\T^{\C}$ is equivalent to the necessary condition \eqref{Futaki} for the existence of a $(v,w)$-cscK metric). We thus need to ensure that $\cE_{v,w}$ further satisfies the properties (i)-(iv) of Theorem~\ref{thm:coercivity-principle}. Theorem~\ref{M-extension} also yields the  convexity property (i)  whereas the regularity property (ii) is established in Theorem~\ref{W&S}. This last result also yields the uniqueness property (iii), via Theorem~\ref{thm:uniqueness-weighted}. Finally, the compactness property (iv) is established in Theorem~\ref{compact}. 
\end{proof}
\begin{rem}\label{r:hat-G} By virtue of Theorem~\ref{thm:uniqueness-weighted} and Remark~\ref{r:Chi-Li}, the conclusion of  Theorem~\ref{main} holds true if one drops the assumption that $\T\subset \Aut_r(X)$ is a maximal torus, but instead of $\T^{\C}$ one considers the larger reductive group $\hat \G=\Aut^{\T}_r(X)$ (see Proposition~\ref{T-max}). 
 \end{rem}

\section{Proofs of Theorems~\ref{extremal-bundle} and \ref{v-soliton-bundle}}\label{s:proofs2} 

\begin{proof}[\bf Proof of Theorem~\ref{extremal-bundle}] The implication ${\rm (ii)} \Rightarrow {\rm (i)}$ follows from Lemma~\ref{l:Y-Scal} whereas ${\rm (ii)} \Rightarrow {\rm (iii)}$ is established in Theorem~\ref{main}. We shall prove below ${\rm (iii)} \Rightarrow {\rm (ii)}$ and ${\rm (i)} \Rightarrow {\rm (ii)}$. The arguments are very similar to the ones in  the proof of \cite[Thm.~1]{Jubert} where the case when $(X, \T)$ is toric is studied. The main idea is to show that on a semi-simple principal $(X, \T)$-fibration, the continuity path  used  by  Chen--Cheng~\cite{CC}  in the cscK-case and its modification by He~\cite{He-ext} to the extremal case, can be adapted to  bundle-compatible construction. We sketch the proof below for Reader's convenience.
\begin{proof}[Proof of ${\rm (iii)} \Rightarrow {\rm (ii)}$] We shall work on $Y$. Let $\tilde \omega_0$ be a bundle-compatible K\"ahler metric on $Y$, corresponding to a $\T_X$-invariant K\"ahler metric $\bom$ on $X$. By Lemma~\ref{T}, $\tilde \omega_0$ is invariant under a maximal torus $\Ko_Y \subset \Aut_r(Y)$ (containing $\T_{Y}$), and  by this lemma and Lemma~\ref{l:Mabuchi}, the extremal affine-linear function corresponding to $\Ko_Y$  is the pull-back  to the vector space $\ko_Y^*=\left({\rm Lie}(\Ko_Y)\right)^*$ of the extremal affine-linear function $\ell^{\rm ext}(\mu)$ on $\tor$ defined in Theorem~\ref{extremal-bundle}-(ii). Furthermore, by Lemma~\ref{l:Mabuchi}, we have  that the restriction of $\cE^Y_{1, \ell^{\rm ext}}$ to the subspace $\cK_{\T}(X, \bom) \subset \cK_{\Ko_Y}(Y, \tilde \omega_0)$ (see Corollary~\ref{c:bundle-compatible} and Lemma~\ref{T}) is a positive multiple of $\cE_{p,\tilde w}^X$,  where the weights are the one defined in Theorem~\ref{extremal-bundle}-(ii).  In this setup, the main ingredients of the proof are as follows.

\smallskip
\noindent
{\it Step 1.}  Following \cite{CC,Hash,He-ext}, one considers the continuity path  $\varphi_t \in \cK_{\Ko_Y}(Y, \tilde \omega_0)$, determined by the solution of the PDE
\begin{equation}\label{e:continuity-path}
t\left(\Scal(\tilde{\omega}_{\varphi_t}) - \ell^{\rm ext}(\m_{\tilde \omega_{\varphi_t}})\right)= (1-t)\left(\tr_{\tilde \omega_{\varphi_t}}(\tilde \rho)- (n+m)\right), \qquad t\in (0, 1), 
\end{equation}
where $\tilde \rho$ is a suitable (fixed) $\Ko_Y$-invariant K\"ahler metric on $Y$ in the class $[\tilde \omega_0]$. By \cite{CC,He-ext}, there exits  $\tilde \rho \in [\tilde \omega_0]$ and a $t_0 \in (0,1)$,  such  that  a solution $\varphi_t$ of  \eqref{e:continuity-path} exits for $t$ in the interval $[t_0, 1)$; furthermore,  the solution $\varphi_t(y)$ is smooth as a function on $[t_0, 1) \times Y$.   The main observation of \cite{Jubert} is that,  with a suitable choice for $\tilde \rho$, the path \eqref{e:continuity-path} can in fact be reduced to a continuity path on $X$. To see this, we observe that,  by \cite[Prop. 3.1]{He-ext}, one can take $\tilde \rho$ in \eqref{e:continuity-path}
to be of the form $\tilde \rho= \tilde \omega_0 + \frac{1}{r_0}dd^c f$ with $r_0$ large enough,  where $f$ is the smooth function on $Y$  with zero mean with respect to $\tilde \omega_0$, which solves the Laplace equation
\[\Delta_{\tilde \omega_0} f = \left(\Scal(\tilde{\omega}_{0}) - \ell^{\rm ext}(\m_{\tilde \omega_{0}})\right).\]
By Lemmas~\ref{l:Y-Scal} and~\ref{l:operators-tilde},  $f\in C^{\infty}_{\T}(X)$,  whereas by Lemma~\ref{l:Y-embedding}  $\tilde \rho$  is bundle-compatible, i.e.
\[\tilde{\rho}=\rho+\sum_{a=1}^{k}(\langle p_a,\mu_\rho\rangle+c_a\rangle)\pi^{*}_B\omega_a+\langle d\mu_\rho\wedge\theta\rangle), \]
where $\rho = \omega_0 + \frac{1}{r_0} dd^c f$ is a $\T$-invariant K\"ahler metric on $X$, see \eqref{Y-Kahler}.
Using Lemma~\ref{l:Y-Scal} and  that  both $\tilde \omega_{\varphi}$ and  $\tilde \rho$ are of the form  \eqref{Y-Kahler}, we get a path of PDE's on $X$ of the form  
\begin{equation}\label{reduced-path}
 t \left(\Scal_p({\omega}_{\varphi_t}) - \tilde w(\m_{\omega_{\varphi_t}})\right)= (1-t) H(\varphi_t), \qquad t\in (t_0, 1),  \end{equation}
where $\varphi_t \in \cK_{\T}(X, \omega_0)$ and $H(\varphi_t) :=\left(\tr_{\tilde \omega_{\varphi_t}}(\tilde \rho)-(n+m)\right)$ is manifestly 
a second order differential operator on $X$  for  $\varphi_t \in \cK_{\T}(X, \omega_0)\subset \cK_{\Ko_{Y}}(Y, \tilde \omega_0)$. It follows that the solution $\varphi_t, \, t\in [t_0, 1)$ of \eqref{e:continuity-path} will actually belong to $\cK_{\T}(X, \bom) \subset \cK_{\Ko_Y}(Y, \tilde \omega_0)$.  This last point  is a consequence of the implicit function theorem (used in \cite{Hash,He-ext} to establish the openness) which can be applied directly to \eqref{reduced-path}; to find the linearization of \eqref{reduced-path}, we use~\cite{Hash}  that the linearization of $H(\varphi)$ on $Y$ is the operator $\Ha^{\tilde\rho}_{\tilde \omega_{\varphi}, 1}$ (see Definition~\ref{diff-oper}) so that, by virtue of  Lemma~\ref{l:operators-tilde}, the linearization of $H(\varphi)$ when restricted to $\cK_{\T}(X, \bom) \subset \cK_{\Ko_Y}(Y, \tilde \omega_0)$ is given by the $p$-weighted operator $\Ha^{\rho}_{\omega_{\varphi}, p}$ introduced in Appendix~\ref{s:appendix}. Similar argument allows us to identify the linearization of  $\Scal_p(\omega_{\varphi})$ (see  also \cite[Lemma~B1]{lahdili2}). We refer the Reader to \cite[Sect.~6]{Jubert} for further details.

\smallskip
\noindent
{\it Step 2.} 
The next ingredient is a deep result from \cite{CC} with a complement in \cite{He-ext}, showing that  if  $\cE_{1, \ell^{\rm ext}}^Y$ is $G$-coercive along the continuity  path $\varphi_t$ with respect to a reductive subgroup $G\subset \Aut_r(Y)$ containing the torus generated by the extremal vector field $\xi^Y_{\rm ext}= \d\ell^{\rm ext}\in \tor_Y$ in its center, then there exists a subsequence of times $j\to 1$ and elements $\sigma_j \in G$, such that $\sigma_j^*(\tilde \omega_{\varphi_j})$ converges in $C^{\infty}(Y)$ to an extremal K\"ahler metric $\tilde \omega_1$. In our case, assuming {\rm (iii)}, we have that $\cE^Y_{1, \ell^{\rm ext}}(\varphi_t)= \V(B, \omega_B) \cE_{p,\tilde w}^{X}(\varphi_t)$ (see Lemma~\ref{l:Mabuchi}) is $\G=\T_Y^{\C}$-coercive (see Lemmas~\ref{l:I_p}, \ref{Han-Li} and Proposition~\ref{p:J-proper}). We can thus find $\sigma_j\in \T_Y^{\C}$ and $\varphi_j$ as above. The K\"ahler metrics $\sigma_j^*(\tilde \omega_{\varphi_j})$ are bundle-compatible in the sense of Definition~\ref{d:(X, T)-principal}, and thus are of the form $\sigma_j^*(\tilde \omega_{\varphi_j})= \tilde \omega_0 + d_Y d^c_Y \sigma_j[\varphi_j]$, $\sigma_j[\varphi_j] \in \cK_{\T}(X, \bom)\subset \cK_{\Ko_{Y}}(Y, \tilde \omega_0)$. It follows that $\tilde \omega_1$ is bundle-compatible extremal K\"ahler metric on $Y$ (as $\cK_{\T}(X, \bom)$ is $C^{\infty}(Y)$ closed in $\cK_{\Ko_Y}(Y, \tilde \omega_0)$). By Lemma~\ref{l:Y-Scal}, the corresponding K\"ahler metric $\omega_1$ on $X$ is then $(p,\tilde w)$-cscK.  \end{proof}

\begin{proof}[\it Proof of ${\rm (i)} \Rightarrow {\rm (ii)}$] The proof is very similar to the proof of ${\rm (iii)} \Rightarrow {\rm (ii)}$. As in the Step 1 of the latter, we consider the continuity path \eqref{e:continuity-path} which defines potentials  $\varphi_t \in \cK_{\T}(X, \bom) \subset \cK_{\Ko_Y}(Y, \tilde \omega_0)$ for $t\in [t_0, 1)$.   We can assume without loss~\cite{calabi}  that $Y$ admits a $\Ko_{Y}$-invariant extremal K\"ahler metric in $[\tilde \omega_0]$, where $\Ko_Y\subset \Aut_r(Y)$ is the maximal torus given by Lemma~\ref{T}. This implies that $\cE_{1, \ell^{\rm ext}}^Y$ is $G$-coercive for $G=\Ko_Y^{\C}$. Indeed, this can be justified for instance by applying Theorem~\ref{main} and Proposition~\ref{p:J-proper} in the case $(v, w)=(1, \ell^{\rm ext})$.  As in the Step 2 of the proof of ${\rm (iii)} \Rightarrow {\rm (ii)}$, we use \cite{CC,He-ext} and the $G$-coercivity of $\cE^Y_{1, \ell^{\rm ext}}$ along the path in order to find a sub-sequence of times  $j\to 1$ and elements $\sigma_j \in G$,  such that $\sigma_j^*(\tilde \omega_{\varphi_j})$ converges in $C^{\infty}(Y)$ to a $\Ko_{Y}$-invariant extremal K\"ahler metric $\tilde \omega_1 \in [\tilde \omega_0]$. However, unlike the proof of ${\rm (iii)} \Rightarrow {\rm (ii)}$, in general $\sigma_j^*(\tilde \omega_{\varphi_j})$  and hence $\tilde \omega_1$ are not bundle-compatible, as $\sigma_j$ can act non-trivially on $B$ (see the proof of Lemma~\ref{T}). We thus need to modify slightly the argument in order to show that $\tilde \omega_1$ still induces a $(p, \tilde w)$-cscK metric on any given fibre $X_b= \pi_B^{-1}(b) \subset Y$. We denote by $\omega_j(b): = (\tilde \omega_{\varphi_j})_{|_{X_b}}$ and $\overline{\omega}_j (b):= (\sigma_j^*(\tilde \omega_{\varphi_j}))_{|_{X_b}}$ the induced $\T$-invariant metrics on $X_b$. As $\tilde \omega_{\varphi_j}$ is bundle-compatible, Lemma~\ref{l:Y-Scal} yields \[\Scal_p(\omega_j(b) )= \left[p(\m_{\tilde \omega_{\varphi_j}}) \Scal(\tilde \omega_{\varphi_j}) - p(\m_{\tilde \omega_{\varphi_j}}) q(\m_{\tilde \omega_{\varphi}})\right]_{|_{X_b}}.\] Using that $\sigma_j \in \Ko_Y^{\C}$  sends the fibre $X_b$ to the fibre $X_{\sigma_j(b)}$ (this follows from the construction of $\Ko_Y$ in the proof of Lemma~\ref{T}) the above equality holds true for the metrics $\overline{\omega}_j(b)$, where in the RHS we replace the metric $\tilde \omega_{\varphi_j}$ on $Y$ with $\bar \omega_j :=\sigma_j^*(\tilde \omega_{\varphi_j})$. It thus follows by the smooth convergence of $\overline{\omega}_j(b)$ to $\omega_1(b)$ that \[ \begin{split} \Scal_p(\omega_1(b)) & = \left[p(\m_{\tilde \omega_{1}}) \Scal(\tilde \omega_{1}) - p(\m_{\tilde \omega_{1}}) q(\m_{\tilde \omega_{1}})\right]_{|_{X_b}}\\ &=  \left[p(\m_{\tilde \omega_{1}})(\ell^{\rm ext}(\m_{\tilde \omega_1}) - q(\m_{\tilde \omega_1})\right]_{|_{X_b}} = \tilde w(\m_{\omega_1(b)}), \end{split}\] where for the equalities on the second line we have used that the $\Ko_Y$-extremal function $\ell^{\rm ext}\in \Aff(\tor_X^*)$ (see Lemma~\ref{T}). Thus $\omega_1(b)$ is a $(v, \tilde w)$-cscK metric on $X$. \end{proof}

\begin{proof}[\bf Proof of Theorem~\ref{v-soliton-bundle}] In \cite{HL}, Han--Li introduced a  functional  $\cE^{\HL}_v:\cK_\T(X,\bom)\to \R$ whose critical points are the $v$-solitons, see \cite[Lemma 4.4]{HL}. A careful inspection using \eqref{Chen-Tian} shows that $\cE_{v}^{\HL}(\omega)=\cE_{v, w}(\omega)- \int_X \log(v(\m_\omega))v(\m_\omega)\omega^{[m]}$, where
$w$ is the weight function defined in Proposition~\ref{tilde v}. Thus,  the difference of the two functionals is a constant independent of  the choice of a $\T$-invariant K\"ahler metric $\omega \in 2\pi c_1(X)$, see e.g. \cite{lahdili2}.
Thus, by \cite[Thm.~3.5]{HL} applied to $(X, 2\pi c_1(X), \T)$ (and weights $pv, \tilde w$), the  $\T^{\C}$-coercivity of $\cE_{pv, \tilde w}^X$ is equivalent with the existence of a $vp$-soliton on $X$.  By Lemma~\ref{Fano},  this implies that $Y$ admits a (bundle-compatible) $v$-soliton.

By~\cite[Thm.~1.7]{HL},  the  $\T^{\C}$-coercivity of $\cE_{pv, \tilde w}^X$ is also equivalent to the uniform $vp$-K-stability on $\T$-equivariant \emph{special} test configurations. When $(X, \T)$ is a toric Fano variety, the only such test configurations are the product test configurations, and thus by \cite[Prop.3]{lahdili2}, the condition is  reduced to verifying  \eqref{Futaki}  on $X$  with respect to the weights $(pv, \tilde w)$.

\smallskip
By the above conclusion, in order to show the existence of  a K\"ahler--Ricci soliton, it is sufficient  to find  $\xi_0\in \tor$, such \eqref{Futaki} is satisfied for the weights functions $v(\x)=e^{\langle \xi_0, \x\rangle}p(\xi)$ and $\tilde w(\x)=2p(\x)e^{\langle \xi_0, \x\rangle}(m + \langle \xi_0, \x \rangle + \langle d\log p, \x\rangle)$.  We detail the proof of this fact below.

Let $\omega \in 2\pi c_1(X)$ be any $\T$-invariant K\"ahler metric with canonically normalized momentum map $\m_{\omega} : X \to \Pol$. We then consider the following $p$-weighted version of a functional on $\tor$, defined originally by Tian--Zhu~\cite[Lemma~2.2]{TZ}:
\begin{equation}\label{reduced-functional}
 \xi \to \int_X e^{\langle \xi, \m_{\omega}\rangle}\,  p(\m_{\omega}) \omega^{[m]}, \qquad \xi \in \tor.\end{equation}
The convexity and properness  of the above functional follow by the arguments in \cite[Lemma~2.2]{TZ}, but under our toric assumption these can also be seen directly by rewriting  the RHS in \eqref{reduced-functional}  as an integral over the Delzant polytope: 
\[ \xi \to (2\pi)^m \int_{\Pol} e^{\langle \xi, \x\rangle} \p(\x) d\x.\] 
The properness of the latter follows by the fact that the origin is in the interior of $\Pol$ (by the canonical normalization condition of $\Delta$, see Remark~\ref{r:canonical-normalization}). Let $\xi_0\in \tor$ be the unique critical point of \eqref{reduced-functional}. We have that 
\[ \int_X \langle \zeta, \m_{\omega}\rangle \, e^{\langle \xi_0, \m_{\omega}\rangle}\,  p(\m_{\omega}) \omega^{[m]}=0, \]
which is precisely the condition ${\rm Fut}_{v, \tilde w}=0$ according to  Lemma~\ref{l:Futaki-TZ} in the Appendix~\ref{s:Futaki}.

\smallskip
The existence of a Sasaki--Einstein structure follows by a similar  argument: By Proposition~\ref{SE}, Lemma~\ref{Fano} and Proposition~\ref{tilde v} in that order, we want to find  $\xi_0\in \tor$ such that \eqref{Futaki} holds true for the weights given as in Proposition~\ref{tilde v},  with $v(\x)= p(\x)(\langle \xi_0, \x \rangle + a)^{-(m+n+2)}$. (This will be enough to conclude  the existence of a $pv$-soliton on the toric Fano manifold $(X, \T)$ and hence a $v$-soliton on $Y$ by the general arguments evoked above.) We argue based on \cite{MSY} who introduced the \emph{volume functional} on the space of \emph{normalized} positive affine-linear functions on $\Delta$. Strictly speaking, the functional  in  \cite[Sect.3]{MSY} is introduced on the principle $\Sph^1$-bundle $N$ over $(X,\omega)$ (which admits a natural strictly pseudo-convex CR structure $(\Ds, J)$ coming from $X$),  and  is then defined as the Sasaki volume of a $(\Ds, J)$-compatible normalized Sasaki--Reeb vector field $\hat \xi$ on $N$; using the point of view of \cite{ACL} (see in particular Lemma 1.4),  the volume functional can also be written on $X$,  noting that positive affine-linear functions $\ell_{\xi}= \langle \xi, \x \rangle + a$ over $\Pol$ are in bijection with Sasaki--Reeb vector fields $\hat \xi$ on $(N, \Ds, J)$,  and the normalization condition used in \cite{MSY} is equivalent to requiring  $\ell_{\xi}(0)= a=1$. Specifically, in our toric weighted setting, we let
\[
  \xi   \to   \int_X  (\langle \xi,  \m_{\omega} \rangle + 1)^{-(m+n+1)} p(\m_{\omega}) \omega^{[m]} 
   = (2\pi)^m \int_{\Pol}(\langle \xi,  \x \rangle + 1)^{-(m+n+1)} p(\x) d\x,\]
which is defined for $\xi\in \tor$ such that  $(\langle \xi,  \x \rangle + 1)> 0 $ on  $\Pol$; the properness of the functional follows by the fact that a canonically normalized  Delzant  polytope of a Fano toric manifold is determined  by  $\Pol=\{\x : L_j(\x) \ge 0\}$ where the affine-linear functions $L_j(\x)$ satisfy $L_j(0)=1$, see e.g. \cite[Sect.~7.4]{apostolov-notes}.  The unique critical point $\xi_0 \in \tor$ of the above convex functional  then satisfies 
\[ \int_X \langle \zeta, \m_{\omega}\rangle \,(\langle \xi_0,  \m_{\omega} \rangle + 1)^{-(m+n+2)} p(\m_{\omega})  \omega^{[m]}=0, \qquad \zeta \in \tor, \]
which, by Lemma~\ref{l:Futaki-TZ},  is precisely the condition \eqref{Futaki} for the weight functions considered. This concludes the proof of Theorem~\ref{v-soliton-bundle}. \end{proof}

\appendix
\section{Weighted differential operators}\label{s:appendix} Let $(X, \omega, \T)$ be as in Section~\ref{s:weighted-setup} and $v>0$ be a positive smooth weight function defined over the polytope $\Pol$.  We denote by  $\nabla^{\omega}$ the Levi--Civita connection of the Riemannian metric $g_{\omega}$,  and by $\delta_{\omega}$ the formal adjoint of $\nabla^{\omega}$.
We define the following  weighted differential operators which are self-adjoint with respect to the volume form $v(\m_{\omega})\omega^{[m]}$ on $X$.
\begin{defn}\label{diff-oper}
The  \emph{$v$-weighted Laplacian} of $\psi$ is the second order operator acting of smooth functions defined by
\begin{equation}\label{v-Delta}
\Delta_{\omega,v}(\psi) = \frac{1}{v(\m_\omega)}\delta_{\omega} (v(\m_{\omega}) d \psi).
\end{equation}
The \emph{$v$-weighted linear Lichnerowicz} operator is the forth-order operator given  by
\begin{equation}\label{Lic_v}
\mathbb{L}_{\omega,v}(\psi):=\frac{\delta_\omega\delta_\omega \left(v(\mu_\omega)(\nabla^{\omega}d\psi)^{-}\right)}{v(\mu_\omega)},
\end{equation}
where  $(\nabla^{\omega} d\phi)^{-}$ stands for the $(0,2)$-symmetric tensor of type $(2,0)+(0,2)$ with respect to the complex structure of $X$. 
For any $\T$-invariant K\"ahler form $\rho$ on $X$,  we define the second-order  operator given by
\begin{equation}
\begin{split}
{\Ha}^{\rho}_{\omega, v}(\psi):=& \langle \rho, dd^c \psi \rangle_{\omega}  + \langle d \tr_{\omega} (\rho), d\psi\rangle_{\omega} +\frac{1}{v(\m_{\omega})}\langle \rho, d v(\m_{\omega})\wedge d^c \psi\rangle_{\omega},
\end{split}
\end{equation}
where $\tr_{\omega} (\rho):= \left(\rho\wedge \omega^{[m-1]}\right)/\omega^{[m]} =\langle \rho, \omega \rangle_{\omega}$. The operator 
${\Ha}^{\rho}_{\omega, v}$ is  a $v$-weighted version of the linear operator used in~\cite{Hash}.
\end{defn}
A straightforward computation shows that
\begin{lemma}\label{l:Lic_v}
The $v$-weighted Lichnerowicz's operator can be written as 
\begin{equation*}\label{Lic_v-eq}
\mathbb{L}_{\omega,v}(\psi)=\frac{1}{2}(\Delta_{\omega,v})^{2}(\psi)+\delta_{\omega,v}\left(\rho_{\omega,v}((d^{c}\psi)^{\sharp})\right),
\end{equation*}
where $\delta_{\omega,v}:=\frac{1}{v(\mu_\omega)}\delta_\omega v(\mu_\omega)$ is the formal adjoint of the exterior derivative $d$ on functions with respect to the weighted volume form $v(\mu_\omega)\omega^{[m]}$,  $\rho_{\omega,v}:=\rho_{\omega}-\frac{1}{2}dd^{c}\left(\log v(\mu_{\omega})\right)$ is the Ricci form of the weighted volume form $v(\mu_\omega)\omega^{[m]}$, and ${\sharp}=g_{\omega}^{-1}$ stands for the riemannian duality between $TM$ and $T^*M$ by using the K\"ahler metric $\omega$.
\end{lemma}

We now specialize to the case  when $(Y, \tilde \omega, \T_Y)$ is a semi-simple principal $(X, \omega,  \T_X)$-fibration over $B$, as in Section~\ref{s:geometric}. We then denote by $\Delta_{\tilde \omega}^Y$, $\bL_{\tilde \omega}^Y$ and $({\Ha}_{\tilde \omega}^{\tilde \rho})^Y$ the corresponding unweighted operators on $(Y, \tilde \omega)$, where  the K\"ahler form $\tilde \rho$  in the definition of ${\Ha}^{\tilde \rho}_{\tilde \omega}$  is bundle-compatible, i.e. given by \eqref{Y-Kahler} for a $\T_X$-invariant K\"ahler form $\rho$ on $X$. We further let $\Delta^{B_a}_{\omega_a}$ denote the Laplacian on $(B_a, \omega_a)$,  and $\Delta^B_x$ and $\bL^B_x$ respectively the Laplacian and Lichnerowicz operators on $B$ with respect to the K\"ahler metric $\omega_B(x):= \sum_{a=1}^k(\langle p_a, \m_{\omega}(x)\rangle + c_a)\omega_{a}$. We thus have the following result.
\begin{lemma}\label{l:operators-tilde} Let $\psi$ be a $\T_Y$-invariant smooth function on $Y$, seen as a $\T_X$-invariant function on $X \times B$ via \eqref{function-splitting}, and  $\tilde \omega$ a bundle-compatible $\T_Y$-invariant K\"ahler metric on $Y$ associated to a $\T_X$-invariant K\"ahler metric $\omega$ on $X$.  We then have 
\begin{equation*}\label{operators-tilde}
\begin{split}
\Delta_{\tilde \omega}^Y \psi &= \Delta^X_{\omega,p} \psi_b + \Delta^B_{x} \psi_x,\\ 
\bL^Y_{\tilde \omega} \psi  &= \bL^{X}_{\omega, p} \psi_b  + \bL^B_x \psi_x + \Delta^B_x\left(\Delta^X_{\omega,p}\psi_b) \right)_x +  \Delta^X_{\omega,v}\left(\Delta^B_{x} \psi_x \right)_{b} \\ 
                  &+ \sum_{a=1}^k Q_a(x)\Delta^{B_a}_{\omega_a} \psi_x ,  \\              
({\Ha}^{\tilde \rho}_{\tilde \omega, 1})^Y \psi  &= ({\Ha}^{\rho}_{\omega, p})^X \psi_b  + \sum_{a=1}^k P_a(x)\Delta^{B_a}_{\omega_a} \psi_x, 
                   \end{split}\end{equation*}
where $P_a(x), Q_a(x)$  are  smooth $\T$-invariant functions on $X$, and $\psi_x$ and $\psi_b$ are respectively the induced smooth functions on $B$ and $X$ via \eqref{function-splitting}.
\end{lemma}
\begin{proof} This first two equalities are established in \cite{HFKG4} (see the proof of Lemma 8) in the special case when $(X, \omega, \T_X)$ is a toric variety whereas the third identity is proved in \cite{Jubert} (also in the case when $(X, \T_X)$ is toric). These computations extend to the general setting with no substantial additional difficulty (by using Lemma~\ref{l:Lic_v} above for the second identity), but we include them below for the sake of self-containedness.   

In the notation of Sect.~\ref{s:geometric}
\begin{equation} \label{laplacian}
\begin{split}
\Delta_{\tilde{\omega}}^{Y}(\psi)=&-\frac{d_Yd^{c}_Y\psi\wedge\tilde{\omega}^{[n+m-1]}}{\tilde{\omega}^{[n+m]}}\quad \text{(on }Y\text{)}\\
=&-\frac{d_Yd^{c}_Y\psi\wedge\tilde{\omega}^{[n+m-1]}\wedge\theta^{\wedge \dtor}}{\tilde{\omega}^{[n+m]}\wedge\theta^{\wedge \dtor}} \quad \text{(on }Z=X\times \bN\text{)},
\end{split}
\end{equation}
where   $\theta^{\wedge \dtor}:=\bigwedge_{i=1}^{\dtor}\theta_i$ with respect to any lattice basis $(\xi_i)_i$ of $\Lambda\subset \tor$.
Viewing $d^{c}_{X\times B}\psi$ as a 1-form on $Z$, it admits the following decomposition with respect to \eqref{split-T(X x N)} 
\begin{equation}\label{dc-XxB}
d^{c}_{X\times B}\psi=(d^{c}_{X\times B}\psi)_{\Hs}+\sum_{i=1}^{\dtor} (d^{c}_{X\times B}\psi)(\xi_i^{\bN}-\xi_i^X)\theta_i=d^{c}_{Y}\psi-\langle d^{c}_{X}\psi,\theta\rangle.
\end{equation}
We thus compute on $Z$:
\begin{align}
\begin{split}\label{(ddc)Y-psi}
(d_Yd^{c}_Y\psi)_{(x,b)}=&d_{Z}\left(d^{c}_X\psi+\sum_{j=1}^{\dtor}d^{c}_X\psi(\xi^{X}_j)\theta_j+d^{c}_B\psi\right)\\
=&d_{Z}d^{c}_X\psi+\sum_{j=1}^{\dtor}d_{Z}(d^{c}_X\psi(\xi^{X}_j))\theta_j\\
&+\sum_{j=1}^{\dtor}d^{c}_X\psi_b(\xi^{X}_j) \left(\sum_{a=1}^{k}\xi^{j}(p_a)\pi_B^{*}\omega_a\right)+d_Zd^{c}_B\psi\\
=&d_Xd^{c}_X\psi_b+d_Bd^{c}_B\psi_x+\sum_{j=1}^{\dtor}d_{Z}(d^{c}_X\psi(\xi^{X}_j))\wedge\theta_j+\sum_{a=1}^{k}d^{c}_X\psi(p^{X}_a)\pi_B^{*}\omega_a\\
&+d_Bd_X^{c}\psi+d_Xd_B^{c}\psi,
\end{split}
\end{align}
where for getting  the third equality we used \eqref{d-theta},  as well as the identities  $d_{\bN}d^{c}_X\psi=d_Bd_X^{c}\psi$ and $d_{\bN} d^{c}_B\psi=d_B d_B^{c}\psi$ (which follow from the identification \eqref{function-splitting}).
Using  \eqref{Z-measure} and \eqref{omega^(m+n-1)},  we derive from \eqref{laplacian}  and \eqref{(ddc)Y-psi}
\[
\begin{split}
\Delta_{\tilde{\omega}}^{Y}(\psi)(x,b)=&(\Delta^X_{\omega} \psi_b)(x) +(\Delta^B_{\omega_B(x)} \psi_x)(b)-\sum_{a=1}^{k}\frac{n_a}{(\langle \m_\omega,p_a\rangle + c_a)}(d^{c}_X\psi_b)(p^{X}_a),
\end{split}
\]
where, we recall, for a fixed $x\in X$, we have set $\omega_B(x):=\sum_{a=1}^k(\langle p_a, \m_{\omega}\rangle + c_a)\omega_a$, and $p_a^X$ denotes the vector field field on $X$ corresponding to $p_a\in \tor$.  The first equality  in the Lemma follows from the identity \eqref{p-laplacian}, taking in mind that for any smooth function on $u$ on $\Pol$ and any $\T$-invariant smooth function $\phi$ on $X$,  $g_{\omega}(d (u(\m_{\omega})), d\phi)=\sum_{i=1}^{\dtor}u_{,i}(\m_{\omega})d^c\phi(\xi_i)$.

\bigskip
Now, we establish the expression of the corresponding Lichnerowicz operators. Recall that (see e. g.~\cite{gauduchon-book})  
\begin{equation}\label{Lic-tilde-0}
\mathbb{L}_{\tilde{\omega}}^{Y}\psi :=\frac{1}{2}(\Delta_{\tilde{\omega}}^{Y})^{2}(\psi)+\delta_{\tilde{\omega}}(\rho_{\tilde{\omega}}(d^{c}_Y\psi)).
\end{equation}
Using  the decomposition of $\Delta_{\tilde \omega}^Y$ we have just established, we have
\begin{equation}\label{Delta^2}
(\Delta_{\tilde{\omega}}^{Y})^{2}(\psi)= (\Delta^X_{\omega,p} )^{2}(\psi_b)+ (\Delta_{x}^{B})^{2}(\psi_x)+\Delta^X_{\omega,p} (\Delta_{x}^{B}(\psi_x))+ \Delta_{x}^{B}\left(\Delta^X_{\omega,p}(\psi_b)\right).
\end{equation}
It remains to compute the Ricci term in \eqref{Lic-tilde-0}.
From \eqref{ddc-tkappa1}, we have
\begin{align}
\begin{split}\label{tilde-Ricci}
\rho_{\tilde{\omega}}=&\rho_{\omega,p}+\pi^{*}_B\rho_{\omega_B}+\frac{1}{2}\sum_{a=1}^{k} \Delta^{X}_{\omega,p}(\langle \mu_{\omega},p^{X}_{a}\rangle)\pi^{*}_B\omega_a\\
&+\sum_{j=1}^{\dtor} d_X\left(d^{c}_X\left(\kappa-\frac{1}{2}\log p(\mu_{\omega})\right)(\xi^{X}_j)\right)\wedge\theta_j.
\end{split}
\end{align}
where $\rho_{\omega,p} := \rho_{\omega} -\frac{1}{2}d_X d^c_X \log p(\m_{\omega})$ is the Ricci form of the weighted volume form $p(\mu_\omega)\omega^{[m]}$.
Using integration by parts, for any  $\T_Y$-invariant smooth test function $\phi$ on $Y$, seen as a $\T_X$ and $\T_{\bN}$-invariant  function on $Z=X \times \bN$ via \eqref{function-splitting}, we have
\begin{align}
\begin{split}
&\int_Z \phi \delta_{\tilde{\omega}}(\rho_{\tilde{\omega}}(d^{c}_Y\psi)) \tilde{\omega}^{[n+m]}\wedge\theta^{\wedge r} =-\int_Z \rho_{\tilde{\omega}}(d_Y\phi,d^{c}_Y\psi) \tilde{\omega}^{[n+m]}\wedge \theta^{\wedge r}\\
&=\int_Z \rho_{\tilde{\omega}}\wedge d_Y\phi\wedge d^{c}_Y\psi\wedge\tilde{\omega}^{[n+m-2]}\wedge\theta^{\wedge r}-\frac{1}{2}\int_Z\Scal(\tilde{\omega}) \tilde{g}_{\tilde \omega}(d_Y\phi,d_Y\psi)\tilde{\omega}^{[n+m]}\wedge \theta^{\wedge r}\\
=&\int_Z\rho_{\tilde{\omega}}\wedge d_Y\phi\wedge d^{c}_Y\psi\wedge\tilde{\omega}^{[n+m-2]}\wedge\theta^{\wedge \dtor}-\frac{1}{2}\int_Z\left(\frac{\Scal_p(\omega)}{p(\mu_\omega)}+q(\mu_\omega)\right) d_Y\phi\wedge d^{c}_Y\psi\wedge \tilde{\omega}^{[n+m-1]}\wedge\theta^{\wedge \dtor}.
\end{split}
\end{align}
From the above formula, using  \eqref{omega^(m+n-1)},  \eqref{dc-XxB}  and \eqref{tilde-Ricci},  we compute (after some straightforward but long algebraic manipulations and integration by parts over $X$ and $B$)
\begin{equation}\label{del-ric-1}
\begin{split}
&\delta^Y_{\tilde{\omega}}(\rho_{\tilde{\omega}}(d^{c}_Y\psi))= \delta^X_{\omega, p}\left(\rho_{\omega,p}(d^{c}_X\psi)\right)+\delta^B_{\omega_B(x)}\left(\rho_{\omega_B}(d^{c}_B\psi)\right)\\
&+\frac{1}{2}\sum_{a=1}^{k}\frac{q(\mu_\omega)}{ (\langle \mu_{\omega},p_{a}\rangle+c_a)}\Delta^B_{\omega_a}(\psi) +\frac{1}{2}\sum_{a=1}^{k}\frac{(n_a-1)}{(\langle \mu_{\omega},p_{a}\rangle+c_a)^{2}} \Delta^{X}_{\omega,p}(\langle \mu_{\omega},p_{a}\rangle)\Delta^B_{\omega_a}(\psi_x)\\
&+\sum_{a,b=1}^{k}\frac{n_b}{(\langle \mu_{\omega},p_{a}\rangle+c_a)(\langle \mu_{\omega},p_{b}\rangle+c_b)}\Delta^{X}_{\omega,p}(\langle \mu_{\omega},p_{b}\rangle)\Delta^B_{\omega_a}(\psi_x).
\end{split}
\end{equation}
Combining \eqref{Lic-tilde-0}, \eqref{Delta^2} and \eqref{del-ric-1}  yields the desired expression.

\bigskip
The expression for  $({\Ha}^{\tilde \rho}_{\tilde \omega, 1})^Y(\psi)$ is obtained by similar arguments, using that
\[
\begin{split}
({\Ha}^{\tilde \rho}_{\tilde \omega, 1})^Y(\psi) =& \langle \tilde{\rho},d_Yd^{c}_Y\psi\rangle_{\tilde{\omega}}+\langle d_Y\tr_{\tilde{\omega}}(\tilde{\rho}),d_Y\psi\rangle_{\tilde{\omega}}\\
=& -\tr_{\tilde{\omega}}(\tilde{\rho}) \Delta^Y_{\tilde{\omega}}(\psi)-\frac{\tilde{\rho}\wedge d_Yd^{c}_Y\psi\wedge \tilde{\omega}^{[n+m-2]}}{\tilde{\omega}^{[n+m]}}+\frac{d_Y\tr_{\tilde{\omega}}(\tilde{\rho})\wedge d^{c}_Y\psi\wedge \tilde{\omega}^{[n+m-1]} }{\tilde{\omega}^{[n+m]}}.
\end{split}
\]
\end{proof} 
\end{proof}

\section{Weighted Futaki invariants}\label{s:Futaki} On  a smooth Fano manifold $(X,  \T)$ as in the setting and notation of Section~\ref{s:v-KRS},  we further relate the weighted Futaki obstruction ${\rm Fut}_{v, w} =0$ (see \eqref{Futaki}) with weights $v(\x), w(\x)$  as in  Proposition~\ref{tilde v} with the Futaki-type obstructions studied by Tian--Zhu~\cite{TZ} in the case of K\"ahler--Ricci solitons (i.e. when  $v= e^{\langle \xi, \x\rangle}$):  
\begin{lemma}\label{l:Futaki-TZ} Let $(X, \T)$ be a smooth Fano manifold $(X, \T)$ with canonically normalized momentum polytope $\Pol$,  and $v>0, w$ smooth functions on  $\Pol$  as in Proposition~\ref{tilde v}. Then, for any $\T$-invariant K\"ahler metric $\omega \in 2\pi c_1(X)$ with momentum map  $\m_{\omega}$ and $\T$-invariant Ricci potential $h$ (i.e. $\rho_{\omega} - \omega = \frac{1}{2} dd^c h$),  the weighted Futaki invariant ${\rm Fut}_{v, w}$ introduced in \eqref{Futaki} satisfies
\[ {\rm Fut}_{v, w} (\ell_{\zeta}) = \int_X \Big(\cL_{J\zeta}\big(\log v(\m_{\omega}) -h\big)\Big) v(\m_{\omega})\omega^{[m]}=-2\int_X \langle \zeta, \m_{\omega}\rangle v(\m_{\omega}) \omega^{[m]}, \qquad \ell_{\zeta} = \langle\zeta, \x \rangle + a, \, \, \zeta \in \tor.\]
\end{lemma}
\begin{proof} We have
\[
\begin{split}
 & \int_X\Big(\cL_{J\zeta}\big(\log v(\m_{\omega}) -h\big)\Big) v(\m_{\omega})\omega^{[m]} = \int_X g_{\omega}\Big(d\ell_{\zeta}, d \log (v(\m_{\omega}) -h)\Big) v(\m_{\omega}) \omega^{[m]} \\
 &= \int_X \ell_{\zeta} \left(\Delta_{\omega,v}\big( \log v(\m_{\omega}) -h\big)\right) v(\m_{\omega}) \omega^{[m]} = \int_X \ell_{\zeta}\left(\Scal_v(\omega) - w(\m_{\omega})\right) \omega^{[m]} = {\rm Fut}_{v, w}(\ell_{\zeta}), 
\end{split} \]
where for the last equality we have used \eqref{Scal-v-tilde v}. The second equality in the Lemma follows from the first, the second relation in \eqref{basic-relations} and integration by parts.
\end{proof}


\end{document}